\pdfoutput=1
\documentclass[11pt,a4paper,reqno]{amsart}
\usepackage[includehead,includefoot,margin=25mm]{geometry}
\usepackage{times}
\usepackage{amsfonts} %
\usepackage{amsmath} %
\usepackage{amssymb} %
\usepackage{amsthm} %
\usepackage{enumerate} %
\usepackage{bbm} %
\usepackage{graphicx} %
\usepackage{nicefrac} %
\usepackage{color} %
\usepackage{hyperref}
\usepackage{tikz-cd}
\usepackage{mathrsfs}

\allowdisplaybreaks

\usepackage{graphicx} % Required for inserting images
\usepackage{tikz}

\usetikzlibrary{circuits.logic.US,circuits.logic.IEC,fit}

\hypersetup{
	colorlinks=true,       % false: boxed links; true: colored links
	linkcolor=blue,          % color of internal links
	citecolor=blue,        % color of links to bibliography
	filecolor=blue,      % color of file links
	urlcolor=blue           % color of external links
}

\definecolor{zzqqtt}{rgb}{0.6,0.,0.2}
\definecolor{qqttcc}{rgb}{0.,0.2,0.8}
\definecolor{zzttqq}{rgb}{0.6,0.2,0.}
\definecolor{wwccff}{rgb}{0.4,0.8,1.}

\usepackage[all]{xy}
\usepackage{calrsfs} 

\usepackage{imakeidx}
\makeindex

\theoremstyle{plain}
\newtheorem{thm}{Theorem}[section]
\theoremstyle{plain}
\newtheorem{conj}[thm]{Conjecture}
\newtheorem{lem}[thm]{Lemma}
\theoremstyle{definition}
\newtheorem{defi}[thm]{Definition}
\theoremstyle{definition}
\newtheorem{example}[thm]{Example}
\theoremstyle{plain}
\newtheorem{prop}[thm]{Proposition}
\theoremstyle{plain}
\newtheorem{cor}[thm]{Corollary}
\theoremstyle{remark}

\theoremstyle{remark}
\newtheorem{remark}[thm]{Remark}

\newcommand{\F}{\mathbb{F}}

\usepackage{scalerel,stackengine}
\stackMath
\newcommand\bbhat[1]{%
	\savestack{\tmpbox}{\stretchto{%
			\scaleto{%
				\scalerel*[\widthof{\ensuremath{#1}}]{\kern-.6pt\bigwedge\kern-.6pt}%
				{\rule[-\textheight/2]{1ex}{\textheight}}%WIDTH-LIMITED BIG WEDGE
			}{\textheight}% 
		}{0.5ex}}%
	\stackon[1pt]{#1}{\tmpbox}%
}
\newcommand{\Zet}{\operatorname{Z}}
\newcommand{\A}{\mathbb{A}}
\newcommand{\R}{\mathbf{R}}
\newcommand{\bS}{\mathbf{S}}

\newcommand{\PP}[0]{\ensuremath{\mathbb{P}}}
\newcommand{\CC}[0]{\ensuremath{\mathbb{C}}}
\newcommand{\ZZ}[0]{\ensuremath{\mathbb{Z}}}

\newcommand{\QQ}[0]{\ensuremath{\mathbb{Q}}}
\newcommand{\RR}[0]{\ensuremath{\mathbb{R}}}

\newcommand{\CO}[0]{\ensuremath{\mathcal{O}}}

\newcommand{\Pic}[0]{\ensuremath{\operatorname{Pic}}}

\newcommand{\FF}{\mathbb{F}}

\newcommand{\rk}{\operatorname{rk}}

\newcommand{\Div}{\operatorname{Div}}

\newcommand{\Val}{\operatorname{Val}}
\newcommand{\TT}{\mathbb{T}}

\usepackage[normalem]{ulem}
\usepackage{dynkin-diagrams}
\usepackage{multirow}

\makeatletter
\let\@wraptoccontribs\wraptoccontribs
\makeatother

\begin{document}
	\title[]{Counting rational points on Hirzebruch--Kleinschmidt varieties over global function fields}
	
 \author{Sebastián Herrero}
\address{Universidad de Santiago de Chile, Dept.~de Matem\'atica y Ciencia de la Computaci\'on, Av.~Libertador Bernardo O'Higgins 3363, Santiago, Chile, and ETH, Mathematics Dept., CH-8092, Z\"urich, Switzerland}
  \email{sebastian.herrero.m@gmail.com}
   
  \author{Tobías Martínez}
\address{Departamento de Matem\'aticas, Universidad T\'ecnica
  Fe\-de\-ri\-co San\-ta Ma\-r\'\i a, Av.~Espa\~na 1680, Valpara\'\i
  so, Chile}
\email{tobias.martinez@usm.cl}

\author{Pedro Montero}
\address{Departamento de Matem\'aticas, Universidad T\'ecnica
  Fe\-de\-ri\-co San\-ta Ma\-r\'\i a, Av.~Espa\~na 1680, Valpara\'\i so, Chile}  \email{pedro.montero@usm.cl}

\begin{abstract}
	Inspired by Bourqui's work on anticanonical height zeta functions on Hirzebruch surfaces, we study height zeta functions of split toric varieties with Picard rank 2 over global function fields, with respect to height functions associated with big metrized line bundles. We show that these varieties can be naturally decomposed into a finite disjoint union of subvarieties, where precise analytic properties of the corresponding height zeta functions can be given. As application, we obtain asymptotic formulas for the  number of rational points of large height on each subvariety, with explicit leading constants and controlled error terms.
	\end{abstract}

 
	\maketitle

% 14G05: Rational points
% 14G40: Arithmetic varieties and schemes; Arakelov theory; heights
% 11G50: Heights 
% 14M25: Toric varieties, Newton polyhedra, Okounkov bodies
% 11R59(2020–now)Zeta functions and L-functions of function fields
% 14G25(1973–now)Global ground fields in algebraic geometry

%Bourqui non-split
%11G35(1980–now) Varieties over global fields
% 11G50: Heights 
% 14M25: Toric varieties, Newton polyhedra, Okounkov bodies
% 11M41(1980–now)Other Dirichlet series and zeta functions For local and global ground fields

%In Batirev-Tschinkel (Manin's conjecture for toric varieties):
%14G10: Zeta functions and related questions in algebraic geometry (e.g., Birch-Swinnerton-Dyer conjecture)
% 14G05: Rational points
% 11M41: Other Dirichlet series and zeta functions For local and global ground fields

	%\printindex
	
	\setcounter{tocdepth}{1}
	\tableofcontents
 
	\section{Introduction}
 In the papers~\cite{Fran/Man/Tsch89} and~\cite{Bat/Man90}, a program was initiated to understand the asymptotic behavior of the number of rational points of bounded height on Fano algebraic varieties over global fields.  In the case of number fields, a refined version of the original expectations is as follows (see, e.g., \cite[{\it Formule empirique} 5.1]{Peyre2003PHBTA} or \cite[Conjecture~6.3.1.5]{Arzhantsev_Derenthal_Hausen_Laface_2014CoxRings} for precise definitions).

\begin{conj}[Manin--Peyre]\label{conj MP}
Let $X$ be an almost Fano variety\footnote{Following~\cite[{\it D\'efinition}~3.1]{Peyre2003PHBTA}, an almost Fano variety is a smooth, projective, geometrically integral variety~$X$ defined over a field~$K$, with~$H^1(X, \CO_X) = H^2(X, \CO_X) = 0$, torsion-free geometric Picard group~$\Pic(X_{\overline{K}})$ and $-K_X$ big.} over
a number field $K$, with dense set of rational points $X(K)$, finitely generated effective cone~
$\Lambda_{\textup{eff}}(X_{\overline{K}})$ and trivial Brauer group $\operatorname{Br}(X_{\overline{K}})$. Let $H=H_{-K_X}$ be the anticanonical height function, and assume that there is an open subset $U$ of $X$ that is the complement
of the weakly accumulating subvarieties on $X$ with respect to $H$. Then, there is a constant $C>0$  such that
\begin{equation}\label{eq MP Conjecture}
N(U, H,B):=\#\{ P\in U(K):H(P)\leq B \} \sim C B(\log B)^{\rk \operatorname{Pic}(X)-1} \quad \text{as } B\rightarrow \infty.    
\end{equation}
Moreover, the leading constant is of the form
$$C=\alpha(X)\beta(X)\tau_H(X),$$
where
\begin{align*}
       \alpha(X)&:=\frac{1}{\left(\rk \Pic(X)-1\right)!}\int_{\Lambda_{\textup{eff}}(X)^\vee} e^{-\left \langle -K_X,\mathbf{y}\right \rangle }\textup{d}\mathbf{y},  \\
       \beta(X)&:=\#H^1(\textup{Gal}(\overline{K}/K), \Pic(X_{\overline{K}})), 
\end{align*}
and~$\tau_H(X)$ is the Tamagawa number of $X$ with respect to $H$ defined by Peyre in~\cite[Section~4]{Peyre2003PHBTA}.
\end{conj}

\begin{remark}
In general, the union of all weakly accumulating subvarieties of~$X$ is not necessarily contained within a proper closed subset, and more general versions of the Manin--Peyre conjecture suggest that there exists a thin\footnote{Following \cite{Serre1989}, a subset $Z \subset X(K)$ is called \emph{thin} if it can be expressed as a finite union $Z=\bigcup_i \pi_i(Y_i(K))$ where each $\pi_i:Y_i \to X$ is a morphism that is generically finite onto its image and has no rational sections.
} subset~$Z\subset X(K)$ such that the asymptotic formula~\eqref{eq MP Conjecture} holds for the number of rational points in~$X(K)\setminus Z$ of height~$\leq B$ (see, e.g., \cite{LT17}, \cite{LR19} and references therein). However, for the purpose of this paper, the above formulation is sufficient.
\end{remark}

There is also a conjecture, originating in the work of Batyrev and Manin~\cite{Bat/Man90}, concerning the asymptotic growth of the number
$$N(U,H_L,B):=\#\{P\in U(K):H_L(P)\leq B\},$$
when considering height functions~$H_L$ associated with big metrized line bundles~$L$ on a variety~$X$ as above, and for appropriate open subsets~$U\subseteq X$. More precisely, if we write $\tau \prec \sigma$ whenever $\tau$ is a face of a cone $\sigma$, then one defines  the following classical numerical invariants for a line bundle class~$L$ on $X$:
\begin{equation*}
    \begin{split}
       a(L)&:=\inf\{a\in \RR: aL+K_X\in \Lambda_{\textup{eff}}(X)\}, \\
       b(L)&:=\max\{\textup{codim}(\tau): a(L)L+K_X\in \tau\prec \Lambda_{\textup{eff}}(X)\},
    \end{split}
\end{equation*}
which measure the position of $L$ inside the effective cone~$\Lambda_{\textup{eff}}(X)$. With this notation, the more general version of the above conjecture states that there exists a constant~$C>0$ such that
\begin{equation}\label{eq Batirev-Manin conjecture}
  N(U,H_L,B)\sim CB^{a(L)}(\log B)^{b(L)-1} \quad \text{as }B\to \infty.  
\end{equation}

Note that~$a(-K_X)=1$ and~$b(-K_X)=\rk \operatorname{Pic}(X)$, hence~\eqref{eq Batirev-Manin conjecture} is indeed a generalization of~\eqref{eq MP Conjecture}.

The above conjectures have been proven by various authors, either in specific examples or in certain families of varieties (see for instance~\cite{Peyre2003PHBTA}, \cite[Section~3]{Tschinkel03}, \cite{Bro07} and~\cite[Section~4]{Tsch008} for accounts of such results).

Typically, asymptotic formulas of the form~\eqref{eq MP Conjecture} and~\eqref{eq Batirev-Manin conjecture} are deduced, via a Tauberian theorem, from the analytic properties of the associated height zeta functions. More precisely, considering~$U\subseteq X$ and a height function~$H_L$ associated to a big metrized line bundle~$L$, the corresponding height zeta function is defined as
$$\zeta_{U,L}(s):=\sum_{P\in U(K)}H_L(P)^{-s}. $$ % \quad \text{for }s\in \CC \text{ with }\Re(s)\gg 0.$$
If~$\zeta_{U,L}(s)$ converges absolutely in the half-plane~$\Re(s)>a>0$,  has  analytic continuation to~$\Re(s)\geq a, s\neq a$ and  has a pole of order~$b\geq 1$ at~$s=a$, then one obtains
\begin{equation}\label{eq asymptotic formula from Tauberian thm}
 N(U,H_L,B)\sim CB^{a}(\log B)^{b-1} \quad \text{as }B\to \infty,   
\end{equation}
with
$$C:=\frac{1}{(b-1)!a}\lim_{s\to a}(s-a)^{b}\zeta_{U,L}(s)$$
(see, e.g., \cite[Th\'eor\`eme~III]{DelangeTauberianThm}).

In the case where~$K$ is a global function field of positive characteristic, the relevant height functions typically  have values contained in~$q^{\ZZ}$, where~$q$ the cardinality of the constant subfield~$\F_q\subset K$. This implies that the associated height zeta functions are invariant under~$s\mapsto s+\frac{2\pi i}{\log(q)}$. In particular, having a pole at a point~$s=a>0$ implies the existence of infinitely many poles on the line~$\Re(s)=a$, and the strategy described above breaks down. Moreover, in this setting, it is hopeless to expect an asymptotic formula of the form~\eqref{eq asymptotic formula from Tauberian thm} since~$N(U,H_L,q^n)=N(U,H_L,q^{n+\frac{1}{2}})$ would imply~$\sqrt{q}=1$ (as explained in~\cite[Section~1.1]{Bourqui2011VarietiesNonSplit}).

For these reasons, for varieties defined over global fields of positive characteristic, the analogue of Conjecture~\ref{conj MP} is sometimes presented in terms of the analytic properties of the respective height zeta function. This is the approach taken by Peyre in~\cite{Pey12PHB} in the case of flag varieties, and by Bourqui in~\cite{BOURQUI2002}, \cite{Bourqui2003VarietiesSplit} and~\cite{Bourqui2011VarietiesNonSplit,Bou16} in the case of Hirzebruch surfaces, split toric varieties, and general toric varieties, respectively. As a variant of Conjecture~\ref{conj MP}, it is then expected that the anticanonical height zeta function~$\zeta_{U,-K_X}(s)$, for appropriate open subsets~$U\subseteq X$, converges absolutely for~$\Re(s)>1$, it admits meromorphic continuation to~$\Re(s)\geq 1$ and  has a pole at~$s=1$ of order~$\rk \operatorname{Pic}(X)$ satisfying
\begin{equation*}\label{eq MP for global fields positive characteristic}
 \lim_{s\to 1}(s-1)^{\rk \operatorname{Pic}(X)}\zeta_{U,-K_X}(s)=\alpha^\ast(X)\beta(X)\tau_H(X),   
\end{equation*}
where
\begin{align*}
       \alpha^\ast(X)&:=\int_{\Lambda_{\textup{eff}}(X)^\vee} e^{-\left \langle -K_X,\mathbf{y}\right \rangle }\textup{d}\mathbf{y},  \\
       \beta(X)&:=\#H^1(\textup{Gal}(K^{\mathrm{sep}}/K), \Pic(X_{K^{\mathrm{sep}}})), 
\end{align*}
and~$\tau_H(X)$ is the Tamagawa number of~$X$ with respect to the anticanonical height function~$H=H_{-K_X}$ defined by Peyre in~\cite[Section~2]{Pey12PHB} in the case of global fields of positive characteristic. Similarly, as a variant of~\eqref{eq Batirev-Manin conjecture}, one expects~$\zeta_{U,L}(s)$ to converge absolutely for~$\Re(s)> a(L)$ and have meromorphic continuation to~$\Re(s)\geq  a(L)$ with a pole at~$s=a(L)$ of order~$b(L)$.

Remaining in the context of varieties defined over global function fields, there is yet another approach that focuses instead on the asymptotic behaviour of the number
$$\tilde{N}(U,H_L,B):=\#\{P\in U(K):H_L(P)=B\}$$
of rational points of large height. Asymptotic formulas for these numbers can be deduced, via another Tauberian theorem, from finer analytic properties of the corresponding height zeta functions. This is the approach that we take in this paper.

\begin{remark}
    Given a projective variety~$V$ and a smooth projective curve~$\mathcal{C}$, both defined over a finite field~$\FF_q$, the study of the number of points in~$V(\FF_q(\mathcal{C}))$ of large height  is equivalent to the problem of counting morphisms~$\mathcal{C}\to V$ of large degree. This geometric perspective has been pursued by several authors (see, e.g., \cite{Bou2011}, \cite{Bou16}, \cite{LT19}, \cite{Tan21}, \cite{BLRT22} and references therein).
\end{remark}

In the present paper we focus on  smooth projective split toric varieties with Picard rank 2 over a global field~$K$ of positive characteristic, and consider height zeta functions associated to  big metrized line bundles. Our results show that these varieties can be naturally decomposed into a finite disjoint union of subvarieties for which finer analytic properties of the corresponding height zeta functions can be given. This allows us to obtain %, via an ad hoc Tauberian theorem, 
asymptotic formulas for the number of rational points of large height on each component of our disjoint union decomposition, going beyond the scope of the classical expectations. Moreover, our asymptotic formulas have controlled error terms and explicit leading coefficients. We achieve these results by using concrete defining equations for our varieties and by performing explicit computations on the height zeta functions.

This work was motivated by the results of Bourqui in~\cite{BOURQUI2002}, \cite{Bourqui2003VarietiesSplit} and~\cite{Bourqui2011VarietiesNonSplit}, where the classical expectation for the anticanonical height zeta function 
is verified for toric varieties over global fields of positive characteristic, with~$U$ the dense torus orbit. Bourqui's work was inspired by Batyrev and Tschinkel's proof of Conjecture~\ref{conj MP} for toric varieties over number fields~\cite{BT98}.

In order to present our results more precisely, let us recall a geometric result due to Kleinschmidt~\cite{Kleinschmidt88} stating that all smooth projective toric varieties of Picard rank 2 are (up to isomorphism) of the form 
\[
X=\PP(\CO_{\PP^{t-1}}\oplus \CO_{\PP ^{t-1}}(a_1)\oplus\cdots\oplus\CO_{\PP ^{t-1}}(a_r)),
\]
where $r \geq 1$, $t \geq 2$, and $0 \leq a_1 \leq \cdots \leq a_r$ are integers. Moreover, if we define~$I_k:=\{1,\ldots,k\} \subset \ZZ$, then~$X$ is isomorphic to the subvariety~$X_d(a_1,\ldots,a_r)\subset \PP^{rt}\times \PP^{t-1}$ given in homogeneous coordinates $([x_0: (x_{ij})_{i\in I_t,j\in I_r}],[y_1:\ldots:y_t])$ by the equations
\begin{equation*}
x_{mj}y^{a_j}_n=x_{nj}y^{a_j}_m ,\, \text{ for all } j\in I_{r} \text{ and all } m,n\in I_t \text{ with }m\neq n.   
\end{equation*}
We refer to these varieties as \emph{Hirzebruch--Kleinschmidt varieties}. As usual, in order to avoid convergence issues, we find it necessary to first restrict our attention to rational points in a dense open subset. To this purpose, we define the \emph{good open subset}~$U_d(a_1, \ldots, a_r)\subset X_d(a_1,\ldots,a_r)$ as the complement of the closed subvariety define by the equation~$x_{tr}=0$.

Our main results describe the analytic properties of the height zeta function~$\zeta_{U,L}(s)$ for the good open subset~$U:=U_d(a_1, \ldots, a_r)\subset X_d(a_1,\ldots,a_r)$, for every big line bundle class~$L\in \Pic(X)$ that we equip with a ``standard metrization''. These general results are given in Section~\ref{subsection general height zeta functions}. For simplicity, we present here the statement for~$L=-K_X$.

We let $\mathcal{C}$ be a projective, smooth, geometrically irreducible curve of genus $g$ defined over the finite field $\mathbb{F}_q$ of~$q$ elements, and let~$K=\mathbb{F}_q(\mathcal{C})$ denote the field of rational functions of $\mathcal{C}$.

\begin{thm}\label{thm:MainThmAnticanonicalCaseoverFF}
    Let $X=X_d(a_1,\ldots,a_r)$ be a Hirzebruch--Kleinschmidt variety  of dimension $d=r+t-1$ over the global function field~$K=\F_q(\mathcal{C})$, let~$H=H_{-K_X}$ denote the anticanonical height function on~$X(K)$, and let~$U=U_d(a_1,\ldots,a_r)$ be the good open subset of~$X$. Moreover, put
    $$\eta_X:=\mathrm{g.c.d.}(r+1,t-|\mathbf{a}|).$$
    Then, the anticanonical height zeta function~$\zeta_{U,-K_X}(s)$ converges absolutely for~$\Re(s)>1$, it is a rational function in~$q^{-\eta_X s}$ and has a pole of order 2 at~$s=1$ with
     \begin{equation}\label{eq constant main thm anticanonical}
       \lim_{s\to 1}(s-1)^2\zeta_{U,-K_X}(s)= \frac{q^{(d+2)(1-g)} h_K^2}{\zeta_K(t)\zeta_K(r+1)((r+1)a_r-|\mathbf{a}|+t)(r+1)(q-1)^2\log(q)^2},  
     \end{equation}
    where~$|\mathbf{a}|:=\sum_{i=1}^r a_i$, $h_K$ is the class number of~$K$, $g$ is the genus of~$\mathcal{C}$ and~$\zeta_K$ is the zeta function of~$K$. Moreover, it is holomorphic for
    $$\Re(s)>1-\min\left\{\frac{1}{r+1},\frac{1}{(r+1)a_r-|\mathbf{a}|+t}\right\},\Re(s)\neq 1,$$
    and on~$\Re(s)=1, s\not \in \left\{1+\frac{2\pi i m}{\eta_X \log(q)}:m\in \ZZ\right\}$ it has at most simple poles.
\end{thm}

\begin{remark}
\begin{enumerate}
    \item The fact that~$\zeta_{U,-K_X}(s)$ is a function of~$q^{-\eta_Xs}$ follows from the elementary observation that the anticanonical height function~$H_{-K_X}$ assumes values in~$q^{\eta_X\ZZ}$ (see Proposition~\ref{prop effective cone}(2) and equation~\eqref{eq height function HL} in Section~\ref{sec height zeta HK}).
    \item In Section~\ref{section Peyre alpha} we show that
$$\alpha^\ast(X)=\frac{1}{(r+1) ((r+1)a_r-|\mathbf{a}|+t)}.$$
Also, since $X$ is a split toric variety, it is well known that~$\beta(X)=1$ (see, e.g., \cite[Remark~1.7 and Corollary~1.18]{BT98}). Hence, Theorem~\ref{thm:MainThmAnticanonicalCaseoverFF} and~\cite[\emph{Th\'eor\`eme}~1.1]{Bourqui2003VarietiesSplit} imply that the Tamagawa number of~$X$ with respect to the anticanonical height function~$H=H_{-K_X}$ is
$$\tau_H(X)=\frac{q^{(d+2)(1-g)} h_K^2}{\zeta_K(t)\zeta_K(r+1)(q-1)^2\log(q)^2}.$$
\item In the case of Hirzebruch surfaces, i.e.~$r=1$ and~$t=2$, Theorem~\ref{thm:MainThmAnticanonicalCaseoverFF} is in accordance with~\cite[Proposition~3.1]{BOURQUI2002} (see also p.~355 in loc.~cit.). % where the case of the anticanonical height zeta function over the good open subset~$U_2(a)$ was studied. 
\end{enumerate}
\end{remark}

As mentioned before, we can deduce from the analytic properties of~$\zeta_{U,-K_X}(s)$ given in Theorem~\ref{thm:MainThmAnticanonicalCaseoverFF}, good estimates for the number~$\tilde{N}(U,H_{-K_X},q^M)=\#\{P\in K:H_{-K_X}(P)=q^M\}$ of rational points in~$U$ of anticanonical height~$q^M$ for large~$M$. Since this number is trivially zero when~$M$ is not divisible by~$\eta_X$, we can assume~$M\in \eta_X\ZZ$.

\begin{cor}\label{cor:MainThmAnticanonicalCaseoverFF}
Under the hypothesis of Theorem~\ref{thm:MainThmAnticanonicalCaseoverFF}, for every~$\delta>1-\min\left\{\frac{1}{r+1},\frac{1}{(r+1)a_r-|\mathbf{a}|+t}\right\}$ we have the estimate
 $$\tilde{N}(U,H_{-K_X},q^M)=q^{M}(C_1M+C_0(M))+O\left(q^{\delta M}\right) \quad \text{as }M\to \infty,M\in \eta_{X}\ZZ,$$
where
     $$C_{1}=\frac{q^{(d+2)(1-g)} \eta_X h_K^2}{\zeta_K(t)\zeta_K(r+1)((r+1)a_r-|\mathbf{a}|+t)(r+1)(q-1)^2},$$
     and~$C_0(M)\in \RR$ depends only on~$\frac{M}{\eta_X}$ mod~$\frac{(r+1)((r+1)a_r-|\mathbf{a}|+t)}{\eta_X^2}$.
\end{cor}

We now explain how to study further the number of rational points of large height in the complement of the good open subset of~$X_d(a_1,\ldots,a_r)$. Assume~$a_r>0$ for simplicity. Then, there is a natural decomposition of~$X=X_d(a_1,\ldots,a_r)$ of the form
\begin{equation}\label{eq decomp HK intro}
 X\simeq \left\{
\begin{array}{ll}
 X_{d-1}(a_1,\ldots,a_{r-1})\sqcup \A^r \sqcup \bigg(\bigsqcup_{2\leq t'\leq t}U_{t'+r-1} (a_1,\ldots,a_{r})\bigg)     &  \text{if }r>1,\\
 \PP^{t-1}\sqcup \A^1 \sqcup \bigg(\bigsqcup_{2\leq t'\leq t}U_{t'} (a_1)\bigg)    &  \text{if }r=1
\end{array}
   \right.
\end{equation}
(see Section~\ref{sec decomposition of HK} for details), which allows us to work recursively on the dimension of~$X$, and decompose~$\zeta_{X,-K_X}(s)$ as a finite sum of height zeta functions of projective (and affine) spaces and height zeta functions of good open subsets of Hirzebruch--Kleinschmidt varieties of dimension~$d=t+r-1,\ldots,2$. On the one hand, height zeta functions of projective (and affine) spaces are well understood (see Section~\ref{sec height zeta function of projective spaces}). On the other hand, for each good open subset~$U'=U_{t'+r'-1}(a_1,\ldots,a_{r'})$ with~$2\leq t'\leq t,1\leq r'\leq r$ we get a contribution of a height zeta function of the form~$\zeta_{U',L}(s)$ with~$L\in \Pic(X_{d'}(a_1,\ldots,a_{r'}))$ generally distinct from the corresponding anticanonical class. Nevertheless, provided~$L$ is big, we can still describe the analytic properties of~$\zeta_{U',L}(s)$ in detail, see Theorems~\ref{thm:ThmGeneralversion} and~\ref{thm:ThmGeneral when ar=0} in Section~\ref{subsection general height zeta functions}. Note that the bigness condition on~$L$ is necessary, because when~$L$ is not big, $\zeta_{U',L}(s)$ has no finite abscissa of absolute convergence. The aforementioned results allow us to give precise estimates on the number of rational points of large height on each component in~\eqref{eq decomp HK intro}; see Theorems~\ref{thm:AsymptoticFormulaGeneralversion} and~\ref{thm:AsymptoticFormulaGeneralversion when ar=0} in Section~\ref{sec counting rationa points large height}. The following example illustrates these results. 

\begin{example}\label{ex intro HK threefold}
Given an integer~$a>0$, let us consider the Hirzebruch--Kleinschmidt threefold~$X:=X_3(a)$ with projection map~$\pi:X_3(a)\to \PP^2$. The decomposition~\eqref{eq decomp HK intro} in this case becomes
\begin{equation}
\label{eq:decomposition of example}
    X\simeq %F\sqcup F'\sqcup U\sqcup U' \simeq 
    \PP^2 \sqcup \A^1 \sqcup U_3(a) \sqcup U_2(a).
\end{equation}
Put~$U:=U_3(a),U':=U_2(a)\subset X':=X_2(a)$ and let~$\pi':X'\to \PP^1$ denote the corresponding projection map. The anticanonical class on~$X$ is~$-K_X=\CO_X(2)\otimes \pi^\ast (\CO_{\PP^2}(3-a))$ and~$\eta_X=\mathrm{g.c.d.}(2,3-a)$. In particular, $\eta_X=1$ if~$a$ is even, and~$\eta_X=2$ if~$a$ is odd. Using Corollary~\ref{cor:MainThmAnticanonicalCaseoverFF} we get
%that~$\zeta_{U,-K_X}(s)$ converges absolutely for~$\Re(s)>1$ and has meromorphic continuation to~$\CC$ with a pole of order 2 at~$s=1$ with
%$$\lim_{s\to 1}(s-1)^2\zeta_{U,-K_X}(s)=\frac{q^{5(1-g)} h_K^2}{\zeta_K(3)\zeta_K(2)(a+3)2(q-1)^2\log(q)^2}.$$
$$\tilde{N}(U,H_{-K_X},q^M)=(C_1 M+C_0(M))q^{M} +O\left(q^{\delta_1 M}\right) \quad \text{as }M\to \infty,M\in \eta_{X}\ZZ,$$
for every~$\delta_1>\frac{1}{2}$, with~$C_1>0$ and~$C_0(M)\in \RR$ depending only on the class of~$M$ mod~$2(a+3)$ if~$a$ is even, and only on~$\frac{M}{2}$ mod~$\frac{a+3}{2}$ if~$a$ is odd. By Lemma~\ref{lem restriction of heights} in Section~\ref{sec decomposition of HK} the component~$U_2(a)$ in~\eqref{eq:decomposition of example} has height function~$H_L$ with~$L:=\CO_{X'}(2)\otimes (\pi')^\ast (\CO_{\PP^1}(3-a))\in \Pic(X')$, which is big but different from the anticanonical class~$-K_{X'}$. Nevertheless, by Theorem~\ref{thm:AsymptoticFormulaGeneralversion}  we get, for every~$\delta_2>\frac{a+2}{a+3}$, the estimate
$$\tilde{N}(U',H_{-K_X},q^M)=C_{2}(M)q^{M}+O\left(q^{\delta_2 M}\right) \quad \text{as }M\to \infty,M\in \eta_{X}\ZZ,$$
with~$C_{2}(M)>0$ and depending only on the residue class of~$M$ mod 2 if~$a$ is even, and independent of~$M$ if~$a$ is odd. 
%know that~$\zeta_{U',L}$ converges absolutely for~$\Re(s)>1$ and  has meromorphic continuation to~$\CC$ with a simple pole at~$s=1$.
Now, by the same Lemma~\ref{lem restriction of heights} the component~$\PP^2$ in~\eqref{eq:decomposition of example} has height  function~$H_{\PP^2}^{(3-a)}$ where~$H_{\PP^2}$ denotes the ``standard height function'' on~$\PP^2(K)$ defined in Section~$\ref{sec height zeta function of projective spaces}$. We see that~$\PP^2$ has no rational points of height~$>1$ when~$a\geq 3$. Assuming~$a<3$, we can apply the estimate~\eqref{eq Wan's estimate} in Section~\ref{sec height zeta function of projective spaces} and get
$$\tilde{N}(\PP^2,H_{-K_X},q^M)=C_{3}q^{\frac{3M}{3-a}}+O\left(q^{\delta_3 M}\right) \quad \text{as }M\to \infty,M\in (3-a)\ZZ,$$
for every~$\delta_3>\frac{1}{2(3-a)}$ where~$C_3>0$, while~$\tilde{N}(\PP^2,H_{-K_X},q^M)=0$ if~$M \not \in (3-a)\ZZ$. 
%Theorem~\ref{thm:AsymptoticFormulaGeneralversion when ar=0} in Section~\ref{sec counting rationa points large height} to deduce
%that~$\zeta_{\PP^2}((3-a)s)$ converges absolutely for~$\Re(s)>\frac{3}{3-a}$ and has meromorphic continuation to~$\CC$ with a simple pole at~$s=\frac{3}{3-a}>1$. % and an explicitly computable residue. 

Finally, by Lemma~\ref{lem restriction of heights} the component~$\A^1=\PP^1\setminus \{[1:0]\}$ in~\eqref{eq:decomposition of example} has height function~$H_{\PP^1}^2$ where~$H_{\PP^1}$ is the standard height function on~$\PP^1(K)$. In particular, by the estimate~\eqref{eq Wan's estimate} we get
$$\tilde{N}(\A^1,H_{-K_X},q^M)= C_{4}q^{M}+O\left(q^{\delta_4M}\right) \quad \text{as }M\to \infty,M\in 2\ZZ,$$
for every~$\delta_4>\frac{1}{4}$ where~$C_4>0$, while~$\tilde{N}(\A^1,H_{-K_X},q^M)=0$ if~$M \not \in 2\ZZ$. 
%Theorem~\ref{thm height zeta function projective} we know that~$\zeta_{\A^1}(2s)$ converges absolutely for~$\Re(s)>1$ and has meromorphic continuation to~$s\in \CC$ with a simple pole at~$s=1$.

We can then distinguish the following two cases:
\begin{enumerate}
    \item If~$a=1,2$, then~$\eta_X=3-a=2,1$ respectively, and the number of rational points of large anticanonical height in~$X_3(a)$ is ``dominated'' by the number of those points in the component~$\PP^2$ in~\eqref{eq:decomposition of example}. More precisely, we have that
    $$\tilde{N}(X,H_{-K_X},q^M)\sim \tilde{N}(\PP^2,H_{-K_X},q^M)\sim C_{3}q^{\frac{3M}{3-a}} \quad \text{as }M\to \infty,M\in \eta_X\ZZ.$$
    %converges absolutely for~$\Re(s)>\sigma_a:=\frac{3}{3-a}>1$ and has meromorphic continuation to~$\CC$ with a simple pole at~$s=\sigma_a$, while
    %$$\zeta_{X\setminus \PP^2,-K_X}(s)=\zeta_{\A^1}(2s)+\zeta_{U,-K_X}(s)+\zeta_{U',L}(s)$$
    %converges absolutely for~$\Re(s)>1$. 
    Moreover, by~\eqref{eq Wan's estimate} we have
    $$C_3=\frac{h_Kq^{3(1-g)}}{\zeta_K(3)(q-1)}.$$
    %Theorem~\ref{thm height zeta function projective} in Section~\ref{sec height zeta function of projective spaces}, we can compute    
    %$$\lim_{s\to \sigma_a}\left(s-\sigma_a\right)\zeta_{X,-K_X}(s)=\lim_{s\to \sigma_a}\left(s-\sigma_a\right)\zeta_{\PP^2}((3-a)s)=\frac{h_Kq^{3(1-g)}}{(3-a)\zeta_K(3)(q-1)\log(q)}.$$
    \item If~$a\geq 3$, then the component~$\PP^2$ in~\eqref{eq:decomposition of example} does not contribute to the number of rational points of large height, and
    %In fact, $N(\PP^2,-K_X,B)=\infty$ for all~$B\geq 1$. In this case we also have
    %$$\lim_{s\to 1}\left(s-1\right)^2\zeta_{X\setminus \PP^2,-K_X}(s)=\lim_{s\to 1}\left(s-1\right)^2\zeta_{U,-K_X}(s),$$
    $\tilde{N}(X,H_{-K_X},q^M)$ is ``dominated'' by the number of those points in the good open subset~$U=U_2(a)$. More precisely, we have
    $$\tilde{N}(X,H_{-K_X},q^M)\sim \tilde{N}(U,H_{-K_X},q^M)\sim C_1 Mq^{M} \quad \text{as }M\to \infty,M\in \eta_X\ZZ.$$
    Moreover, we have
    \begin{equation*}%\label{eq C1 example intro}
          C_1=\frac{q^{5(1-g)} \eta_X h_K^2}{\zeta_K(3)\zeta_K(2)2(a+3)(q-1)^2}.  
    \end{equation*}
\end{enumerate}
%In both cases, we see that only after removing the closed subvariety~$\PP^2\subset X_3(a)$ we obtain a height zeta function~$\zeta_{X\setminus \PP^2,-K_X}(s)$     satisfying the analogue of Conjecture~\ref{conj MP} over global function fields.

We can also include here the easier case~$a=0$ with~$X=X_3(0)\simeq \PP^1\times \PP^2$. In this case one can directly give a good estimate on the number of rational points of large height in the whole variety~$X$. Indeed, a direct application of  Theorem~\ref{thm:AsymptoticFormulaGeneralversion when ar=0} in Section~\ref{sec counting rationa points large height} gives, for every~$\delta>\frac{1}{4}$, the estimate
$$\tilde{N}(X,H_{-K_X},q^M)=(\tilde{C}_1 M+\tilde{C}_0(M))q^{M} +O\left(q^{\delta M}\right) \quad \text{as }M\to \infty,M\in \ZZ,$$
with~$\tilde{C}_0(M)\in \RR$ depending only on~$M$ mod~$6$ and
$$\tilde{C}_1=\frac{q^{5(1-g)}  h_K^2}{\zeta_K(3)\zeta_K(2)6(q-1)^2}.$$ 
\end{example}

Note that the above strategy also applies if we start with a Hirzebruch--Kleinschmidt variety~$X$ with a height function induced by a general big metrized line bundle class~$L\in \Pic(X)$. We refer the reader to Section~\ref{sec ex Hirzebruch} where this is illustrated in the case of Hirzebruch surfaces.

The techniques used in this paper are inspired by the work of Bourqui~\cite{BOURQUI2002} on the anticanonical height zeta function on Hirzebruch surfaces. Combining Bourqui's ideas with some technical computations, we are able to express the height zeta function~$\zeta_{U,L}(s)$, for~$U$ the good open subset of a Hirzebruch--Kleinschmidt variety and for general big metrized line bundles~$L$, essentially as a rational function of degree 2 on the zeta function~$\zeta_K(s)$ of the base field~$K$. In that regard, it would be interesting to investigate if this method can also be applied to other families of projective bundles defined over global function fields. We expect to return to this problem in the near future.

\begin{remark}
    The results obtained in this article served us as a guiding route for our investigations into the analogous problem for Hirzebruch--Kleinschmidt varieties over number fields, which are reported in the companion paper~\cite{HMM24a}.
\end{remark}

This article is organized as follows. In Section~\ref{sec preliminaries} we introduce notation and recall some classical results concerning the geometry of curves defined over finite fields, and analytic properties the zeta functions of global function fields. There, we also present the Tauberian theorem that will be applied to our height zeta functions (see Section~\ref{sec tauberian}), and include a few general properties of toric varieties. In Section~\ref{section HK} we describe the effective cone, anticanonical class, and big line bundle classes in Hirzebruch--Kleinschmidt varieties. In Section~\ref{sec height zeta function of projective spaces} we recall an asymptotic formula due to Wan~\cite{Wan92} for the number of rational points of large ``standard height'' on projective spaces over global function fields, and the  analytic properties of the associated height zeta functions. The height zeta functions of Hirzebruch--Kleinschmidt varieties are studied in Section~\ref{sec height zeta HK}. In particular, their analytic properties are presented in Section~\ref{subsection general height zeta functions}. These results are used in Section~\ref{sec counting rationa points large height} to prove our asymptotic formulas for the number of rational points of large height. Section~\ref{sec height zeta HK} also includes a precise description of our decomposition of Hirzebruch--Kleinschmidt varieties (see Section~\ref{sec decomposition of HK}) and the proof of Theorem~\ref{thm:MainThmAnticanonicalCaseoverFF} (see Section~\ref{sec proof of main thm for anticanonical class}). Finally, in Section~\ref{sec ex Hirzebruch} we apply our general results to Hirzebruch surfaces. %In particular, we recover the main result of~\cite{BOURQUI2002} as a particular case.

\section*{Acknowledgments} 
The authors thank \textsc{Lo\"{\i}s Faisant} and \textsc{Daniel Loughran} for geometric insight on the Manin--Peyre conjecture for varieties defined over global function fields, and \textsc{Giancarlo Lucchini-Arteche} for helpful comments on the first Galois cohomology group of an algebraic variety.

S.~Herrero's research is supported by ANID FONDECYT Iniciaci\'on grant 11220567 and by SNF grant CRSK-2{\_}220746. T.~Martinez's research is supported by ANID FONDECYT Iniciaci\'on grant 11220567, ANID BECAS Doctorado Nacional 2020-21201321, and by UTFSM/DPP Programa de Incentivo a la Investigación Cient\'ifica (PIIC) Convenio N°055/2023. P.~Montero's research is supported by ANID FONDECYT Regular grants 1231214 and 1240101.

  \section{Preliminaries}\label{sec preliminaries}

 \subsection{Basic notation}
 \label{subsection:Notation}
    Let $\mathcal{C}$ be a projective, smooth, geometrically irreducible curve of genus $g$ defined over the finite field $\mathbb{F}_q$ of~$q$ elements (as usual, $q$ is a positive power of a rational prime). Throughout this article we let~$K$ denote the field~$\mathbb{F}_q(\mathcal{C})$ of rational functions of $\mathcal{C}$. Associated to~$K$ we have the following objects (see, e.g., 
    \cite{Ros02} for details):
	\begin{itemize}
        \item The set of valuations~$\operatorname{Val}(K)$ of $K$, which is in bijection %normalized in such way that it corresponds to
        with the set of closed points of $\mathcal{C}$.
		\item For each $v\in \Val(K)$, let $\mathcal{O}_v$ be the associated valuation ring and $m_v$ its maximal ideal, $k_v:=\mathcal{O}_v/m_v$ and $f_v:=[k_v:\F_q]$, so we have $\#k_v=q^{f_v}$. We define for $x\in K$, $|x|_v:=q^{-f_vv(x)}$, so that for every~$x\in K^\times$  we have the product formula 
	\begin{equation}\label{eq product formula}
	    \prod_{v\in \Val(K)}|x|_v=1.
	\end{equation}
         \item The  free abelian group $\Div(K)$ generated by $\Val(K)$. The elements of $\Div(K)$ are finite sums of the form $\sum_{v\in \Val(K)}n_v v$ with~$n_v\in \ZZ$ and~$n_v=0$ for all but finitely many~$v\in \Val(K)$.
	
	\item For $D=\sum_{v\in \Val(K)}n_vv\in            \Div(K)$ and $v\in \Val(K)$ define                $v(D):=n_v$ and set
	$$\Div^+(K):=\{D\in \Div(K): v(D)\geq 0,           \forall v\in \Val(K)\}.$$
	Moreover, for $D_1,\ldots,D_n\in \Div(K)$, we put 
	$$\sup(D_1,\ldots,D_n):=\sum_{v\in                 \Val(K)}\sup\{v(D_1),\ldots,v(D_n)\}v.$$
	\item For $x\in K^\times$, define the divisor      of $x$ as $(x):=(x)_0-(x)_\infty$ where
    \begin{equation*}
        %\begin{split}
          (x)_0:=\sum_{\substack{v\in                        \Val(K)\\v(x)>0}}v(x)v, \quad 
          (x)_\infty:=-\sum_{\substack{v\in \Val(K)\\v(x)<0}}v(x)v.
        %\end{split}
    \end{equation*}
	We also put~$(0)=(0)_0=(0)_\infty=0$. 
    \item The degree function $\mathrm{\deg }:\Div(K)\to \ZZ$ defined by $\mathrm{\deg }(D):=\sum_{v\in \Val(K)}f_vv(D)$. We have~$\deg((x))=0$ for all~$x\in K$ (by the product formula~\eqref{eq product formula} in the case~$x\neq 0$).
    \item The class number~$h_K=\# (\Div^0(K)/(K))$, where~$\Div^0(K):=\{D\in \Div(K):\deg(D)=0\}$ and~$(K):=\{(x):x\in K\}$.
     \end{itemize}

For a vector~$\mathbf{a}=(a_1,\ldots,a_r)\in \mathbb{N}_0^r$, we write $|\mathbf{a}|:=\sum_{i=1}^r a_i$.

In this article all varieties will be assumed to be irreducible, reduced and separated schemes of finite type over the base field $K$. For simplicity we write~$\mathbb{A}^n$ and~$\mathbb{P}^{n}$ for the affine and projective space of dimension~$n$ over~$K$, and products of varieties are to be understood as fiber products over~$\mathrm{Spec}(K)$. %The sheaf of regular functions on a variety~$X$ is denoted by~$\mathscr{O}_X$.

%\subsection{A Tauberian theorem}
%Once the analytic behavior of the height zeta function is established, results on the asymptotic behavior are obtained using the classical strategy based on the following Tauberian theorem (see \cite{DelangeTauberianThm}).
%\begin{thm}
%		Let $X$ be a countable set, $H:X\to \RR_{>0}$ a function an suppose that
% 	$$Z(s)=\sum_{x\in X} 
% 	H(x)^{-s}$$
% 	is absolutely convergent for $\Re(s)>a>0$ and
% 	$$Z(s)=\frac{g(s)}{(s-a)^b}+h(s)$$ with $g(s), h(s)$ holomorphic for $\Re(s)\geq a$, $g(a)\neq 0$ and $b$ a positive integer. Then 
% 	$$N(X,H,B)=\{x\in X:H(x)\leq B\},$$
% 	is a finite number and
% 	$$N(X,H,B)=\frac{g(a)}{(b-1)!\, a}B^a (\log B)^{b-1}(1+o(1)).$$
%		\label{thm:tauberianThm}
%	\end{thm}

%More precisely, the above Tauberian theorem, together with Theorem \ref{thm:principalResult}, implies that the conjecture holds for this family of varieties when defined over global function fields. Moreover, the nature of the proof, based on Bourqui's original approach for Hirzebruch surfaces, allows us to compute the asymptotic constant very explicitly (see Theorem \ref{thm:ThmGeneralversion}).

\subsection{Geometric and arithmetic tools}\label{subsection arithmetic tools}

For~$D\in \Div(K)$, let $\ell(D):=\dim_{\mathbb{F}_q} H^0(\mathcal{C},\mathcal{O}_\mathcal{C}(D))$, where 
	$$H^0(\mathcal{C},\mathcal{O}_\mathcal{C}(D)):=\{x\in K^\times:(x)+D\geq 0\}\cup\{0\}.$$
	Note that $\ell(D)=0$ if $\deg(D)<0$ and if $D\geq 0$ then $\#\{x\in K:(x)_\infty\leq D\}=q^{\ell(D)}$. 
 
 As a consequence of the  Riemann--Roch theorem for curves (see, e.g., \cite[Theorem~5.4]{Ros02}) we have
	that
	$$\ell(D)-\ell(K_\mathcal{C}-D)=\deg(D)+1-g,$$
 where~$K_\mathcal{C}$ is a canonical divisor class of the curve~$\mathcal{C}$. In particular, if $\deg(D)>2g-2=\deg (K_\mathcal{C})$ then
  $$\ell(D)=\deg(D)+1-g.$$

We also define the zeta function
	$$\Zet_K(T):=\sum_{D\geq 0}T ^{\deg (D)}=\prod_{v\in \Val(K)}\left( 1-T^{f_v}\right)^{-1},$$
	and define the zeta function of $K$ as
	$$\zeta_K(s):=\Zet_K(q^{-s}).$$
	The series defining~$\zeta_K(s)$ converges absolutely and uniformly on compact subsets of the half-plane~$\Re(s)>1$ and has no zeroes in that domain. Moreover, there exits a polynomial~$\mathbf{L}_K(X)\in \ZZ[X]$ of degree~$2g$ such that
 $$\zeta_K(s)=\frac{\mathbf{L}_K(q^{-s})}{(1-q^{-s})(1-q^{1-s})}$$
 and~$\zeta_K(s)$ has simple poles at~$s=1$ and~$s=0$ with
 \begin{equation*}\label{eq residue zetaK}
   \operatorname{Res}_{s=1}\zeta_K(s)=\frac{h_Kq^{1-g}}{(q-1)\log(q)} 
 \end{equation*}
  (see, e.g., \cite[Theorem~5.9]{Ros02}).

\subsection{A Tauberian theorem for global function fields}\label{sec tauberian}

The following theorem is a variant of~\cite[Theorem~17.4]{Ros02}.

\begin{thm}\label{thm:tauberianThmFF}
Let $X$ be a countable set and~$H:X\to q^{\mathbb{N}_0}$ a function. Suppose that
 	$$\zeta_H(s):=\sum_{x\in X} 
 	H(x)^{-s}$$
 	is absolutely convergent for $\Re(s)>a>0$ and there exists~$\varepsilon>0$ such that~$\zeta_H(s)$ extends to a meromorphic function on the half-plane $\Re(s)\geq a-\varepsilon$ satisfying the following properties:
  \begin{enumerate}
      \item $\zeta_H(s)$ has a pole of order~$b\geq 1$ at~$s=s_0:=a$,
      \item There exists real numbers~$0=\theta_0<\theta_1<\ldots<\theta_n<1$ such that~$\zeta_H(s)$ is holomorphic for~$\Re(s)> a-\varepsilon,s\not \in \left\{ a+\frac{2\pi i \theta_k}{\log(q)}+\frac{2\pi i m}{\log(q)}:k\in \{0,\ldots,n\},m\in \ZZ\right\}$.
      \item For every~$k\in \{1,\ldots,n\}$, $\zeta_H(s)$ is either holomorphic or has a pole of order~$\leq b$ at~$s=s_k:=a+\frac{2\pi i \theta_k}{\log(q)}$. 
  \end{enumerate}
  Then, for every~$M\geq 0$ the cardinality
 	$$\tilde{N}(X,H,q^M):=\#\{x\in X:H(x)=q^M\}$$
 	is finite and for each~$k\in \{0,\ldots,n\}$ there exists a polynomial~$Q_k(M)$ of degree~$b-1$ and leading coefficient~$r_k:=\lim_{s\to s_k}(s-s_k)^{b}\zeta_H(s)$ such that %there exists~$\varepsilon\in\, ]0,a[$ such that
 \begin{equation*}%\label{eq asymptotic formula from TTFF}
    \tilde{N}(X,H,q^M)=\frac{\log(q)^{b}}{(b-1)!}q^{aM}\sum_{k=0}^n e^{2\pi i \theta_kM}Q_k(M)+ O\left(q^{\delta M}\right) \quad \text{as }M\to \infty,M\in \ZZ,  
 \end{equation*}
for every~$\delta>a-\varepsilon$.
\end{thm}
\begin{proof}
For~$\Re(s)>a$ we have
$$\zeta_H(s)=\sum_{M=0}^{\infty}\tilde{N}(X,H,q^M)q^{-sM}.$$
This implies that~$\tilde{N}(X,H,q^M)$ is finite for every~$M\geq 0$. Now, put
\begin{equation}\label{eq Zet tauberian}
  \Zet_H(T):=\sum_{M=0}^{\infty}\tilde{N}(X,H,q^M)T^{M} 
\end{equation}
so that~$\Zet_H(q^{-s})=\zeta_H(s)$. By hypothesis, the series~$\Zet_H(T)$ converges absolutely for~$|T|< q^{-a}$ and it extends to a meromorphic function on~$|T|\leq q^{-(a-\varepsilon)}$ with poles contained in~$|T|\in \{q^{-a},q^{-(a-\varepsilon)}\}$. Since the closed disk~$|T|\leq q^{-(a-\varepsilon)}$ is compact, there exists~$\varepsilon'>\varepsilon$ such that~$\Zet_H(T)$ is meromorphic for~$|T|\leq q^{-(a-\varepsilon')}$ and has no poles on~$q^{-(a-\varepsilon)}<|T|\leq q^{-(a-\varepsilon')}$. 
Fix~$\eta\in \, ]0,q^{-a}[$ and consider the curves~$\gamma_1(t):=\eta e^{it}$ and~$\gamma_2(t):=q^{-(a-\varepsilon')} e^{it}$ for~$t\in [0,2\pi]$. 
Recall that~$s_k:=a+\frac{2\pi i \theta_k}{\log(q)}$ and put~$T_k:=q^{-s_k}$ for~$k\in \{0,\ldots,n\}$. Then, $T_0,\ldots,T_n$ are exactly the poles of~$\Zet_H(T)$ in~$|T|=q^{-a}$. We also denote by~$\{T_{n+1},\ldots,T_{n+m}\}$ the (possibly empty) set of poles of~$\Zet_H(T)$ in~$|T|=q^{-(a-\varepsilon)}$. Now, for every integer~$M>0$ we have, by Cauchy's residue theorem, the equality
\begin{equation}\label{eq Cauchy residue}
\frac{1}{2\pi i}\int_{\gamma_2}\frac{\Zet_H(T)}{T^{M+1}}\textup{d}t=\frac{1}{2\pi i}\int_{\gamma_1}\frac{\Zet_H(T)}{T^{M+1}}\textup{d}t+\sum_{k=0}^{n+m}\operatorname{Res}_{T=T_k}\left(\frac{\Zet_H(T)}{T^{M+1}}\right)%+\sum_{|T|=q^{-(a-\varepsilon)}}\operatorname{Res}_{T}\left(\frac{\Zet_H(T)}{T^{M+1}}\right).    
\end{equation}
On the one hand, since~$\Zet_H(T)$ is holomorphic for~$|T|=q^{-(a-\varepsilon')}$, we have
\begin{equation}\label{eq bound integra over gammadelta}
   \frac{1}{2\pi i}\int_{\gamma_2}\frac{\Zet_H(T)}{T^{M+1}}\textup{d}t=O\left(q^{(a-\varepsilon') M}\right) 
\end{equation}
with an implicit constant independent of~$M$. On the other hand, from~\eqref{eq Zet tauberian} we have
\begin{equation}\label{eq integral over gammaepsilon}
\frac{1}{2\pi i}\int_{\gamma_1}\frac{\Zet_H(T)}{T^{M+1}}\textup{d}t=\tilde{N}(X,H,q^M).    
\end{equation}
Next, for~$k\in \{0,\ldots,n\}$ we write
$$\Zet_H(T)=\sum_{j=-b}^{\infty}c_{k,j}(T-T_k)^j$$
with~$c_{k,j}\in \CC$ and recall the Taylor expansion
$$T^{-(M+1)}=\sum_{j=0}^{\infty}  {M+j \choose M}(-1)^jT_k^{-(M+j+1)} (T-T_k)^j,$$
in order to get
\begin{eqnarray*}
\operatorname{Res}_{T=T_k}\left(\frac{\Zet_H(T)}{T^{M+1}}\right)&=&-T_k^{-M}\sum_{j=1}^{b}c_{k,-j}{M+j-1 \choose M}(-1)^{j}T_k^{-j}\\
&=&-q^{Ms_k}\sum_{j=1}^{b}c_{k,-j}{M+j-1 \choose M}(-1)^{j}q^{js_k}. 
\end{eqnarray*}
Since~${M+j-1 \choose M}$ is a polynomial in~$M$ of degree~$j-1$ and leading term~$\frac{M^{j-1}}{(j-1)!}$, we obtain
$$\operatorname{Res}_{T=T_k}\left(\frac{\Zet_H(T)}{T^{M+1}}\right)=(-1)^{b+1}q^{Ms_k}\frac{P_k(M)}{(b-1)!}q^{bs_k},$$
where~$P_k(M)$ is a polynomial of degree~$b-1$ with leading coefficient~$c_{k,-b}$.
But
\begin{eqnarray*}
 c_{k,-b}&=& \lim_{T\to T_k}(T-T_k)^{b}\Zet_H(T)\\
 &=& \lim_{s\to s_k} (q^{-s}-q^{-s_k})^{b}\zeta_H(s)\\
 &=& \lim_{s\to s_k} \left(\frac{q^{-s}-q^{-s_k}}{s-s_k}\right)^{b}\lim_{s\to s_k}(s-s_k)^{b}\zeta_H(s)\\
 &=& (-\log(q)q^{-s_k})^{b}r_k,
\end{eqnarray*}
hence
\begin{equation}\label{eq residue computation}
    \operatorname{Res}_{T=T_k}\left(\frac{\Zet_H(T)}{T^{M+1}}\right)= -q^{Ms_k}\log(q)^{b}\frac{Q_k(M)}{(b-1)!}=-q^{Ma}e^{2\pi i \theta_kM}\log(q)^{b}\frac{Q_k(M)}{(b-1)!},
\end{equation}
where~$Q_k(M)$ is a polynomial of degree~$b-1$ with leading coefficient~$r_k$. Similarly, for~$k\in \{n+1,\ldots,n+m\}$ we have
\begin{equation}\label{eq residue computation2}
    \operatorname{Res}_{T=T_k}\left(\frac{\Zet_H(T)}{T^{M+1}}\right)= O\left(M^{c-1}q^{M(a-\varepsilon)}\right),
\end{equation}
for every integer~$c\geq 1$ larger than the maximum among the order of the poles of~$\Zet_H(s)$ on~$|T|=q^{-(a-\varepsilon)}$ (if~$\Zet_H(s)$ has no poles in that circle, then we choose~$c=1$).
Then, the desired results follows from~\eqref{eq Cauchy residue}, \eqref{eq bound integra over gammadelta}, \eqref{eq integral over gammaepsilon}, \eqref{eq residue computation} and~\eqref{eq residue computation2}, together with the fact that
$$M^{c-1}q^{M(a-\varepsilon)}= O\left(q^{\delta M}\right) \quad \text{as }M\to \infty,$$
for every~$\delta>a-\varepsilon$.  %, together with~$q^{Ms_k}=q^{aM}e^{2\pi i \theta_kM}$. 
This proves the theorem.
\end{proof}

\begin{remark}
\begin{enumerate}
    \item The third condition in Theorem~\ref{thm:tauberianThmFF} is equivalent to the existence of the limit~$r_k=\lim_{s\to s_k}(s-s_k)^{b}\zeta_H(s)$ for every~$k\in\{1,\ldots ,n\}$. In particular, there is no need to compute the exact order of~$\zeta_H(s)$ at each point~$s_k$. %, although it is clear that only poles of order~$B$ contribute to the main term in formula~\eqref{eq asymptotic formula from TTFF}.
\item In all our applications of Theorem~\ref{thm:tauberianThmFF}, the points~$s_k=a+\frac{2\pi i \theta_k}{\log(q)}$ can be assumed to have~$\theta_k= \frac{k}{N}$ with~$k=0,\ldots,N-1$ for some integer~$N\geq 1$. This implies
 \begin{equation*}%\label{eq asymptotic formula from TTFF rmk}
    \tilde{N}(X,H,q^M)=\frac{\log(q)^{b}}{(b-1)!}Q(M)q^{aM}+ O\left(q^{\delta M}\right) \quad \text{as }M\to \infty,M\in \ZZ,  
 \end{equation*}
with
$$Q(M):=C_{b-1}(M)M^{b-1}+C_{b-2}(M)M^{b-2}+\ldots +C_0(M)$$
and~$C_k(M)\in \RR$ depending only on the residue class of~$M$ mod $N$. Moreover, the leading term has
$$C_{b-1}(M)=\sum_{k=0}^{N-1}e^{2\pi i kM/N}r_k=N\lim_{s\to a^+}(s-a)^b\sum_{\substack{m\geq 0 \\ m\equiv M \, (\text{mod }N)}}\tilde{N}(X,H,q^m)q^{-sm}.$$
%and this value depends only on the residue class of~$M$ mod~$N$. 
In particular, if~$C_{b-1}(M_0)>0$ for a given integer~$M_0\in \{0,\ldots,N-1\}$, then one obtains
%Moreover, in all our applications we will check that~$c_M$ takes  strictly positive values, % (provided~$H(X)\subseteq q^{\NN_0}$ contains all large enough powers of~$q$), which in particular allows us to deduce 
the asymptotic formula
$$\tilde{N}(X,H,q^M)\sim \left(\frac{\log(q)^{b}}{(b-1)!}C_{b-1}(M_0)\right)M^{b-1}q^{aM} \quad \text{as }M\to \infty,M\equiv M_0 \, (\text{mod }N).$$
Finally, note that the first condition in the above theorem implies that~$C_{b-1}(M_0)>0$ for at least one~$M_0\in \{0,\ldots,N-1\}$.
\end{enumerate}
\end{remark}
 
\subsection{Toric varieties}
\label{subsec:ToricVectorBundles}

We refer the reader to \cite{CoxLittleSchenckToricVarieties} for the general theory of toric varieties. Here, we simply recall that a toric variety is a normal and separated variety that contains a maximal torus $\TT$ as an open subset and admits an effective regular action of $\TT$ extending the natural action of this torus over itself.  The relevant toric varieties appearing in this paper are all smooth. We recall that on a smooth variety~$X$ every Weil divisor is Cartier, and in particular~$\operatorname{Cl}(X)\simeq \operatorname{Pic}(X)$.

%A toric vector bundle over a toric variety $X$ is a vector bundle $\pi:V\rightarrow X$ such that the  action of $\TT$ on $X$ extends to an action on $V$ in such a way that $\pi$ is $\TT$-equivariant and the action is linear on the fibers. The algebraic variety $V$ is not toric in general, and Oda \cite[\S 7.6]{Oda78} notes that the toric vector bundles which are toric varieties are precisely the decomposables ones, \ie, those corresponding to locally free sheaves of the form $\mathcal{E}=\CO_X(D_0)\oplus \cdots\oplus\CO_X(D_r)$ for some $\TT$-invariant Cartier divisors $D_0,\ldots, D_r$ on $X$.

\section{Hirzebruch--Kleinschmidt varieties}\label{section HK}
 
	Given integers $r\geq 1, t\geq 2, 0\leq a_1\leq \ldots\leq a_r,$ consider the vector bundle
	$$\mathcal{E}:=\CO_{\PP ^{t-1}}\oplus\CO_{\PP ^{t-1}}(a_1)\oplus\cdots\oplus\CO_{\PP ^{t-1}}(a_r).$$
	%Here, as usual, we put $\CO_{\PP ^{t-1}}(a_i):=\CO_{\PP ^{t-1}}(a_iH_0)$ where $H_0\subset \PP^{t-1}$ is a hyperplane, which we can choose as~$H_0=\{x_0=0\}$ in homogeneous coordinantes~$[x_0:\ldots:x_t]$.
 Associated to~$\mathcal{E}$ there is a projective bundle~$\pi:\mathbb{P}(\mathcal{E})\to \PP^{t-1}$. The variety~$\mathbb{P}(\mathcal{E})$ is smooth, projective and toric, of dimension  
    $$d:=\dim(\mathbb{P}(\mathcal{E}))=r+t-1,$$
    and has~$\operatorname{Pic}(\mathbb{P}(\mathcal{E}))\simeq \mathbb{Z}^2$ (see~\cite[Chapter~7]{CoxLittleSchenckToricVarieties} or~\cite[Section~3]{HMM24a}). Conversely, Kleinschmidt \cite{Kleinschmidt88} proved the following classification result of smooth projective\footnote{Actually, the classification in \cite{Kleinschmidt88} does not assume the projectivity hypothesis.} toric varieties of Picard rank 2, and gave explicit defining equations  in \cite[Theorem 3]{Kleinschmidt88}.

	\begin{thm}[Kleinschmidt]
\label{thm:EquationsHirzebruchKleinschmidtVar}
		Let $X$ be a smooth projective toric variety with $\Pic(X)\simeq \ZZ^2$. Then, there exist integers $r\geq 1, t\geq 2, 0\leq a_1\leq\cdots\leq a_r$ with $r+t-1=d=\dim(X)$ such that
		$$X\simeq \PP(\CO_{\PP ^{t-1}}\oplus\CO_{\PP ^{t-1}}(a_1)\oplus\cdots\oplus\CO_{\PP ^{t-1}}(a_r)).$$
  Moreover, if we write $I_k:=\{1,\ldots,k\}\subset \ZZ$, then $X$ is isomorphic to the subvariety~$X_d(a_1,\ldots,a_r)\subset \PP^{rt}\times \PP^{t-1}$ given in homogeneous coordinates $([x_0: (x_{ij})_{i\in I_t,j\in I_r}],[y_1:\ldots:y_t])$ by the equations
\begin{equation}\label{equations HK}
x_{mj}y^{a_j}_n=x_{nj}y^{a_j}_m ,\, \text{ for all } j\in I_{r} \text{ and all } m,n\in I_t \text{ with }m\neq n.   
\end{equation}
	\end{thm}

Recall from the introduction that we define Hirzebruch--Kleinschmidt varieties exactly as the family of algebraic varieties~$X_d(a_1,\ldots,a_r)$ appearing in Theorem~\ref{thm:EquationsHirzebruchKleinschmidtVar}.

% \begin{defi}\label{defi:HK varieties}
% Given integers~$r\geq 1$, $t\geq 2$ and $0\leq a_1 \leq \cdots \leq a_r$, the \emph{Hirzebruch--Kleinschmidt} variety $X_d(a_1,\ldots,a_r)$ is defined as 
% the subvariety of~$\PP^{tr}\times \PP^{t-1}$ given in homogeneous coordinates $([x_0: (x_{ij})_{i\in I_t,j\in I_r}],[y_1:\ldots:y_t])$ by the equations
%\begin{equation}
%x_{mj}y^{a_j}_n=x_{nj}y^{a_j}_m ,\, \text{ for all } j\in I_{r} \text{ and all } m,n\in I_t \text{ with }m\neq n,   
%\end{equation}
%where $d=\dim(X_d(a_1,\ldots,a_r))=r+t-1$ and $I_k=\{1,\ldots,k\}\subset \ZZ$.
% \end{defi}

 \subsection{Effective divisors}

%The following description of the cone of effective divisors on~$X_d(a_1,\ldots,a_r)$ is a natural extension to higher dimensions of the description for Hirzebruch surfaces (see, e.g., \cite[Chapter V, Corollary~2.18]{Har77}). 

In this section we give a basis of~$\Pic(X)$ for a Hirzebruch--Kleinschmidt variety~$X$, and use it to describe the anticanonical class, the effective cone, and the big line bundle classes.

\begin{prop}\label{prop effective cone} 
		Let $X=X_d(a_1, \ldots,a_r)$ be a Hirzebruch--Kleinschmidt variety and let us denote by~$f$ the class of~$\pi^\ast \mathcal{O}_{\mathbb{P}^{t-1}}(1)$ and by~$h$ the class of~$\mathcal{O}_X(1)$, both in~$\Pic(X)$. %  the classes of~$E_0$ and~$D_0$ in $\Pic(X)$.  %of the tautological line bundle $\mathcal{O}_X(1)$ and the pullback $\pi^\ast \mathcal{O}_{\mathbb{P}^{t-1}}(1)$, respectively. 
  Then:
  \begin{enumerate}
  \item $\Pic(X)= \ZZ h \oplus \ZZ f$.
      \item The anticanonical divisor class of $X$ is given by
      $$-K_X=(r+1)h+\left(t-|\mathbf{a}| \right) f,$$
      where $|\mathbf{a}|=\sum_{i=1}^r a_i$.
      \item The cone of effective divisors of $X$ is given by
      $$\Lambda_{\textup{eff}}(X)=\{\gamma  h+\xi f:\gamma\geq 0, \xi\geq -\gamma a_r\}\subset \Pic(X)_\RR$$
      where $\Pic(X)_\RR:=\Pic(X)\otimes_{\ZZ}\RR$.
  \end{enumerate}
%  In particular, $-K_X$ is in the interior of~$\Lambda_{\textup{eff}}(X)$.
  \end{prop}

\begin{proof} The analogue statement for the variety
$$Y:=\PP(\CO_{\PP ^{t-1}}\oplus\CO_{\PP ^{t-1}}(-a_r)\oplus\cdots\oplus\CO_{\PP ^{t-1}}(a_{r-1}-a_r))$$
is given in~\cite[Proposition~3.4]{HMM24a}. More precisely, putting~$f_Y:=\pi_Y^\ast \mathcal{O}_{\mathbb{P}^{t-1}}(1)$ and~$h_Y:=\mathcal{O}_Y(1)$, where~$\pi_Y:Y\to \PP^{t-1}$ is the natural projection, it is proven that
\begin{eqnarray*}
    \Pic(Y)&= & \ZZ h_Y \oplus \ZZ f_Y,\\
    -K_Y&=&(r+1)h_Y+\left((r+1)a_r+t-|\mathbf{a}| \right) f_Y, \text{ and}\\
    \Lambda_{\textup{eff}}(Y)&=&\{\gamma  h_Y+\xi f_Y:\gamma\geq 0, \xi\geq 0\}.
\end{eqnarray*}
Now, under the canonical isomorphism of projective bundles
$$Y\simeq \PP((\CO_{\PP ^{t-1}}\oplus\CO_{\PP ^{t-1}}(-a_r)\oplus\cdots\oplus\CO_{\PP ^{t-1}}(a_{r-1}-a_r))\otimes \CO_{\PP ^{t-1}}(a_r))=X_d(a_1,\ldots,a_r),$$
the line bundle classes~$-K_Y,f_Y$ and~$h_Y$ correspond to~$-K_X,f$ and~$h-a_rf$, respectively (see, e.g., \cite[Chapter V, Corollary~2.18]{Har77}). Then, simple computations lead to the desired results. We omit the details for brevity.
\end{proof}

It follows from~\cite[Theorem~2.2.26]{LazPAGI} that a line bundle class in~$\Pic(X_d(a_1,\ldots,a_r))$ is  big if and only if it lies in the interior of the effective cone~$\Lambda_{\textup{eff}}(X)$. Hence, we get the following corollary from Proposition~\ref{prop effective cone}.

\begin{cor}\label{cor big divisors}
  Let~$L=\gamma h+\xi f$ with~$\gamma,\xi\in \ZZ$, where~$\{h,f\}$ is the basis of~$\Pic(X_d(a_1,\ldots,a_r))$ given in Proposition~\ref{prop effective cone}. Then, $L$ is big if and only if~$\gamma>0$ and~$\xi>-\gamma a_r$.
\end{cor}

\subsection{Peyre's~$\alpha^\ast$-constant}\label{section Peyre alpha}

Let us recall that Peyre's $\alpha^\ast$-constant of an almost Fano\footnote{Every smooth projective toric variety is almost Fano.} variety~$X$ over a global function field can be defined as
\[
\alpha^\ast(X)=\int_{\Lambda_{\textup{eff}}(X)^\vee} e^{-\left \langle -K_X,\mathbf{y}\right \rangle }\textup{d}\mathbf{y},
\]
where~$\langle \cdot ,\cdot \rangle$ denotes the natural pairing~$\Pic(X)_\RR \times \Pic(X)_\RR^\vee \to \RR$ and $\textup{d}\mathbf{y}$ denotes the Lebesgue measure on~$\Pic(X)_\RR^\vee $ normalized to give covolume 1 to the lattice~$\Pic(X)^\vee$ (see \cite[\emph{Proposition}~3.5]{Pey12PHB}). 
Then, a straightforward computation using Proposition~\ref{prop effective cone} gives us the following result.

 \begin{lem}
 \label{lem:calculusAlpha}
     Let $X=X_d(a_1,\ldots,a_r)$ be a Hirzebruch--Kleinschmidt variety. Then, its $\alpha^\ast$-constant is given by
     \[
     \alpha^\ast(X)=\int_{0}^{\infty}\int_{a_r y_2}^{\infty} e^{-(r+1)y_1-(t-|\mathbf{a}|)y_2}\textup{d}y_1\, \textup{d}y_2 =\frac{1}{(r+1)\left((r+1)a_r+t- |\mathbf{a}|\right) }.
     \]
 \end{lem}

\section{Height zeta function of the projective space}\label{sec height zeta function of projective spaces}

For an integer~$n\geq 1$, we consider in this section the projective space~$\PP^n$ over~$K$ with \emph{standard height function}
\begin{equation}\label{eq standard height projective}
 H_{\PP^n}([x_0:\ldots:x_n]):=\prod_{v\in \Val(K)}\sup\{ |x_0|_v,|x_1|_v,\ldots, |x_n|_v\},   
\end{equation}
and \emph{standard height zeta function}
$$\zeta_{\PP^n}(s):=\sum_{P\in \PP^n(K)}H_{\PP^n}(P)^{-s} \quad \text{for }s\in \CC \text{ with }\Re(s)\gg 0.$$
The goal of this section is to state the following well-known theorem, for which we give a proof for convenience of the reader.

\begin{thm}\label{thm height zeta function projective}
The height zeta function~$\zeta_{\PP^n}(s)$ converges absolutely for~$\Re(s)>n+1$ and it is a rational function in~$q^{-s}$. Moreover, it has a simple pole at~$s=n+1$ with residue
$$\operatorname{Res}_{s=n+1}\zeta_{\PP^n}(s)=\frac{h_Kq^{(n+1)(1-g)}}{\zeta_K(n+1)(q-1)\log(q)},$$
and it is holomorphic in~$\Re(s)>\frac{1}{2},s\not \in \left\{n+1+\frac{2\pi i m}{\log(q)}:m\in \ZZ\right\}$.
\end{thm}

In the proof of Theorem~\ref{thm height zeta function projective} below, we use the following estimate for
$$\tilde{N}(\PP^n,q^d):=\#\{P\in \PP^n(K):H_{\PP^n}(P)=q^d\}$$
due to Wan (\cite[Corollary~4.3]{Wan92}\footnote{There is a misprint in the exponent of the error term in~\cite[Equation~(4.8)]{Wan92}. By~\cite[Theorem~4.1]{Wan92} this error term is bounded above by~$q^{\frac{d}{2}}$ times a polynomial in~$d$.}): For every~$\varepsilon>0$ we have
\begin{equation}\label{eq Wan's estimate}
    \tilde{N}(\PP^n,q^d)=\frac{h_Kq^{(n+1)(1-g)}}{\zeta_K(n+1)(q-1)}q^{d(n+1)}+O\left(q^{d\left(\frac{1}{2}+\varepsilon\right)}\right) \quad \text{as }d\to \infty, d\in \ZZ.
\end{equation}

\begin{proof}[Proof of Theorem~\ref{thm height zeta function projective}]
The fact that~$\zeta_{\PP^n}(s)$ is a rational function in~$q^{-s}$ follows from~\cite[Theorem~3.2]{Wan92}. Now, %for~$m\geq 1$ put~$\tilde{N}(\PP^n,m):=\#\{P\in \PP^n(K):H_{\PP^n}(P)=m\}$. B
by definition of the standard height function~\eqref{eq standard height projective}, we have %~$\tilde{N}(\PP^n,m)=0$ if~$m$ is not of the form~$m=q^d$ with~$d\geq 0$ an integer. Hence
$$\zeta_{\PP^n}(s)=\sum_{d\geq 0}\tilde{N}(\PP^n,q^d)q^{-ds}.$$
%Now, by~\cite[Corollary~4.3]{Wan92} we have, for every~$\varepsilon>0$, the estimate
%\begin{equation}
% \tilde{N}(\PP^n,q^d)=\frac{h_Kq^{(n+1)(1-g)}}{\zeta_K(n+1)(q-1)}q^{d(n+1)}+O(q^{\frac{d}{2}+\varepsilon}) \quad \text{as }d\to \infty. 
%\end{equation}
Then, \eqref{eq Wan's estimate} implies that~$\zeta_{\PP^n}(s)$ converges absolutely for~$\Re(s)>n+1$, and moreover
$$\zeta_{\PP^n}(s)=\frac{h_Kq^{(n+1)(1-g)}}{\zeta_K(n+1)(q-1)}\left(\frac{1}{1-q^{(n+1)-s}}\right)+F(s),$$
with~$F(s)$ an analytic function for~$\Re(s)>\frac{1}{2}$. In particular, $\zeta_{\PP^n}(s)$ is analytic in~$\Re(s)>\frac{1}{2},s\not \in \left\{n+1+\frac{2\pi i m}{\log(q)}:m\in \ZZ\right\}$. 
 We also get
$$\operatorname{Res}_{s=n+1}\zeta_{\PP^n}(s)=\frac{h_Kq^{(n+1)(1-g)}}{\zeta_K(n+1)(q-1)}\left(\lim_{s\to n+1}\frac{s-(n+1)}{1-q^{(n+1)-s}}\right)=\frac{h_Kq^{(n+1)(1-g)}}{\zeta_K(n+1)(q-1)\log(q)}.$$
This shows the desired results and completes the proof of the theorem.
\end{proof}

 \section{Height zeta functions of Hirzebruch--Kleinschmidt varieties}\label{sec height zeta HK}

In this section, we use the explicit equations~\eqref{equations HK} defining our Hirzebruch--Kleinschmidt varieties in order to study the analytic properties of their height zeta functions. To do this, we first present some lemmas concerning Möbius inversion, supremum of divisors and relations between different zeta functions. We will first focus on the case~$a_r>0$ in the family of varieties~$X=X_d(a_1,\ldots, a_r)$, and later treat the much easier case where~$a_r=0$, since then~$X\simeq \PP^{r}\times \PP^{t-1}$ and the  relevant height zeta functions are just products of height zeta functions of projective spaces.

In the notation of Proposition \ref{prop effective cone}, let $L=\gamma h+\xi f$ be the class of a big line bundle. Then, the height function induced by $L$ on the set of rational points of $X_d(a_1,\ldots,a_r)$  is 
\begin{equation}\label{eq height function HL}
 H_L\left( ([x_0: (x_{ij})_{i\in I_t,j\in I_r}],[y_1:\ldots:y_t])\right):=\prod_{v\in \Val(K)}\sup_{i\in I_{t}, j\in I_r}\left\lbrace|x_0|_v,|x_{ij}|_v  \right\rbrace ^{\gamma}\prod_{v\in \Val(K)}\sup_{i\in I_t} \left\lbrace|y_i|_v  \right\rbrace^{\xi }.   
\end{equation}
It is easy to see that this height function is associated to an adelic metrization of~$L$ in the sense of~\cite[Secion~1.2]{Pey12PHB}. We refer to this as the ``standard metrization'' of~$L$.

As explained in the introduction, to carry out the analysis of the corresponding height zeta functions, we find it necessary to remove a closed subset in order to avoid convergence issues, due to accumulation of rational points of bounded height. %accumulation of rational points of bounded height. %Geometrically, this translates to removing the subvariety with negative self-intersection and thus having too many rational points. 
%In this case, this closed set turned out to be another Hirzebruch--Kleinschmidt variety with different parameters.
For this purpose, the following definition was given in the introduction.

\begin{defi}\label{defi:good open} 
Given integers~$r\geq 1$, $t\geq 2$ and $0\leq a_1 \leq \cdots \leq a_r$ defining the Hirzebruch--Kleinschmidt variety~$X_d(a_1,\ldots,a_r)$ as in Theorem~\ref{thm:EquationsHirzebruchKleinschmidtVar}, we define the \emph{good open subset} $U_d(a_1,\ldots,a_r)\subset X_d(a_1,\ldots,a_d)$ as the complement of the closed subvariety defined by the equation~$x_{tr}=0$. 
\end{defi}

\subsection{Height zeta functions}\label{subsection general height zeta functions}

Let~$L=\gamma h+\xi f\in \Pic(X_d(a_1,\ldots,a_r))$ big. The \emph{height zeta function} associated to~$H_L$ over the good open subset~$U:=U_d(a_1,\ldots,a_r)$ is defined as
	$$\zeta_{U,L}(s):=\sum_{x\in U(K)} H_L(x)^{-s}.$$

Associated to~$L$ as above, we also define
     \begin{equation}\label{eq def of AL and BL}
        A_L:=\frac{r+1}{\gamma}, \quad B_L:=\frac{(r+1)a_r-|\mathbf{a}|  +t}{ \gamma a_r+\xi},
     \end{equation}
and put
\begin{eqnarray*}
    \mathcal{A}_L&:=& \left\{A_L+\frac{2\pi i k}{\gamma \log(q)}:k\in\{1,\ldots,\gamma-1\}\right\},\\
    \mathcal{B}_L&:=& \left\{B_L+\frac{2\pi i k}{(\gamma a_r+\xi) \log(q)}:k\in\{1,\ldots,\gamma a_r+\xi-1\}\right\}.
\end{eqnarray*}
Note that by Proposition~\ref{prop effective cone} we have
\begin{equation}\label{eq a(L) and b(L)}
     a(L)=\max\{A_L,B_L\}\quad \text{and}\quad b(L)=\left\{\begin{array}{ll}
2     &  \text{if }A_L=B_L, \\
1     &  \text{if }A_L\neq B_L.
\end{array}
\right.
\end{equation}

Below we present the first main result of this section, which describes the analytic properties of the function $\zeta_{U,L}(s)$ in the case where~$a_r>0$. First, we note that if~$L\in \Pic(X_d(a_1,\ldots,a_r))$ is big and~$m\geq 1$ is an integer, then~$mL$ is also big and~$H_{mL}=H_L^m$. Moreover, we have the equality of height zeta functions~$\zeta_{U,mL}(s)=\zeta_{U,L}(ms)$. In particular, the analytic properties of~$\zeta_{U,mL}(s)$ can be deduced from those of~$\zeta_{U,L}(s)$. This shows that we can restrict ourselves to the case where~$L$ is primitive in~$\Pic(X_d(a_1,\ldots,a_r))$. Writing~$L=\gamma h+\xi f$, this is equivalent to the condition~$\mathrm{g.c.d.}(\gamma,\xi)=1$. 

In order to state the next theorem, we define
$$\R_K(a,b):= \sum_{D\geq 0}q^{-(a\ell(D)+b\deg(D))} \quad \text{for }a,b\in \CC \text{ with }\Re(a+b)>1.$$ 
Note that the Riemann--Roch theorem (see Section~\ref{subsection arithmetic tools}) implies the equality
\begin{equation}\label{eq identity Rk}
    \R_K(a,b)= \sum_{\substack{D\geq 0 \\ \deg(D)\leq 2g-2}}q^{-(a\ell(D)+b\deg(D))}+q^{a(g-1)}\sum_{\substack{D\geq 0\\ \deg(D)> 2g-2}}q^{-(a+b)\deg(D)}.
\end{equation}
Hence, $\R_K(a,b)=q^{a(g-1)}\zeta_K(a+b)+\bS_K(a,b)$ with the  finite sum
$$\bS_K(a,b):=\sum_{\substack{D\geq 0\\ \deg(D)\leq 2g-2}}(q^{-(a\ell(D)+b\deg(D))}-q^{a(g-1)-(a+b)\deg(D)}).$$
In particular, $\R_K(a,b)=q^{-a}\zeta_K(a+b)$ if~$g=0$.

% and later present the easier case when~$a_r=0$. 
 
	\begin{thm}\label{thm:ThmGeneralversion}
	    Let~$X:=X_d(a_1,\ldots,a_r)$ be a Hirzebruch--Kleinschmidt variety over the global function field~$K=\F_q(\mathcal{C})$ with~$a_r>0$. Moreover, let~$U:=U_d(a_1,\ldots,a_r)$ be the  good open subset of~$X$, let~$L=\gamma h+\xi f\in \Pic(X)$ be a big and primitive line bundle class and let~$\zeta_{U,L}(s)$ be the associated  height zeta function. Then,  $\zeta_{U,L}(s)$ converges absolutely for~$\Re(s)>a(L)$ and is a rational function in~$q^{-s}$. Moreover, the following properties hold:
     \begin{enumerate}
         \item If~$A_L=B_L$, then $\zeta_{U,L}(s)$ is holomorphic for
    $$\Re(s)>\max\left\{A_L-\frac{1}{\gamma},B_L-\frac{1}{\gamma a_r+\xi}\right\},s\not \in \left\{w+\frac{2\pi im}{\log(q)}:w\in  \{a(L)\}\cup \mathcal{A}_L\cup\mathcal{B}_L, m\in \ZZ\right\},$$
    it  has a pole of order~$b(L)=2$ at~$s=a(L)$ and satisfies
    $$\lim_{s\to w}(s-w)^{2}\zeta_{U,L}(s)=\left\{
    \begin{array}{ll}
      \frac{q^{(d+2)(1-g)} h_K^2}{\zeta_K(t)\zeta_K(r+1)(\gamma a_r+\xi)\gamma (q-1)^2\log(q)^2}   & \text{if }w=a(L),  \\
        0 & \text{if }w\in  \mathcal{A}_L\cup \mathcal{B}_L.
    \end{array}\right.$$
       %it is analytic in~$\Re(s)=a(L),\Im(s)\in \left[0,\frac{2\pi}{\log(q)}\right[,s\not \in \{a(L)\}\cup \mathcal{A}_L\cup \mathcal{B}_L$ and at every~$w\in \mathcal{A}_L\cup \mathcal{B}_L$ we have~$\lim_{s\to w}(s-w)^2\zeta_{U,L}(s)=0$.
       \item If~$A_L<B_L$, then $\zeta_{U,L}(s)$ is holomorphic for
    $$\Re(s)>\max\left\{A_L,B_L-\frac{1}{\gamma a_r+\xi}\right\},s\not \in \left\{w+\frac{2\pi im}{\log(q)}:w\in  \{a(L)\}\cup \mathcal{B}_L, m\in \ZZ\right\},$$
       it  has a pole of order~$b(L)=1$ at~$s=a(L)$ and for every~$w\in \{a(L)\}\cup \mathcal{B}_L$ it satisfies
       $$\lim_{s\to w}(s-w)\zeta_{X,L}(s)= \frac{q^{(d+2-N_X)(1-g)}h_K\R_K(1-N_X,\gamma w-r+N_X-1)}{\zeta_K\left(t  \right)\zeta_K(\gamma w)(\gamma a_r+\xi)(q-1)\log(q)},$$
       where~$N_X:=\#\{j\in \{1,\ldots,r\}:a_j=a_r\}$. %, and is analytic in~$\Re(s)=a(L),\Im(s)\in \left[0,\frac{2\pi}{\log(q)}\right[,s\not \in  \{a(L)\}\cup \mathcal{B}_L $.
       
       \item If~$A_L>B_L$, then $\zeta_{U,L}(s)$ is holomorphic for
    $$\Re(s)>\max\left\{A_L-\frac{1}{\gamma},B_L\right\},s\not \in \left\{w+\frac{2\pi im}{\log(q)}:w\in  \{a(L)\}\cup \mathcal{A}_L, m\in \ZZ\right\},$$       
       it has a pole of order~$b(L)=1$ at~$s=a(L)$ and for every~$w\in \{a(L)\}\cup \mathcal{A}_L$ it satisfies
       $$\lim_{s\to w}(s-w)\zeta_{U,L}(s)= \frac{ q^{(r+1)(1-g)}h_K\R_K\left(1-t, ( \gamma a_r+\xi)w- ((r+1)a_r-|\mathbf{a}|) \right)}{\zeta_K\left(\left( \gamma a_r+\xi \right)w-\left( (r+1)a_r-|\mathbf{a}|\right)\right)\zeta_K(r+1)\gamma (q-1)\log(q)}.$$
       %for every~$w\in \{a(L)\}\cup \mathcal{A}_L$ and is analytic in~$\Re(s)=a(L),\Im(s)\in \left[0,\frac{2\pi}{\log(q)}\right[,s\not \in  \{a(L)\}\cup \mathcal{A}_L$.
     \end{enumerate}
     %$\zeta_{U,L}(s)$ converges absolutely for~$\Re(s)>a(L)$ and it has a pole of order~$b(L)$ at~$s=a(L)$ with
     %\begin{equation*}
     %    \lim_{s\to a(L)}(s-a(L))^{b(L)}\zeta_{U,L}(s)=
     %    \left\{
     %    \begin{array}{ll}
     %    \vspace{0.2cm}   \frac{q^{(d+2)(1-g)} h_K^2}{\zeta_K(t)\zeta_K(r+1)(\gamma a_r+\xi)\gamma (q-1)^2\log(q)^2}  & \text{if }  \xi=\left(\frac{t-|\mathbf{a}|}{r+1}\right)\gamma,\\
     %    \vspace{0.2cm}      \frac{q^{(d+2-N_X)(1-g)}h_K\R_K(1-N_X,\gamma B_L-r+N_X-1)}{\zeta_K\left(t  \right)\zeta_K(\gamma B_L)(\gamma a_r+\xi)(q-1)\log(q)}  & \text{if }  \xi<\left(\frac{t-|\mathbf{a}|}{r+1}\right)\gamma,\\
      %    \vspace{0.2cm}    \frac{ q^{(r+1)(1-g)}h_K\R_K\left(1-t, A_L\xi  +|\mathbf{a}| \right)}{\zeta_K\left( A_L\xi +|\mathbf{a}|\right)\zeta_K(r+1)\gamma (q-1)\log(q)} & \text{if }  \xi>\left(\frac{t-|\mathbf{a}|}{r+1}\right)\gamma,\\
       %  \end{array}
       %  \right.
     %\end{equation*}    
	\end{thm}

The proof of Theorem~\ref{thm:ThmGeneralversion} is given in Section~\ref{sec proof of general thm}, after establishing some preliminary lemmas in Section~\ref{sec preliminary lemmas}.

\begin{remark}\label{rmk after main general thm}
%\begin{enumerate}
%    \item The Riemann--Roch theorem (see Section~\ref{subsection arithmetic tools}) implies that
%\begin{equation}\label{eq identity Rk}
%    \R_K(a,b)= \sum_{\substack{D\geq 0 \\ \deg(D)\leq 2g-2}}q^{-(a\ell(D)+b\deg(D))}+q^{a(g-1)}\sum_{\substack{D\geq 0\\ \deg(D)> 2g-2}}q^{-(a+b)\deg(D)},
%\end{equation}
%hence~$\R_K(a,b)=q^{a(g-1)}\zeta_K(a+b)+\bS_K(a,b)$ with the  finite sum
%$$\bS_K(a,b):=\sum_{\substack{D\geq 0\\ \deg(D)\leq 2g-2}}(q^{-(a\ell(D)+b\deg(D))}-q^{a(g-1)-(a+b)\deg(D)}).$$
%In particular, $\R_K(a,b)=q^{-a}\zeta_K(a+b)$ if~$g=0$.
%\item 
As in the case of Hirzebruch--Kleinschmidt varieties over number fields (see~\cite[Remark~6.8]{HMM24a}), the different cases that appear in Theorem~\ref{thm:ThmGeneralversion} give a subdivision of the \emph{big cone} of $X$, i.e.~the interior of~$\Lambda_{\textup{eff}}(X)$ (see Figure~\ref{fig:subdivisionofBigCone} for an illustration, where we assume~$t>|\mathbf{a}|$ for simplicity). Indeed, the line bundles~$L$ contained in the ray passing through the anticanonical class have height zeta functions with a double pole at $s=A_L=B_L$, while line bundles outside this ray have height zeta functions with a simple pole at $s=\max\{A_L,B_L\}$.
%\end{enumerate}
\end{remark}

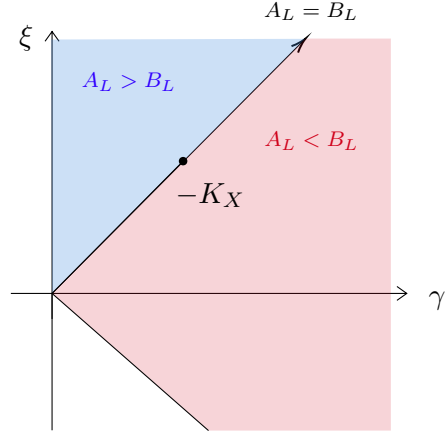
\begin{figure}[h]\label{fig:subdivisionofBigCone}
\centering
\begin{tikzpicture}[x=0.75pt,y=0.75pt,yscale=-1,xscale=1,scale=0.7]
%uncomment if require: \path (0,333); %set diagram left start at 0, and has height of 333

%Shape: Axis 2D [id:dp8167954265933379] 
\draw  (181.17,221.1) -- (463.5,221.1)(210.56,32.92) -- (210.56,239.45) (456.5,216.1) -- (463.5,221.1) -- (456.5,226.1) (205.56,39.92) -- (210.56,32.92) -- (215.56,39.92)  ;
%Straight Lines [id:da8126079591399067] 
\draw    (210.56,221.1) -- (389.93,39.37) ;
\draw [shift={(391.33,37.95)}, rotate = 134.63] [color={rgb, 255:red, 0; green, 0; blue, 0 }  ][line width=0.75]    (10.93,-3.29) .. controls (6.95,-1.4) and (3.31,-0.3) .. (0,0) .. controls (3.31,0.3) and (6.95,1.4) .. (10.93,3.29)   ;

%Shape: Polygone [id:dp6077691191667765] red
\draw  [draw opacity=0][fill={rgb, 255:red, 208; green, 2; blue, 27 }  ,fill opacity=0.17 ] (391.33,37.95) -- (210.56,221.1) -- (322.13,319.67) -- (452.17,319.67) -- (452.17,37.95) -- cycle ;
%old \draw  [draw opacity=0][fill={rgb, 255:red, 208; green, 2; blue, 27 }  ,fill opacity=0.17 ] (391.33,37.95) -- (210.56,221.1) -- (391.33,221.1) -- cycle ;

%Shape: Right Triangle [id:dp4195768251947244] blue
\draw  [color={rgb, 255:red, 0; green, 0; blue, 0 }  ,draw opacity=0 ][fill={rgb, 255:red, 74; green, 144; blue, 226 }  ,fill opacity=0.28 ] (210.55,221.09) -- (390.91,38.15) -- (209.88,38.63) -- cycle ;

%Shape: Rectangle [id:dp724005052188271] red
%\draw  [draw opacity=0][fill={rgb, 255:red, 208; green, 2; blue, 27 }  ,fill opacity=0.17 ] (391.33,37.95) -- (452.17,37.95) -- (452.17,221.1) -- (391.33,221.1) -- cycle ;

%Straight Lines [id:da23941727453928063] 
\draw    (304,126.05) -- (210.56,221.1) ;
\draw [shift={(304,126.05)}, rotate = 134.51] [color={rgb, 255:red, 0; green, 0; blue, 0 }  ][fill={rgb, 255:red, 0; green, 0; blue, 0 }  ][line width=0.75]      (0, 0) circle [x radius= 2.34, y radius= 2.34]   ;
%Straight Lines [id:da577711654277957] 
\draw    (210.56,221.1) -- (210.56,319.44) ;

%Shape: Rectangle [id:dp48482203214827724] red
%\draw  [draw opacity=0][fill={rgb, 255:red, 208; green, 2; blue, 27 }  ,fill opacity=0.17 ] (210.56,221.1) -- (452.17,221.1) -- (452.17,319.44) -- (210.56,319.44) -- cycle ;
%Shape: Right Triangle [id:dp9645899389235556] blue
%\draw  [draw opacity=0][fill={rgb, 255:red, 255; green, 255; blue, 255 }  ,fill opacity=1 ] (211.13,222.11) -- (322.13,319.67) -- (211.13,319.67) -- cycle ;

%Straight Lines [id:da21493525627756194] 
\draw    (210.56,221.1) -- (322.13,319.67) ;

% Text Node
\draw (185.41,26.09) node [anchor=north west][inner sep=0.75pt]    {$\xi $};
% Text Node
\draw (477.49,218.08) node [anchor=north west][inner sep=0.75pt]    {$\gamma $};
% Text Node
\draw (296.95,139.97) node [anchor=north west][inner sep=0.75pt]  [color={rgb, 255:red, 0; green, 0; blue, 0 }  ,opacity=1 ]  {$-K_{X}$};
% Text Node
\draw (360,10) node [anchor=north west][inner sep=0.75pt]  [font=\scriptsize]  {$A_L=B_L$};
% Text Node
\draw (360.33,103.03) node [anchor=north west][inner sep=0.75pt]  [font=\scriptsize,color={rgb, 255:red, 208; green, 2; blue, 27 }  ,opacity=1 ]  {$A_L<B_L$};
% Text Node
\draw (230,60.65) node [anchor=north west][inner sep=0.75pt]  [font=\scriptsize,color={rgb, 255:red, 49; green, 0; blue, 255 }  ,opacity=1 ]  {$A_L>B_L$};
\end{tikzpicture}
 \caption{Subdivision of the big cone of $X_d(a_1,\ldots,a_r)$ according to the values of~$A_L=\frac{r+1}{\gamma}$ and~$B_L=\frac{(r+1)a_r-|\mathbf{a}|+t}{\gamma a_r+\xi}$ (assuming~$t>|\mathbf{a}|$ for simplicity).}
 \end{figure}

The much simpler case when~$a_r=0$ is given by the following theorem, where we see that there is no need to remove a closed subvariety of~$X=X_d(a_1,\ldots,a_r)$ and we can directly give the analytic properties of the height zeta function
$$\zeta_{X,L}(s):=\sum_{P\in X(K)}H_L(P)^{-s}.$$
Note that, in this case, we have
$$A_L=\frac{r+1}{\gamma}, \quad B_L=\frac{t}{\xi}.$$

 \begin{thm}\label{thm:ThmGeneral when ar=0}
  Let~$X\simeq \PP^r\times \PP^{t-1}$ be a Hirzebruch--Kleinschmidt variety over the global function field~$K=\F_q(\mathcal{C})$ with~$a_r=0$ and let~$L=\gamma h+\xi f\in \Pic(X)$ big and primitive. Then, $\zeta_{X,L}(s)$ converges absolutely for~$\Re(s)>a(L)$ and is a rational function in~$q^{-s}$. Moreover, the following properties hold:
  \begin{enumerate}
      \item If~$A_L=B_L$, then $\zeta_{X,L}(s)$ is holomorphic for
    $$\Re(s)>\max\left\{\frac{1}{2\gamma},\frac{1}{2\xi}\right\},s\not \in \left\{w+\frac{2\pi im}{\log(q)}:w\in  \{a(L)\}\cup \mathcal{A}_L\cup\mathcal{B}_L, m\in \ZZ\right\},$$
    it  has a pole of order~$b(L)=2$ at~$s=a(L)$ and satisfies
    $$\lim_{s\to w}(s-w)^{2}\zeta_{X,L}(s)=\left\{
    \begin{array}{ll}
      \frac{q^{(d+2)(1-g)} h_K^2}{\zeta_K(t)\zeta_K(r+1)\xi\gamma (q-1)^2\log(q)^2}   & \text{if }w=a(L),  \\
        0 & \text{if }w\in  \mathcal{A}_L\cup \mathcal{B}_L.
    \end{array}\right.$$
    % has a pole of order~$b(L)=2$ at~$s=a(L)$ with
    %$$\lim_{s\to a(L)}(s-a(L))^{b(L)}\zeta_{X,L}(s)=\frac{q^{(d+2)(1-g)} h_K^2}{\zeta_K(t)\zeta_K(r+1)\gamma\xi (q-1)^2\log(q)^2},$$
     %  it has simple poles at every~$s\in \mathcal{A}_L\cup \mathcal{B}_L$  and is analytic in~$\Re(s)=a(L),\Im(s)\in \left[0,\frac{2\pi}{\log(q)}\right[,s\not \in \{a(L)\}\cup \mathcal{A}_L\cup \mathcal{B}_L$.
       \item If~$A_L<B_L$, then $\zeta_{X,L}(s)$ is holomorphic for
    $$\Re(s)>\max\left\{A_L,\frac{1}{2\xi}\right\},s\not \in \left\{w+\frac{2\pi im}{\log(q)}:w\in  \{a(L)\}\cup \mathcal{B}_L, m\in \ZZ\right\},$$
       it  has a pole of order~$b(L)=1$ at~$s=a(L)$ and for every~$w\in \{a(L)\}\cup \mathcal{B}_L$ it satisfies
       %has a pole of order~$b(L)=1$ at every~$s\in \{a(L)\}\cup \mathcal{B}_L $ with
       $$\lim_{s\to w}(s-w)\zeta_{X,L}(s)=\frac{\zeta_{\PP^r}(\gamma w)  h_Kq^{t(1-g)}}{ \zeta_K(t)\xi(q-1)\log(q)}.$$
       %for every~$w\in \{a(L)\}\cup \mathcal{B}_L$ and is analytic in~$\Re(s)=a(L),\Im(s)\in \left[0,\frac{2\pi}{\log(q)}\right[,s\not \in  \{a(L)\}\cup \mathcal{B}_L $. 
       \item If~$A_L>B_L$, then $\zeta_{X,L}(s)$ is holomorphic for
    $$\Re(s)>\max\left\{\frac{1}{2\gamma},B_L\right\},s\not \in \left\{w+\frac{2\pi im}{\log(q)}:w\in  \{a(L)\}\cup \mathcal{A}_L, m\in \ZZ\right\},$$       
       it has a pole of order~$b(L)=1$ at~$s=a(L)$ and for every~$w\in \{a(L)\}\cup \mathcal{A}_L$ it satisfies
       %has a pole of order~$b(L)=1$ at every~$s\in \{a(L)\}\cup \mathcal{A}_L $ with
       $$\lim_{s\to w}(s-w)\zeta_{X,L}(s)=\frac{\zeta_{\PP^{t-1}}(\xi w)  h_Kq^{(r+1)(1-g)}}{\gamma\zeta_K(r+1)(q-1)\log(q)}.$$
       %for every~$w\in \{a(L)\}\cup \mathcal{A}_L$ and is analytic in~$\Re(s)=a(L),\Im(s)\in \left[0,\frac{2\pi}{\log(q)}\right[,s\not \in  \{a(L)\}\cup \mathcal{A}_L$.
  \end{enumerate}
 % putting
 %  has a pole of order~$b(L)$ at~$s=a(L)$ with
 %    \begin{equation*}
 %        \lim_{s\to a(L)}(s-a(L))^{b(L)}\zeta_{X,L}(s)=
 %        \left\{
 %    \begin{array}{ll}
 %     \frac{q^{(d+2)(1-g)} h_K^2}{\zeta_K(t)\zeta_K(r+1)\gamma\xi (q-1)^2\log(q)^2}    & \text{if }  \xi=\left(\frac{t}{r+1}\right)\gamma, \vspace{0.1cm}\\
  %     \frac{\zeta_{\PP^r}(\gamma B_L)  h_Kq^{t(1-g)}}{ \zeta_K(t)\xi(q-1)\log(q)}    &  \text{if }  \xi<\left(\frac{t}{r+1}\right)\gamma, \vspace{0.1cm}\\
   %   \frac{\zeta_{\PP^{t-1}}(\xi A_L)  h_Kq^{(r+1)(1-g)}}{\gamma\zeta_K(r+1)(q-1)\log(q)}    &  \text{if }  \xi>\left(\frac{t}{r+1}\right)\gamma. \vspace{0.1cm}
  %   \end{array}
  %   \right.
  %   \end{equation*}
 \end{thm}
\begin{proof}
    By~\eqref{eq height function HL} and~$\eqref{eq standard height projective}$ we have
    $$\zeta_{X,L}(s)=\zeta_{\PP^r}(\gamma s)\zeta_{\PP^{t-1}}(\xi s).$$
    Then, using Theorem~\ref{thm height zeta function projective}, we see that~$\zeta_{X,L}(s)$ is a rational function in~$q^{-s}$ and it converges absolutely for~$\Re(s)>\max\{A_L,B_L\}=a(L)$. Moreover, we have the following properties:
    \begin{enumerate}
        \item If~$A_L=B_L$, then~$\zeta_{X,L}(s)$ %is analytic in~$\Re(s)>\max\left\{\frac{r}{\gamma},\frac{t-1}{\xi}\right\}$ and it 
        has a double pole at~$s=a(L)$ with
        \begin{equation*}
            \begin{split}
        \lim_{s\to a(L)}(s-a(L))^2\zeta_{X,L}(s)&=\frac{\operatorname{Res}_{s=r+1}\zeta_{\PP^r}(s) \operatorname{Res}_{s=t}\zeta_{\PP^{t-1}}(s)}{\gamma \xi}\\
        &=\frac{h_K^2q^{(d+2)(1-g)}}{\gamma \xi \zeta_K(r+1)\zeta_K(t)(q-1)^2\log(q)^2}.
            \end{split}
        \end{equation*}
         Moreover,  $\zeta_{U,L}(s)$ is holomorphic for
    $$\Re(s)>\max\left\{\frac{1}{2\gamma},\frac{1}{2\xi}\right\},s\not \in \left\{w+\frac{2\pi im}{\log(q)}:w\in  \{a(L)\}\cup \mathcal{A}_L\cup\mathcal{B}_L, m\in \ZZ\right\},$$
    and satisfies~$\lim_{s\to w}(s-w)^2\zeta_{U,L}(s)=0$ for every~$w\in \mathcal{A}_L\cup\mathcal{B}_L$, since~$\mathcal{A}_L\cap\mathcal{B}_L=\emptyset$.
        \item If~$A_L<B_L$, then $\zeta_{X,L}(s)$ is holomorphic for
    $$\Re(s)>\max\left\{A_L,\frac{1}{2\xi}\right\},s\not \in \left\{w+\frac{2\pi im}{\log(q)}:w\in  \{a(L)\}\cup \mathcal{B}_L, m\in \ZZ\right\},$$
and for every~$w\in \{a(L)\}\cup \mathcal{B}_L$ it satisfies
        %then~$\zeta_{X,L}(s)$ %is analytic in~$\Re(s)>\max\left\{A_L,\frac{t-1}{\xi}\right\}$ and it 
        %has a simple pole at~$s=B_L=a(L)$ with
        \begin{equation*}
            \begin{split}
        \lim_{s\to w}(s-w)\zeta_{X,L}(s)&=\frac{\zeta_{\PP^r}(\gamma w) \operatorname{Res}_{s=t}\zeta_{\PP^{t-1}}(s)}{\xi}\\
        &=\frac{\zeta_{\PP^r}(\gamma w)  h_Kq^{t(1-g)}}{\xi \zeta_K(t)(q-1)\log(q)}.
            \end{split}
        \end{equation*}
        In particular, since~$\zeta_{\PP^r}(\gamma B_L)$ is positive, $\zeta_{X,L}(s)$ has a pole of order~$b(L)=1$ at~$s=a(L)$.
        \item If~$A_L>B_L$, then $\zeta_{X,L}(s)$ is holomorphic for
    $$\Re(s)>\max\left\{\frac{1}{2\gamma},B_L\right\},s\not \in \left\{w+\frac{2\pi im}{\log(q)}:w\in  \{a(L)\}\cup \mathcal{A}_L, m\in \ZZ\right\},$$
and for every~$w\in \{a(L)\}\cup \mathcal{A}_L$ it satisfies
        \begin{equation*}
            \begin{split}
        \lim_{s\to w}(s-w)\zeta_{X,L}(s)&=\frac{\operatorname{Res}_{s=r+1}\zeta_{\PP^{r}}(s) \zeta_{\PP^{t-1}}(\xi w)}{\gamma}\\
        &=\frac{\zeta_{\PP^{t-1}}(\xi w)  h_Kq^{(r+1)(1-g)}}{\gamma \zeta_K(r+1)(q-1)\log(q)}.
            \end{split}
        \end{equation*}
        As in the previous case, using that~$\zeta_{\PP^{t-1}}(\xi A_L)$ is positive, we conclude that~$\zeta_{X,L}(s)$ has a pole of order~$b(L)=1$ at~$s=a(L)$.
    \end{enumerate}
    %These cases correspond exactly to~$\xi$ equal, less than, or greater than~$\left(\frac{t}{r+1}\right)\gamma$, respectively. %Finally, the asymptotic formula for~$N(X,H_L,B)$ follows from Theorem~\ref{thm:tauberianThm}. 
    This proves the theorem.
\end{proof}

\subsection{Preliminary lemmas}\label{sec preliminary lemmas}

In this section we establish some preliminary lemmas needed for the proof of Theorem~\ref{thm:ThmGeneralversion} in the next section.
 
\begin{defi}\label{def convolution}
    Given functions $f,g: \Div^+(K)\rightarrow \CC$, we define their convolution product as
	$$(f\star g)(D):=\sum_{0\leq D'\leq D}f(D')g(D-D').$$
	  Moreover, we define the functions 	$1, u, \mu:\Div^+(K)\rightarrow \CC$ by 
	\begin{itemize}\setlength\itemsep{2mm}

  \item $1(D):=1$ for every $D\in \Div^+(K)$.
		\item $u(D):=\left\lbrace \begin{array}{ll}
			1& \textup{if } D=0,\\
			0	&  \textup{otherwise}.
		\end{array}\right. $
		\item $\mu(D):=\left\lbrace \begin{array}{ll}
			%1& \textup{if }D=0, \\
			(-1)^{\sum_v v(D)}	& \textup{if~$v(D)=0$ or $1$ for all $v\in \operatorname{Val}(K)$},  \\
			0	& \textup{otherwise}.
		\end{array}\right.$
	\end{itemize}
 	In particular, $u$ is a unit for the convolution product and $1\star \mu =\mu\star 1=u$.
\end{defi}

The following result, based on a straightforward computation, is a M\"{o}bius inversion formula in this context.

	\begin{lem}
		\label{lem:muCouple}
		Let $f,g:\Div^+(K)\to \CC$ be two functions. Then, the relation
		$$f(D)=\sum_{0\leq D'\leq D}g(D') \quad \text{for all }D\in \Div^+(K)$$
		is equivalent to 
		$$g(D)=\sum_{0\leq D'\leq D}\mu(D-D')f(D')\quad \text{for all }D\in \Div^+(K).$$
	\end{lem}

	Lemma~\ref{lem:muCouple} motivates the following definition that will be useful for counting purposes.
 
	\begin{defi}
		Let $f,g:\Div^+(K)\to \CC$ be functions. We say that the pair $(f,g)$ forms a \emph{$\mu$-couple} if they satisfy any of the two equivalent relations in Lemma \ref{lem:muCouple}.
	\end{defi}
	 By definition of the function $u$, the pair $(1, u)$ forms a $\mu$-couple. %  and thus we have the following result.
  
	\begin{lem}\label{lem inversion Z_K(T)}
		We have the following equality of formal series:
		$$\sum_{D\geq 0}\mu(D)T^{\deg (D)}=\frac{1}{\Zet_K(T)}.$$
	\end{lem}

\begin{proof}
		Since $(1,u)$ is a $\mu$-couple, we have that
		\begin{equation*}
			\begin{split}
				\sum_{D'\geq 0}\mu(D')T^{\deg(D')}\Zet_K(T)&=\sum_{D'\geq 0}\mu(D')T^{\deg(D')}\sum_{D\geq 0}T^{\deg(D)}\\
				&=\sum_{D\geq 0}\sum_{D'\geq 0}\mu(D')T^{\deg(D+D')}\\
				%&=\sum_{D\geq 0}\sum_{\substack{D'\geq 0, D'' \geq0\\ D'+D''=D}}\mu(D')T^{\deg(D)}\\
				&=\sum_{D\geq0} \sum_{0\leq  D'\leq D} \mu(D-D')T^{\deg(D)}\\
				&=\sum_{D\geq 0} u(D)T^{\deg(D)}=1,
			\end{split}
		\end{equation*}
  where the last equality follows from the definition of the function $u$. This proves the lemma.
	\end{proof}

Putting~$T=q^{-s}$ in the above lemma, we get the following corollary.

 \begin{cor}\label{cor inverse of zetaK converges absolutely}
     For every~$s\in \CC$ with~$\Re(s)>1$, we have
     $$\sum_{D\geq 0}\mu(D)q^{-s\deg(D)}=\frac{1}{\zeta_K(s)},$$
     and the series on the left hand side converges absolutely and uniformly on compact subsets of the half-plane~$\Re(s)>1$.
 \end{cor}
 
 For computations involving supremum of divisors, we will make use of the following result.
 	\begin{lem}
		\label{lem:sepDivisor}
		For all~$x,y\in K$ we have 
		$$\sup\{(x)_\infty, (xy)_\infty, (y)_\infty\}=(x)_\infty+(y)_\infty.$$
	\end{lem}
	\begin{proof}
		Let $v\in \Val(K)$. We recall that for a divisor $D=\sum_{v\in \Val(K)}n_vv\in \Div(K)$, we defined $v(D)=n_v$. We have to show that
  $$\sup\{v((x)_\infty), v((xy)_\infty), v((y)_\infty)\}=v((x)_\infty)+v((y)_\infty),$$
  for all~$v\in \Val(K)$.   We have the following four cases to consider:
		\begin{enumerate}
			\item If $v(x),v(y)\leq 0$ then~$v\left((xy)_\infty \right)=v\left((x)_\infty \right)  + v\left((y)_\infty \right)$.
			\item If $v(x) \leq 0$ and $v(y)>0$ then $v\left((xy)_\infty \right)\leq v\left((x)_\infty \right)$ and $v\left( (y)_\infty\right)=0$.
			\item  If $v(x) > 0$ and $v(y)\leq 0$ then $v\left((xy)_\infty \right)\leq v\left((y)_\infty \right)$ and $v\left( (x)_\infty\right)=0$.
			\item If $v(x), v(y)>0$ then $v\left((xy)_\infty \right)=v\left((x)_\infty \right)=v\left((y)_\infty \right)=0$.
		\end{enumerate}
 The desired results follows directly from these computations. This proves the lemma.
	\end{proof}

%%%%%%%%%%%%%%

\begin{lem}\label{lem:explicit_U}
Assume~$a_r>0$. Then, the good open subset~$U_d(a_1,\ldots,a_r)$ equals the image of the map
\begin{equation}\label{eq lem explicit U}
    (x_1,\ldots,x_d)\in \A^d \mapsto ([x_d:(\tilde{x}_{ij})_{i\in I_t,j\in I_r}],[x_1:\ldots:x_{t-1}:1])\in X_d(a_1,\ldots,a_r),
\end{equation}
where
\begin{equation*}
   \tilde{x}_{ij}=\left\{
\begin{array}{ll}
x_{t+j-1}x_i^{a_j}     & \text{if }i<t,j<r,  \\
x_{t+j-1}     & \text{if }i=t,j<r, \\
x_i^{a_r}     & \text{if }i<t,j=r,  \\
1    & \text{if }i=t,j=r.
\end{array}
\right.
\end{equation*}
Moreover, the map~\eqref{eq lem explicit U} is injective.
\end{lem}
\begin{proof}
Let~$P:=([x_0: (x_{ij})_{i\in I_t,j\in I_r}],[y_1:\ldots:y_t])$ be a point in~$U_d(a_1,\ldots,a_r)$, and note that the condition $x_{tr}\neq 0$ implies $y_t\neq 0$. Indeed, due to the equations 
	$$x_{tr}y_n^{a_r}=x_{nr}y_t^{a_r} \quad  \text{for }n=1,\ldots, t-1,$$
    the vanishing of $y_t$ would imply that all the coordinates of~$[y_1:\ldots:y_t]\in \mathbb{P}^{t-1}$ would be zero, which is absurd. Hence, we can assume~$y_t=x_{tr}=1$. Putting~$x_i:=y_i$ for~$1\leq i<t$, $x_{t+j-1}:=x_{tj}$ for~$1\leq j<r$, and~$x_d:=x_0$, we obtain a point~$(x_1,\ldots,x_d)\in \A^d$ which defines coefficients~$\tilde{x}_{ij}$ as above. By construction, we have~$\tilde{x}_{ij}=x_{ij}$ for~$i=t$ and for all~$j\in I_r$. Now, it follows from the equations~\eqref{equations HK} with~$n=t$ and~$j<r$ that~$x_{ij}=x_{tj}y_i^{a_j}=x_{t+j-1}x_i^{a_j}=\tilde{x}_{ij}$ for all~$i$ with~$1\leq i<t$. Similarly, using the equations~\eqref{equations HK} with~$n=t$ and~$j=r$, we get~$x_{ir}=x_{tr}y_i^{a_r}=x_i^{a_r}=\tilde{x}_{ir}$ for all~$i$ with~$1\leq i<t$. This shows that the image of~$(x_1,\ldots,x_d)$ under~\eqref{eq lem explicit U} equals~$P$, and proves the first statement of the lemma. Finally, the injectivity of the map~\eqref{eq lem explicit U} follows from the fact that one can recover~$(x_1,\ldots,x_d)$ from the coefficients~$x_0$, $x_{tj}$ with~$1\leq j<r$, and~$y_k$ for~$1\leq k<t$, of~$P$. This completes the proof of the lemma.
	%Therefore, $U_d(a_1,\ldots,a_r)$ corresponds to the affine open subset of $\mathbb{P}^{rt}\times \mathbb{P}^{t-1}$ given by $x_{tr}=1$ and $y_t=1$, which is isomorphic to $\mathbb{A}^d$. Explicitly, the above equations reduce in this case to $x_{nr}=z_n^{a_r}$ for $n=1,\ldots, t-1$ (resp. $x_{ti}=z_{t-1+i}$ for $i=1,\ldots, r-1$, resp. $x_{ni}=z_{t-1+i}z_n^{a_i}$ for $i=1,\ldots, r-1$ and $n=1,\ldots, t-1$), where $(z_i)\in \mathbb{A}^d$. Thus, we obtain an explicit isomorphism via the map that sends $(x_1,\ldots,x_d)\in \mathbb{A}^d$ to the following point in $U_d(a_1,\ldots,a_d)\subseteq \PP^{rt}\times \PP^{t-1}$
	%\begin{equation*}
	%\begin{split}\left([x_d, x_tx_1^{a_1},x_tx_2^{a_1},\ldots,x_tx_{t-1}^{a_1}, x_t, \ldots,x_{d-1}x_1^{a_{r-1}}, \ldots,x_{d-1}x_{t-1}^{a_{r-1}},x_1^{a_r},\ldots, x_{t-1}^{a_r},1],[x_1,\ldots, x_{t-1},1]\right).\end{split}	\end{equation*}
\end{proof}

%%%%%%%%%%%%%%
 
% In the notation of Proposition \ref{prop effective cone}, let $L=\gamma h+\xi f$ be the class of a big line bundle. Then, the height on the set of rational points of $X_d(a_1,\ldots,a_r)$ induced by $D$ is   $$H_L\left( ([x_0: (x_{ij})_{i\in I_t,j\in I_r}],[y_1:\ldots:y_t])\right):=\prod_{v\in \Val(K)}\sup_{i\in I_{t}, j\in I_r}\left\lbrace|x_0|_v,|x_{ij}|_v  \right\rbrace ^{\gamma}\prod_{v\in \Val(K)}\sup_{i\in I_t} \left\lbrace|y_i|_v  \right\rbrace^{\xi },$$ 

We now introduce the formal power series
\begin{equation}\label{eq def Z(T)}
\Zet_{U,L}(T):=\sum_{x\in \mathbb{A}^d} T^{\deg(\mathbf{d}_L(x))},    
\end{equation}
where for~$x=(x_i)\in \mathbb{A}^d$ we put
 \[
 \mathbf{d}_L(x):=\gamma  \sup_{\substack{i\in I_{t-1}\\j\in I_{r-1}}} \big((x_d)_\infty, (x_{t-1+j})_\infty +a_j(x_i)_\infty, a_r(x_i)_\infty\big)+\xi \sup_{i\in I_{t-1}}\big(    (x_i)_\infty\big) .
 \]

%    Taking into account the explicit description of the good open set $U:=U_d(a_1,\ldots,a_r)$ given by
 %    \begin{equation*}
%		\begin{split}
%			U=&			\left\lbrace\left(  [x_d, x_tx_1^{a_1},x_tx_2^{a_1},\ldots,x_tx_{t-1}^{a_1}, x_t, x_{t+1}x_1^{a_2},\ldots, x_{t+1}x_{t-1}^{a_2}, x_{t+1}, \ldots,x_{d-1}x_1^{a_{r-1}}, \ldots,x_{d-1}x_{t-1}^{a_{r-1}},\right. \right. \\ &\left. \left.x_1^{a_r},\ldots, x_{t-1}^{a_r},1], [x_1,\ldots, x_{t-1},1]\right) :(x_i)\in \mathbb{A}^{d}\right\rbrace,
%		\end{split}
%	\end{equation*}
% we consider the formal series 

As in the case of the zeta function of the base field~$K$ (see Section~\ref{subsection arithmetic tools}), we have the following relation.
 
	\begin{lem}\label{lem zetaUL=ZetaUL}
 Assume~$a_r>0$. Then, with the above notation, we have the equality 
 \[
 \zeta_{U,L}(s)=\Zet_{U,L}(q^{-s}).
 \]
	\end{lem}
 
	\begin{proof}
		Let $x=(x_i)\in \mathbb{A}^d$  and denote by~$P_x$ the image of~$x$ in~$U$ under the map  defined in Lemma~\ref{lem:explicit_U}. Then, we have
  %consider  it as an element of $U\subset X_d(a_1,\ldots,a_r)$ (see Remark \ref{remark:explicit_U}). Then, we have the following formula 
		
   \begin{align*}
				H_L(P_x)&=
				\prod_{v}\sup_{\substack{i\in I_{t-1}\\j\in I_{r-1}}}\left\lbrace|x_d|_v,|x_{t-1+j}x_i^{a_j}|_v, |x_{t-1+j}|_v,|x_i^{a_r}|_v,1  \right\rbrace ^{\gamma}\prod_{v}\sup_{i\in I_{t-1}} \left\lbrace|x_i|_v, 1 \right\rbrace^{\xi },
			\end{align*} 
   where each product runs over all~$v\in \Val(K)$. Putting
   \begin{eqnarray*}
       V_1 & := & \left\{v\in \Val(K):\inf_{\substack{i\in I_{t-1}\\j\in I_{r-1}}}\left\lbrace v(x_d),v(x_{t-1+j}x_{i}^{a_j}),v(x_{t-1+j}),v(x_i^{a_r})\right\rbrace <0\right\}, \\
       V_2 & := & \left\{v\in \Val(K): \inf_{i\in I_{t-1}}\left\lbrace v(x_{i})\right\rbrace <0\right\},
   \end{eqnarray*}
   we get
   \begin{align*}
				H_L(P_x)&=\prod_{v\in V_1}q^{-f_v\gamma\inf_{\substack{i\in I_{t-1} \\j\in I_{r-1}}}\left\lbrace v(x_d),v(x_{t-1+j}x_{i}^{a_j}),v(x_{t-1+j}),v(x_i^{a_r})\right\rbrace }\prod_{v\in V_2}q^{-f_v\xi \inf_{i\in I_{i-1}}\left\lbrace  v(x_i)\right\rbrace }.
			\end{align*} 
   By taking $\log_q$ on both sides, we obtain
   	 \begin{align*}
				\log _q H_L(P_x)&=\gamma\sum_{v\in V_1}-f_v \inf_{\substack{i\in I_{t-1}\\j\in I_{r-1}}}\left\lbrace v(x_d),v(x_{t-1+j}x_{i}^{a_j}),v(x_{t-1+j}),v(x_i^{a_r})\right\rbrace \\
    & \hspace{3cm} +\xi\sum_{v\in V_2}-f_v\inf_{i\in I_{t-1}}\left\lbrace  v(x_i)\right\rbrace\\
				&=\gamma\deg \left( \sup_{\substack{i\in I_{t-1}\\j\in I_{r-1}}} \big((x_d)_\infty, (x_{t-1+j}x_i^{a_j})_\infty, (x_{t-1+j})_\infty, a_r(x_i)_\infty\big)\right)\\
    & \hspace{3cm}+\xi \deg \left(\sup_{i\in I_{t-1}}\big(  (x_i)_\infty\big)  \right).
			\end{align*}  
		Since $a_j(x_i)_\infty\leq a_r(x_i)_\infty$, we can add $a_j(x_i)_\infty$ inside the first supremum and use Lemma~\ref{lem:sepDivisor} to get
  		 \begin{align*}
    %&=\gamma\deg \left( \sup_{\substack{i\in I_{r-1}\\j\in I_{t-1}}} \left((x_d)_\infty, (x_{t-1+i}x_j^{a_i})_\infty, (x_{t-1+i})_\infty, a_r(x_j)_\infty\right)\right)+\xi\deg \left(\sup_{j\in I_{t-1}}\left(   (x_j)_\infty\right)  \right)\\
				%&=\gamma\deg \left( \sup_{\substack{i\in I_{r-1}\\j\in I_{t-1}}} \left((x_d)_\infty, (x_{t-1+i}x_j^{a_i})_\infty, (x_{t-1+i})_\infty, a_i(x_{t-1+i})_\infty, a_r(x_j)_\infty\right)\right)\\
				%&\hspace{1.5cm} + \xi \deg \left(\sup_{j\in I_{t-1}}\left( (x_j)_\infty\right)  \right)\\
		\log_q H_L(P_x)&=\gamma\deg \left( \sup_{\substack{i\in I_{t-1}\\j\in I_{r-1}}} \big((x_d)_\infty, (x_{t-1+j})_\infty +a_j(x_i)_\infty, a_r(x_i)_\infty\big) \right)\\
    & \hspace{3cm}+ \xi\deg \left(\sup_{i\in I_{t-1}}\big(  (x_i)_\infty\big)  \right) \\
    &=\deg(\mathbf{d}_L(x)).
		\end{align*}  
  It follows that~$H_L(P_x)^{-s}=q^{-s \deg(\mathbf{d}_L(x))}$, hence
  $$\zeta_{U,L}(s)=\sum_{x\in \A^d} H_L(P_x)^{-s}=\sum_{x\in \A^d}q^{-s \deg(\mathbf{d}_L(x))}=\Zet_{U,L}(q^{-s})$$
  by Lemma~\ref{lem:explicit_U}. This completes the proof of the lemma.
	\end{proof}

Analogous to Bourqui in~\cite[p.~351]{BOURQUI2002}, we define the counting functions %~$\Div^+(K)\to \ZZ$ by
	\begin{equation*}
		\begin{split}
			\tilde{R}_L(D):=&\#\left\lbrace x\in \mathbb{A}^d: 
    \mathbf{d}_L(x)=D \right\rbrace,\\
    R_L(D):=&\#\left\lbrace x\in \mathbb{A}^d: \mathbf{d}_L(x)
      \leq D \right\rbrace.
    %\gamma\sup_{\substack{i\in I_{t-1}\\j\in I_{r-1}}} \left\lbrace (x_d)_\infty,  (x_{t-1+j})_\infty +a_j(x_i)_\infty, a_r(x_i)_\infty\right\rbrace + \xi \sup_{i\in I_{t-1}}\left( (x_i)_\infty\right)   =D \right\rbrace,
		\end{split}
	\end{equation*} 
	%and
	%\begin{equation*}
	%	\begin{split}
	%		R_L(D):=&\#\left\lbrace (x_i)\in \mathbb{A}^d: \gamma\sup_{\substack{i\in I_{t-1}\\j\in I_{r-1}}} \left((x_d)_\infty, (x_{t-1+j})_\infty +a_j(x_i)_\infty, a_r(x_i)_\infty\right)+\xi\sup_{i\in I_{t-1}}\left(  (x_i)_\infty\right)   \leq D \right\rbrace.
	%	\end{split}
	%\end{equation*} 
 
	\begin{remark}\label{remark:R-mu-couple}
	    Note that 
	$R_L(D)=\sum_{0\leq D'\leq D}\tilde{R}_L(D')$, so $(R_L,\tilde{R}_L)$ forms a $\mu$-couple. 
	\end{remark}

It follows from the definition of~$\Zet_{U,L}(T)$ in~\eqref{eq def Z(T)} that
	$$\Zet_{U,L}(T)=\sum_{D\geq 0}\tilde{R}_L(D)T ^{\deg (D)}.$$
 
The following result is crucial for the study of the series $\Zet_{U,L}(s)$, as it allows us to relate this series to the zeta function of the base field~$K$.

	\begin{lem}\label{lem:prodOfZT}
 Assuming~$a_r>0$, the following relation holds:
		$$\Zet_{U,L}(T)=\left( \sum_{D\geq 0}R_L(D)T^{\deg (D)}\right)\frac{1}{\Zet_K(T)}.$$ 
	\end{lem}
\begin{proof}Since $(R_L, \tilde{R}_L)$ is a $\mu$-couple, using Lemma \ref{lem:muCouple} we get
		\begin{equation*}
			\begin{split}
				\Zet_{U,L}(T)&=\sum_{D\geq 0}\tilde{R}_L(D)T ^{\deg(D)}= \sum_{D\geq 0}\sum_{0\leq D'\leq D}\mu(D-D')R_L(D')T^{\deg(D)}\\
				&=\sum_{D\geq0}\sum_{D'\geq 0}\mu(D)R_L(D')T^{\deg(D+D')}=\left( \sum_{D\geq 0}R_L(D)T^{\deg(D)}\right)\left( \sum_{D'\geq 0}\mu(D')T^{\deg(D')}\right) \\
				&=\left( \sum_{D\geq 0}R_L(D)T ^{\deg(D)}\right) \frac{1}{\Zet_K(T)},
			\end{split}
		\end{equation*}
  where in the last equality we used Lemma~\ref{lem inversion Z_K(T)}. This proves the desired identity.
	\end{proof}

\subsection{Proof of Theorem~\ref{thm:ThmGeneralversion}}\label{sec proof of general thm}

In this section we give the proof of Theorem~\ref{thm:ThmGeneralversion}, using the preliminary lemmas from Section~\ref{sec preliminary lemmas}.

For the sake of the reader's convenience, we divide the proof into four independent steps.

\noindent  {\bf Step 1. Rewriting the function $R_L(D)$.} For positive integers $n,m$ and~$D\in \Div^+(K)$, we define
    \begin{equation}\label{eq def of Nmn(D)}
      N_m^n(D):=\#\left\{(x_i)\in \mathbb{A}^m:n\sup_{i\in I_m}(x_i)_\infty\leq D\right\},  
    \end{equation}
	and
 \begin{equation}\label{eq def of tildeNm(D)}
	\tilde{N}_m(D):=\#\left\{(x_i)\in \mathbb{A}^m:\sup_{i\in I_m}(x_i)_\infty=D\right\}.
 \end{equation}
     Note that $(N_m^1, \tilde{N}_m)$ forms a $\mu$-couple and $N_m^1(D)=(N_1^1(D))^m=q^{m\ell(D)}$ (see Section~\ref{subsection arithmetic tools}).
     
	For each integer~$c\geq 1$, we also define the set 
 $$\Div_c:=\{D\in \Div^+(K):v(D)\leq c \text{ for all }v\in \Val(K)\},$$
 and  note that by the division algorithm all divisor~$D\geq 0$ can be written uniquely as $D=(c+1)D_1+D_2$, where $D_1\geq 0$ and $D_2\in \Div_c$. Observe also that 
 \begin{equation}\label{eq prop of N1}
 N_1^{c+1}((c+1)D_1+D_2)=N_1^1(D_1) \quad \text{for all }D_1\geq 0,D_2\in \Div_c.
 \end{equation}
 In particular, choosing~$c=c_L-1$ where
 $$c_L:=\gamma a_r+\xi,$$
 we have that % $\Div_c=\{D\in \Div^+(K):v(D)\leq \gamma a_r+\xi-1\}.$ Then each 
  every~$D\geq 0$ may be written uniquely as 
	%$$D=(\gamma a_r+\xi)D_1+D_2, $$
 \begin{equation}\label{eq D decomposition wrt c_L}
 D=c_LD_1+D_2, \quad \text{with }D_1\geq 0,D_2\in \Div_{c_L-1},     
 \end{equation}
 and for every~$D'\geq 0$, we have
 \begin{equation}\label{eq criterio ineq D'leq D}
 c_L D'\leq D \Leftrightarrow D'\leq D_1.     
 \end{equation}
 %the condition $\left( \gamma a_r+\xi\right) D'\leq D$ is equivalent to $D'\leq D_1$.  
 %Now, for a fixed divisor~$D$ and a point $x=(x_i)\in (K^\times)^d$ such that~$\mathbf{d}_L(x_i)\leq D$,  
 %\begin{equation*}
 %        \gamma\sup_{\substack{i\in I_{r-1}\\j\in I_{t-1}}} \left((x_d)_\infty, (x_{t-1+i})_\infty +a_i(x_j)_\infty, a_r(x_j)_\infty\right)+\sup_{j\in I_{t-1}}\left( \xi (x_j)_\infty\right)   \leq D, 
 %\end{equation*}
%we have that $D_x':=\sup_{i\in I_{t-1}} ((x_i)_\infty)$ satisfies $\left( \gamma a_r+\xi\right) D'\leq D$. As a consequence, for every~$D=\left( \gamma a_r+\xi\right) D_1+D_2\geq 0$,

Now, recalling that~$d=r+t-1$, it follows from the definition of~$R_L(D)$ and~$ \mathbf{d}_L(x)$ that
	\begin{equation*}
		\begin{split}
			R_L(D) &= \#\left\lbrace (x_i)\in \mathbb{A}^d:\gamma \sup_{\substack{i\in I_{t-1}\\j\in I_{r-1}}} \big((x_d)_\infty,(x_{t-1+j})_\infty +a_j (x_i)_\infty , a_r(x_i)_\infty\big) +\xi \sup_{i\in I_{t-1}} \big( (x_i)_\infty\big) \leq D \right\rbrace \\
   &=\sum_{0\leq D'}\tilde{N}_{t-1}(D')\# \big\lbrace (x_i)\in \mathbb{A}^{r}:\gamma \sup_{j\in I_{r-1}}((x_r)_\infty,(x_{j})_\infty +a_jD', a_rD') +\xi D'\leq D\big\rbrace.
		\end{split}
	\end{equation*}
 Since
 $$c_LD'\leq \gamma \sup_{j\in I_{r-1}}((x_r)_\infty,(x_{j})_\infty +a_jD', a_rD') +\xi D',$$
 writing~$D$ as in~\eqref{eq D decomposition wrt c_L} we  get (using~\eqref{eq criterio ineq D'leq D})
 	\begin{equation*}
 \label{eq:newRD}
		\begin{split}
			R_L(D) %&= \#\left\lbrace (x_i)\in \mathbb{A}^d: \mathbf{d}_L(x_i) \leq D \right\rbrace \\
   &=\sum_{0\leq D'\leq D_1}\tilde{N}_{t-1}(D')\# \big\lbrace (x_i)\in \mathbb{A}^{r}:\gamma \sup_{j\in I_{r-1}}((x_r)_\infty,(x_{j})_\infty +a_jD', a_rD') +\xi D'\leq D\big\rbrace.
		\end{split}
	\end{equation*}
Furthermore, using~\eqref{eq criterio ineq D'leq D} we see that for $D'$ satisfying~$0\leq D'\leq D_1$, %as in Equation \ref{eq:newRD}, 
the condition
	$$\gamma \sup_{j\in I_{r-1}}((x_r)_\infty,(x_{j})_\infty +a_jD', a_rD') +\xi D'\leq D,$$
	is equivalent to the system of inequalities
	$$\left\lbrace \begin{array}{ll}
		\gamma  (x_r)_\infty\leq D-\xi D',	&  \\
		\gamma  (x_{j})_\infty \leq D-\left( \xi+\gamma a_j\right)D',	& \text{for all } j\in I_{r-1}.
	\end{array}\right. $$
	Therefore, we obtain
	\begin{equation*}
      R_L(D)=\sum_{0\leq D'\leq D_1}\tilde{N}_{t-1}(D')%\prod_{i=1}^{r-1}
		N_{1}^{\gamma }\left( D-\xi D'\right)\prod_{j=1}^{r-1}N_1^{\gamma}\left( D-\left(\xi+\gamma a_j\right)D'\right).
	\end{equation*}
 Putting~$a_0:=0$, we rewrite this identity as
 \begin{equation}\label{eq:expresionForR(D)}
		R_L(D)=\sum_{0\leq D'\leq D_1}\tilde{N}_{t-1}(D')\prod_{j=0}^{r-1}N_1^{\gamma}\left( D-\left(\xi+\gamma a_j\right)D'\right).
 \end{equation}

   % \vspace{2mm}

\noindent    {\bf Step 2. Analysis of the formal series $\sum_{D\geq 0}R_L(D)T^{\deg (D)}$.}
    Let us define
    $$\Zet_1(T):=\sum_{D\geq 0}R_L(D)T^{\deg (D)}.$$
    %Note that we can rewrite the identity in Lemma \ref{lem:prodOfZT} as
    Using~\eqref{eq D decomposition wrt c_L} we have
	\begin{equation*}
		\Zet_1(T)=\sum_{D_1\geq 0, D_2\in \Div_{c_L-1}}R_L\left( c_L D_1+D_2\right) T^{\deg\left( c_L  D_1+D_2\right) }.
	\end{equation*}
	Fixing now $D_1\geq 0,D_2\in  \Div_{c_L-1}$ and using (\ref{eq:expresionForR(D)}) we have
 	\begin{align*}
			&R_L\left( c_L D_1+D_2\right) T^{\deg(c_LD_1+D_2)}\\
			=&\sum_{0\leq D'\leq D_1}\tilde{N}_{t-1}(D')\prod_{j=0}^{r-1} N_1^{\gamma }\left( c_L D_1+D_2-\left(\xi+\gamma a_j\right)D'\right) T^{\deg\left( c_LD_1+D_2\right) }\\
			=&\sum_{0\leq D'\leq D_1}\tilde{N}_{t-1}(D')\prod_{j=0}^{r-1} N_1^{\gamma}\left( c_L (D_1-D')+D_2+\gamma (a_r-a_j)D' \right) T^{\deg\left( c_L(D_1-D')+D_2+c_LD' \right)}.
	\end{align*}   
  Writing~$D_1-D'=D''$ with~$D''\geq 0$ we get (using~\eqref{eq prop of N1})
   	\begin{align*}
			\Zet_1(T)&=\sum_{\substack{D'\geq 0,D''\geq 0\\D_2\in \Div_{c_L-1}}}\tilde{N}_{t-1}(D')\prod_{j=0}^{r-1} N_1^{\gamma }\left( c_L D''+D_2+\gamma (a_r-a_j)D' \right)T^{\deg \left( c_LD''+D_2+c_LD' \right) }\\
			&=\sum_{D\geq 0, D'\geq 0}\tilde{N}_{t-1}(D')\prod_{j=0}^{r-1} N_1^{\gamma }\left( D+\gamma (a_r-a_j)D' \right)T^{\deg \left( D+c_LD' \right) }\\
			&=\sum_{\substack{D\geq 0,D'\geq 0\\D_1\in \Div_{\gamma-1}}}\tilde{N}_{t-1}(D')\prod_{j=0}^{r-1} N_1^{\gamma }\left( \gamma D+D_1+\gamma (a_r-a_j)D' \right)T^{\deg \left( \gamma D+D_1+c_LD' \right) }\\
			&=\sum_{\substack{D\geq 0,D'\geq 0\\D_1\in \Div_{\gamma-1}}}\tilde{N}_{t-1}(D')\prod_{j=0}^{r-1} N_1^{1}\left( D+(a_r-a_j)D' \right)T^{\deg \left( \gamma D+D_1+c_LD' \right)}.
   \end{align*}    
   Recall that~$N_X=\#\{j\in \{1,\ldots,r\}:a_j=a_r\}$ and note that~$j=r-N_X$ is the largest non-negative index satisfying~$a_j<a_r$. Hence, we have
   \begin{align*}
			\Zet_1(T)
			&=\sum_{\substack{D\geq 0,D'\geq 0\\D_1\in \Div_{\gamma-1}}}\tilde{N}_{t-1}(D')N_1^1(D)^{N_X-1}\prod_{j=0}^{r-N_X} N_1^{1}\left( D+(a_r-a_j)D' \right) T^{\deg \left( \gamma D+D_1+c_LD' \right)}.
   \end{align*} 
%By the Riemann--Roch theorem (see Section~\ref{subsection:Notation}) we can write
Now, define
   \begin{align*}
			\Zet_2(T)		&:=\sum_{\substack{D\geq 0,D'\geq 0\\D_1\in \Div_{\gamma-1}}}\tilde{N}_{t-1}(D')N_1^1(D)^{N_X-1}\prod_{j=0}^{r-N_X}q^{1-g}q^{\deg (D)+(a_r-a_j)\deg (D')} T^{\deg \left( \gamma D+D_1+c_LD' \right)}.
	\end{align*}     
Putting~$A:=\min\{a_r-a_j:j\in \{0,\ldots,r-N_X\}\}=a_{r}-a_{r-N_X}$, we note by the Riemann--Roch theorem (see Section~\ref{subsection:Notation}) that
$$ N_1^{1}\left( D+(a_r-a_j)D' \right)=q^{1-g}q^{\deg (D)+(a_r-a_j)\deg (D')} \quad \text{for all }j\in \{0,\ldots,r-N_X\},$$
provided~$\deg(D+AD')> 2g-2$. This implies
   \begin{eqnarray}
   %\begin{split}
       		\Zet_1(T)-\Zet_2(T)		&=&\sum_{\substack{D\geq 0,D'\geq 0\\ \deg(D)+A\deg(D')\leq 2g-2 \\D_1\in \Div_{\gamma-1}}}\tilde{N}_{t-1}(D')N_1^1(D)^{N_X-1}\left(\prod_{j=0}^{r-N_X} N_1^{1}( D+(a_r-a_j)D') \right. \notag  \\
			&&\hspace{2cm} \left.- q^{(r+1-N_X)(1-g)}\prod_{j=0}^{r-N_X}q^{\deg (D)+(a_r-a_j)\deg (D')}\right) T^{\deg \left( \gamma D+D_1+c_LD' \right)} \notag \\
   &=& P_1(T)\sum_{D_1\in \Div_{\gamma-1}}T^{\deg(D_1)}, \label{eq Z1=Z2+P1}
   %\end{split}
   \end{eqnarray}  %  
   where~$P_1(T)\in \QQ[T]$, since~$A>0$ and there are only finitely many pairs of divisors~$D\geq 0,D'\geq 0$ satisfying~$\deg(D+AD')\leq 2g-2$.
   %  \begin{align*}	P_1(T)&:=\sum_{\substack{D\geq 0,D'\geq 0\\ \deg(D+AD')\leq 2g-2}}\tilde{N}_{t-1}(D') \bigg(N_1^{1}(D+a_rD')\prod_{j=1}^{r-1} N_1^{1}\left( D+(a_r-a_j)D' \right)  \\ 	&\hspace{2cm}- q^{r(1-g)}q^{\deg (D)+a_r\deg (D')}\prod_{j=1}^{r-1}q^{\deg (D)+(a_r-a_j)\deg (D')}\bigg)T^{\deg \left( \gamma D+c_LD' \right)} \in \QQ[T]. \end{align*}
Next, we further compute
	\begin{align*}
			%& %Z_1(T)-P_1(T)\cdot\sum_{D_1\in \Div_{\gamma-1}}T^{\deg (D_1)}\\
			\Zet_2(T)=&q^{(r+1-N_X)(1-g)}\left( \sum_{D'\geq 0}\tilde{N}_{t-1}(D')q^{ \left( (r+1) a_r-|\mathbf{a}| \right) \deg \left(D'\right)}T^{c_L\deg (D')}\right) \\
			&\hspace{2cm}\times \left(\sum_{D\geq 0}N_1^1(D)^{N_X-1}\left( q^{r+1-N_X}T^{\gamma} \right)^{\deg (D)}  \right) \left( \sum_{D_1\in \Div_{\gamma-1}}T^{\deg (D_1)}\right).
%			=&q^{r(1-g)}\left( \sum_{D'\geq 0}\tilde{N}_{t-1}(D')q^{\left( (r+1) a_r-|\mathbf{a}| \right) \deg \left(  D'\right)}T^{c_L\deg (D')}\right) \Zet_K\left( q^rT^{\gamma}\right) \left( \sum_{D_1\in \Div_{\gamma-1}}T^{\deg (D_1)}\right).\\
	\end{align*}
Putting
	\begin{align*}
\Zet_3(T):=&q^{(r+1-N_X)(1-g)}\left( \sum_{D'\geq 0}\tilde{N}_{t-1}(D')q^{ \left( (r+1) a_r-|\mathbf{a}| \right) \deg \left(D'\right)}T^{c_L\deg (D')}\right) \\
			&\hspace{2cm}\times \left(\sum_{D\geq 0}\left(q^{1-g+\deg(D)}\right)^{N_X-1}\left( q^{r+1-N_X}T^{\gamma} \right)^{\deg (D)}  \right) \left( \sum_{D_1\in \Div_{\gamma-1}}T^{\deg (D_1)}\right)\\
   =&q^{r(1-g)}\left( \sum_{D'\geq 0}\tilde{N}_{t-1}(D')q^{ \left( (r+1) a_r-|\mathbf{a}| \right) \deg \left(D'\right)}T^{c_L\deg (D')}\right) \\
			&\hspace{2cm}\times \left(\sum_{D\geq 0}\left( q^{r}T^{\gamma} \right)^{\deg (D)}  \right) \left( \sum_{D_1\in \Div_{\gamma-1}}T^{\deg (D_1)}\right)\\
\end{align*}
we have, using the Riemann--Roch theorem once more, that
\begin{equation}\label{eq Z2-Z3}
\begin{split}
 \Zet_2(T)-\Zet_3(T)=&q^{(r+1-N_X)(1-g)}\left( \sum_{D'\geq 0}\tilde{N}_{t-1}(D')q^{ \left( (r+1) a_r-|\mathbf{a}| \right) \deg \left(D'\right)}T^{c_L\deg (D')}\right)\\
&\hspace{2cm}\times  P_2(T) \left( \sum_{D_1\in \Div_{\gamma-1}}T^{\deg (D_1)}\right),   
\end{split}
\end{equation}
   with
$$P_2(T):=\sum_{\substack{D\geq 0\\ \deg(D)\leq 2g-2}}\left(N_1^1(D)^{N_X-1}-\big(q^{1-g+\deg(D)}\big)^{N_X-1}\right) \left(q^{r+1-N_X}T^{\gamma} \right)^{\deg (D)}\in \QQ[T].$$

\noindent   {\bf Step 3. Rewriting the formal series~$\Zet_3(T)$.} 
	Taking into account that $(N_{t-1}^1,\tilde{N}_{t-1})$ is a $\mu$-couple, and using Lemma \ref{lem:prodOfZT}, we have
 	\begin{equation*}
		\begin{split}
			&\sum_{D'\geq 0}\tilde{N}_{t-1}(D') q^{\left( (r+1) a_r-|\mathbf{a}| \right) \deg\left(D'\right)}T^{c_L\deg(D')}\\
			=&\sum_{D\geq 0}\sum_{0\leq D'\leq D}\mu(D-D')N_{t-1}^1(D') q^{\left( (r+1) a_r-|\mathbf{a}| \right) \deg\left(D\right)}T^{c_L\deg(D)}\\
			=&\sum_{D\geq 0}\sum_{D'\geq 0}\mu(D)N_{t-1}^1(D')q^{ \left( (r+1) a_r-|\mathbf{a}| \right) \deg (D+D')+}T^{c_L \deg (D+D')}\\
			=&\left( \sum_{D'\geq 0}N_{1}^1(D')^{t-1}q^{ \left( (r+1) a_r-|\mathbf{a}| \right) \deg (D')}T^{c_L\deg (D')}\right)\left(\sum_{D\geq 0}\mu(D)q^{ \left( (r+1) a_r-|\mathbf{a}| \right) \deg (D)}T^{c_L\deg (D)} \right) \\
   =& \left( \sum_{D'\geq 0}N_{1}^1(D')^{t-1}q^{ \left( (r+1) a_r-|\mathbf{a}| \right) \deg (D')}T^{c_L\deg (D')}\right)\frac{1}{\Zet_K\left( q^{  (r+1) a_r-|\mathbf{a}| }T^{ c_L}\right)}.
			%=&\left( \sum_{D\geq 0}q^{(t-1)(1-g)}q^{(t-1)\deg (D)}q^{ \left( (r+1) a_r-|\mathbf{a}| \right) \deg (D)}T^{\deg \left( \left( \gamma a_r+\xi \right)D\right) }+P_2(T)\right)\frac{1}{\Zet_K\left( q^{  (r+1) a_r-|\mathbf{a}| }T^{ \gamma a_r+\xi }\right) }\\
			%=&\left( q^{(t-1)(1-g)}\Zet_K\left(q^{  (r+1) a_r-|\mathbf{a}|  +t-1}T^{\gamma a_r+\xi } \right) +P_2(T)\right)\frac{1}{\Zet_K\left( q^{  (r+1) a_r-|\mathbf{a}| }T^{c_L }\right) }.
		\end{split}
	\end{equation*}
 Using again the Riemann--Roch theorem, we obtain
\begin{equation*}
		\begin{split}
			&\sum_{D'\geq 0}\tilde{N}_{t-1}(D') q^{\left( (r+1) a_r-|\mathbf{a}| \right) \deg\left(D'\right)}T^{c_L\deg(D')}\\
			=&\left( \sum_{D'\geq 0}q^{(t-1)(1-g)}q^{(t-1)\deg (D')}q^{ \left( (r+1) a_r-|\mathbf{a}| \right) \deg (D')}T^{c_L\deg (D') }+P_3(T)\right)\frac{1}{\Zet_K\left( q^{  (r+1) a_r-|\mathbf{a}| }T^{ c_L  }\right) }\\
			=&\left( q^{(t-1)(1-g)}\Zet_K\left(q^{  (r+1) a_r-|\mathbf{a}|  +t-1}T^{c_L } \right) +P_3(T)\right)\frac{1}{\Zet_K\left( q^{  (r+1) a_r-|\mathbf{a}|}T^{c_L }\right) },
   \end{split}
   \end{equation*}
   where
 \begin{equation*}
		\begin{split}
			P_3(T):=&\sum_{\substack{D'\geq 0\\\deg(D')\leq 2g-2}}\left(N_{1}^1(D')^{t-1}- q^{(t-1)(1-g+\deg (D'))}\right)q^{ \left( (r+1) a_r-|\mathbf{a}| \right) \deg (D')}T^{c_L\deg (D')}\in \QQ[T].
   \end{split}
   \end{equation*}
 %\vspace{-15mm}
 We conclude
 \begin{equation}\label{eq Z3}
		\begin{split}
			\Zet_3(T)&=\left( q^{(t-1)(1-g)}\Zet_K\left(q^{  (r+1) a_r-|\mathbf{a}|  +t-1}T^{c_L } \right) +P_3(T)\right)
    \\
& \hspace{3cm} \times \frac{q^{r(1-g)}\Zet_K\left( q^rT^{\gamma}\right)}{\Zet_K\left( q^{  (r+1) a_r-|\mathbf{a}|}T^{ c_L  }\right)}\left( \sum_{D_1\in \Div_{\gamma-1}}T^{\deg (D_1)}\right).
   \end{split}
   \end{equation}
Note that, by~\eqref{eq Z2-Z3}, we also get
\begin{equation}\label{eq Z2 and Z3}
\begin{split}
        \Zet_2(T)-\Zet_3(T)=&q^{(r+1-N_X)(1-g)}\left( q^{(t-1)(1-g)}\Zet_K\left(q^{  (r+1) a_r-|\mathbf{a}|  +t-1}T^{c_L } \right) +P_3(T)\right)\\
&\hspace{2cm}\times  \frac{P_2(T)}{\Zet_K\left( q^{  (r+1) a_r-|\mathbf{a}|}T^{c_L }\right) }\left( \sum_{D_1\in \Div_{\gamma-1}}T^{\deg (D_1)}\right).
\end{split}
\end{equation}

\noindent    {\bf Step 4. Analytic behaviour of $\zeta_{U,D}(s)$.}  Since all $D\geq 0$ can be written as $D=\gamma D_1+D_2,$
	with $D_1\geq 0$ and $D_2\in \Div_{\gamma-1}$, we have
 \[\Zet_K(T)=\left( \sum_{D\geq 0}T^{\gamma \deg(D)}\right) \left(\sum_{D\in \Div_{\gamma-1}}T^{\deg (D)}\right).
 \]
 %\vspace{-15mm}
	This implies
	\[
 \sum_{D\in \Div_{\gamma-1}}T^{\deg(D)}=\frac{\Zet_K(T)}{\Zet_K(T^{\gamma})},
 \]
 and by Lemma~\ref{lem:prodOfZT} together with~\eqref{eq Z1=Z2+P1} we get
 $$\Zet_{U,L}(T)=\frac{\Zet_1(T)}{\Zet_K(T)}=\frac{\Zet_2(T)}{\Zet_K(T)}+\frac{P_1(T)}{\Zet_K(T^{\gamma})}.$$
 Then, using~\eqref{eq Z2 and Z3} we have
 \begin{equation*}
     \begin{split}
  \Zet_{U,L}(T)=&\frac{\Zet_3(T)}{\Zet_K(T)}+ \frac{q^{(d+1-N_X)(1-g)}P_2(T)\Zet_K\left(q^{  (r+1) a_r-|\mathbf{a}|  +t-1}T^{c_L } \right)}{\Zet_K\left( q^{  (r+1) a_r-|\mathbf{a}|}T^{c_L }\right) \Zet_K(T^{\gamma})} \\
&\hspace{2cm}   +\frac{q^{(r+1-N_X)(1-g)}P_3(T)P_2(T)}{\Zet_K\left( q^{  (r+1) a_r-|\mathbf{a}|}T^{c_L }\right) \Zet_K(T^{\gamma})}+\frac{P_1(T)}{\Zet_K(T^{\gamma})} 
     \end{split}
 \end{equation*}
 Finally, using~\eqref{eq Z3} we get
 \begin{equation*}
     \begin{split}
  \Zet_{U,L}(T)=& \frac{q^{d(1-g)}\Zet_K\left(q^{  (r+1) a_r-|\mathbf{a}|  +t-1}T^{c_L } \right)\Zet_K\left( q^rT^{\gamma}\right)}{\Zet_K\left( q^{  (r+1) a_r-|\mathbf{a}|}T^{ c_L  }\right)\Zet_K(T^{\gamma})}+
\frac{q^{r(1-g)}P_3(T)\Zet_K\left( q^rT^{\gamma}\right)}{\Zet_K\left( q^{  (r+1) a_r-|\mathbf{a}|}T^{ c_L  }\right)\Zet_K(T^{\gamma})}\\
&+ \frac{q^{(d+1-N_X)(1-g)}P_2(T)\Zet_K\left(q^{  (r+1) a_r-|\mathbf{a}|  +t-1}T^{c_L } \right)}{\Zet_K\left( q^{  (r+1) a_r-|\mathbf{a}|}T^{c_L }\right) \Zet_K(T^{\gamma})}   +\frac{q^{(r+1-N_X)(1-g)}P_3(T)P_2(T)}{\Zet_K\left( q^{  (r+1) a_r-|\mathbf{a}|}T^{c_L }\right) \Zet_K(T^{\gamma})}\\
&\hspace{0.5cm}+\frac{P_1(T)}{\Zet_K(T^{\gamma})}.
     \end{split}
 \end{equation*}
	%\begin{align*}
	%	&\Zet_{U,L}(T)
		%	=Z_1(T)\frac{1}{\Zet_K(T)}\\
			%=&\left[q^{r(1-g)}\left( q^{(t-1)(1-g)}\Zet_K\left(q^{(r+1)a_r-|\mathbf{a}|  +t-1}T^{\gamma a_r+\xi} \right) +P_2(T)\right)\frac{1}{\Zet_K\left( q^{  (r+1)a_r-|\mathbf{a}| }T^{ \gamma a_r+\xi }\right)}\Zet_K\left( q^rT^{\gamma}\right)\right. \\
			%&\hspace{2cm}\times\left. \left( \sum_{D_1\in \Div_{\gamma-1}}T^{deg(D_1)}\right)+ P_1(T)\cdot\sum_{D_1\in \Div_{\gamma-1}}T^{deg(D_1)}\right]\frac{1}{\Zet_K(T)}\\
			%=&q^{r(1-g)}\left( q^{(t-1)(1-g)}\Zet_K\left(q^{(r+1)a_r-|\mathbf{a}|  +t-1}T^{\gamma a_r+\xi }  \right) +P_2(T)\right)  \\
			%&\hspace{2cm} \times \frac{1}{\Zet_K\left( q^{(r+1)a_r-|\mathbf{a}|}T^{ \gamma a_r+\xi}\right)}\Zet_K\left( q^rT^{\gamma }\right)\frac{1}{\Zet_K(T^{\gamma })}  +\frac{1}{\Zet_K(T^{\gamma})} P_1(T)\\
	%\end{align*}
	By Lemma~\ref{lem zetaUL=ZetaUL} we conclude
	\begin{align*}
	   \zeta_{U,L}(s)%&=\Zet_{U,L}(q^{-s})\\
			&= \frac{q^{d(1-g)}\zeta_K\left( \left( \gamma a_r+\xi \right) s-\left( (r+1)a_r-|\mathbf{a}|  +t-1\right) \right)\zeta_K(\gamma s-r)}{\zeta_K\left(\left( \gamma a_r+\xi \right)s-\left( (r+1)a_r-|\mathbf{a}|\right)  \right)\zeta_K(\gamma s) } \\
			&\hspace{.5cm}+\frac{q^{r(1-g)}P_3(q^{-s})\zeta_K(\gamma s-r)}{\zeta_K\left(\left( \gamma a_r+\xi \right)s-\left( (r+1) a_r-|\mathbf{a}|\right)   \right)\zeta_K(\gamma s)}\\
   &\hspace{1cm}+\frac{q^{(d+1-N_X)(1-g)}P_2(q^{-s})\zeta_K\left( \left( \gamma a_r+\xi \right) s-\left( (r+1)a_r-|\mathbf{a}|  +t-1\right) \right)}{\zeta_K\left(\left( \gamma a_r+\xi \right)s-\left( (r+1) a_r-|\mathbf{a}|\right)   \right)\zeta_K(\gamma s)}    \\
   &\hspace{1.5cm}+\frac{q^{(r+1-N_X)(1-g)}P_3(q^{-s})P_2(q^{-s})}{\zeta_K\left(\left( \gamma a_r+\xi \right)s-\left( (r+1) a_r-|\mathbf{a}|\right)   \right)\zeta_K(\gamma s)}+\frac{P_1(q^{-s})}{\zeta_K(\gamma s)}.    
	\end{align*}
It follows from the properties of the zeta function~$\zeta_K(s)$ (see Section~\ref{subsection arithmetic tools}) that~$\zeta_{U,L}(s)$ is a rational function in~$q^{-s}$. Moreover, using Corollary~\ref{cor inverse of zetaK converges absolutely} and~\eqref{eq a(L) and b(L)} we see that~$\zeta_{U,L}(s)$ converges absolutely for
$$\Re(s)>\max\{A_L,B_L\}=a(L),$$
with~$A_L,B_L$ defined in~\eqref{eq def of AL and BL}. Moreover, noting that the sets~$\mathcal{A}_L$ and~$\mathcal{B}_L$ are disjoint (because~$\mathrm{g.c.d.}(\gamma,\gamma a_r+\xi)=\mathrm{g.c.d.}(\gamma,\xi)=1$), we get the following properties:
\begin{enumerate}
    \item If~$A_L=B_L$, then~$\zeta_{U,L}(s)$ has %holomorphic continuation to
    %$$\Re(s)>\max\left\{\frac{(r+1)a_r-|\mathbf{a}|  +t-1}{ \gamma a_r+\xi},\frac{r}{\gamma}\right\},s\neq A_L$$
    %and it has 
    a pole of order 2 at~$s=a(L)$ with
    \begin{equation*}
        \begin{split}
        \lim_{s\to a(L)}(s-a(L))^2\zeta_{U,L}(s)&=\frac{q^{d(1-g)}}{\zeta_K(t)\zeta_K(r+1)}\frac{\big(\textup{Res}_{s=1}\zeta_K(s)\big)^2}{(\gamma a_r+\xi)\gamma}\\
        &=\frac{q^{(d+2)(1-g)} h_K^2}{\zeta_K(t)\zeta_K(r+1)(\gamma a_r+\xi)\gamma (q-1)^2\log(q)^2}.    
        \end{split}
    \end{equation*}
    Moreover,  $\zeta_{U,L}(s)$ is holomorphic for
    $$\Re(s)>\max\left\{A_L-\frac{1}{\gamma},B_L-\frac{1}{\gamma a_r+\xi}\right\},s\not \in \left\{w+\frac{2\pi im}{\log(q)}:w\in  \{a(L)\}\cup \mathcal{A}_L\cup\mathcal{B}_L, m\in \ZZ\right\},$$
    and satisfies~$\lim_{s\to w}(s-w)^2\zeta_{U,L}(s)=0$ for every~$w\in \mathcal{A}_L\cup\mathcal{B}_L$.
    %has simple poles at every~$s\in \mathcal{A}_L\cup\mathcal{B}_L$ and is analytic for~$\Re(s)=a(L),\Im(s)\in \left[0,\frac{2\pi i }{\log(q)}\right[,s\not \in \{a(L)\}\cup \mathcal{A}_L\cup\mathcal{B}_L$.
\item If~$A_L<B_L$, then $\zeta_{U,L}(s)$ is holomorphic for
    $$\Re(s)>\max\left\{A_L,B_L-\frac{1}{\gamma a_r+\xi}\right\},s\not \in \left\{w+\frac{2\pi im}{\log(q)}:w\in  \{a(L)\}\cup \mathcal{B}_L, m\in \ZZ\right\},$$
and for every~$w\in \{a(L)\}\cup \mathcal{B}_L$ it satisfies
 \begin{equation*}
        \begin{split}
\lim_{s\to w}(s-w)\zeta_{U,L}(s)%&=\lim_{s\to a(L)}(s-a(L))\zeta_{U,L}(s) \\
&= \left(q^{d(1-g)}\zeta_K(\gamma w-r)+ q^{(d+1-N_X)(1-g)}P_2(q^{-w})\right) \\
& \hspace{2cm} \times \frac{\textup{Res}_{s=1}\zeta_K(s)}{\zeta_K\left(t  \right)\zeta_K(\gamma w)(\gamma a_r+\xi)}\\
&= \left(q^{d(1-g)}\zeta_K(\gamma w-r)+ q^{(d+1-N_X)(1-g)}P_2(q^{-w})\right)\\
&\hspace{2cm} \times\frac{h_Kq^{1-g}}{\zeta_K\left(t  \right)\zeta_K(\gamma w)(\gamma a_r+\xi)(q-1)\log(q)}.
\end{split}
\end{equation*}
Moreover, from the definition of~$P_2(T)$ and~\eqref{eq identity Rk} we have
\begin{equation*}
\begin{split}
    & q^{d(1-g)}\zeta_K(\gamma w-r)+ q^{(d+1-N_X)(1-g)}P_2(q^{-w})\\
    =&
     q^{(d+1-N_X)(1-g)}\sum_{\substack{D\geq 0 \\ \deg(D)\leq 2g-2}}q^{(N_X-1)\ell(D)+(r+1-N_X-\gamma w)\deg(D)}+q^{d(1-g)}\sum_{\substack{D\geq 0 \\ \deg(D)> 2g-2}}q^{-(\gamma w-r)\deg(D)}\\
     =&q^{(d+1-N_X)(1-g)}\R_K(1-N_X,\gamma w-r+N_X-1).
\end{split}
\end{equation*}
In particular, since~$\R_K(1-N_X,\gamma B_L-r+N_X-1)$ is positive, we see that~$\zeta_{U,L}(s)$ has a pole of order~$b(L)=1$ at~$s=a(L)$. %with Since this value is positive, we conclude that~$\zeta_{U,L}(s)$ has a simple pole at every~$s\in \{a(L)\}\cup \mathcal{B}_L$ in this case.
\item If~$A_L>B_L$, then~$\zeta_{U,L}(s)$ is holomorphic for
    $$\Re(s)>\max\left\{A_L-\frac{1}{\gamma},B_L\right\},s\not \in \left\{w+\frac{2\pi im}{\log(q)}:w\in  \{a(L)\}\cup \mathcal{A}_L, m\in \ZZ\right\},$$
and for every~$w\in \{a(L)\}\cup \mathcal{A}_L$ it satisfies
%has holomorphic continuation to
    %$$\Re(s)>\left\{B_L,\frac{r}{\gamma}\right\},s\neq A_L$$
    %and satisfies    satisfies
  \begin{equation*}
        \begin{split}
\lim_{s\to w}(s-w) \zeta_{U,L}(s)&=\left( q^{d(1-g)}\zeta_K\left( \left( \gamma a_r+\xi \right)w-\left( (r+1)a_r-|\mathbf{a}|  +t-1\right)\right)+q^{r(1-g)}P_3(q^{-w})\right)\\
& \hspace{2cm} \times \frac{\textup{Res}_{s=1}\zeta_K(s)}{\zeta_K\left(\left( \gamma a_r+\xi \right)w-\left( (r+1)a_r-|\mathbf{a}|\right)  \right)\zeta_K(r+1)\gamma}\\
&=\left( q^{d(1-g)}\zeta_K\left( \left( \gamma a_r+\xi \right)w-\left( (r+1)a_r-|\mathbf{a}|  +t-1\right)\right)+q^{r(1-g)}P_3(q^{-w})\right)\\
& \hspace{2cm} \times \frac{h_K q^{1-g}}{\zeta_K\left(\left( \gamma a_r+\xi \right)w-\left( (r+1)a_r-|\mathbf{a}|\right)  \right)\zeta_K(r+1)\gamma (q-1)\log(q)}.
        \end{split}
        \end{equation*}
        Using the definition of~$P_3(T)$ and~\eqref{eq identity Rk}, we get
    \begin{equation*}
    \begin{split}
        & q^{d(1-g)}\zeta_K((\gamma a_r+\xi )w-( (r+1)a_r-|\mathbf{a}|  +t-1))+q^{r(1-g)}P_3(q^{-w})\\
        =& q^{r(1-g)}\sum_{\substack{D\geq 0 \\ \deg(D)\leq 2g-2}}q^{\ell(D)(t-1)-\deg(D)((\gamma a_r+\xi )w-((r+1)a_r-|\mathbf{a}|))} \\
        & \hspace{6cm}+q^{d(1-g)}\sum_{\substack{D\geq 0 \\ \deg(D)> 2g-2}}q^{- (( \gamma a_r+\xi)w-((r+1)a_r-|\mathbf{a}|  +t-1))\deg(D)}\\
        =& q^{r(1-g)}\R_K\left(1-t, \left( \gamma a_r+\xi \right)w- ((r+1)a_r-|\mathbf{a}|) \right).
    \end{split}      
    \end{equation*}
    As in the previous case, using that~$\R_K\left(1-t,\left( \gamma a_r+\xi \right)A_L- ((r+1)a_r-|\mathbf{a}|) \right)$ is positive, we conclude that~$\zeta_{U,L}(s)$ has a pole of order~$b(L)=1$ at~$s=a(L)$.
\end{enumerate}
%Since the condition~$A_L=B_L$ (resp.~$A_L<B_L$, $A_L>B_L$) is equivalent to~$\xi=\left(\frac{t-|\mathbf{a}|}{r+1}\right)\gamma$ (resp.~$\xi<\left(\frac{t-|\mathbf{a}|}{r+1}\right)\gamma$, $\xi>\left(\frac{t-|\mathbf{a}|}{r+1}\right)\gamma$), the desired result follows from the properties stated above. 
This completes the proof of the Theorem~\ref{thm:ThmGeneralversion}.

%\begin{remark}
%    Equation (?) can be used to compute the value of
%    $$\sigma_{U,L}:=\min\left\{\sigma \in \RR\setminus \{a(L)\}: \zeta_{U,L}(s) \text{ is analytic for }\Re(s)>\sigma, \Re(s)\neq a(L) \right\}$$
%    is given as follows:
%\end{remark}

\subsection{Counting rational points of large height}\label{sec counting rationa points large height}

In this section we use Theorems~\ref{thm:ThmGeneralversion} and~\ref{thm:ThmGeneral when ar=0} to deduce, via Theorem~\ref{thm:tauberianThmFF}, asymptotic formulas for the number of points of large height on a Hirzebruch--Kleinschmidt variety~$X$, with respect to height functions~$H_L$ associated to~$L\in \Pic(X)$ big. We start by introducing the following useful definition.

\begin{defi}
Given a Hirzebruch--Kleinschmidt variety~$X$ and~$L\in \Pic(X)$ big, we define
$$\eta_L:=\max\{\eta\in \ZZ^+:\eta^{-1}L\in \Pic(X)\}.$$
In other words, $\eta_L$ is the unique positive integer such that~$\eta_L^{-1}L\in \Pic(X)$ and is primitive.
\end{defi}
Note that, if we write~$L=\gamma h+\xi f$ with~$\{h,f\}$ the basis of~$\Pic(X)$ given in Proposition~\ref{prop effective cone}, then~$\eta_L=\mathrm{g.c.d.}(\gamma,\xi)$. Moreover, it follows from~\eqref{eq height function HL} that the height function~$H_L$ takes values in~$q^{\eta_L\ZZ}$, hence the only interesting values of the counting function~$\tilde{N}(q^M,H_L,X)$ occur when~$M$ is divisible by~$\eta_L$.

In order to state the following theorem, %Motivated by Theorems~\ref{thm:ThmGeneralversion} and~\ref{thm:ThmGeneral when ar=0}, 
we also define, for~$L=\gamma h+\xi f\in \Pic(X)$ big, the number
$$a'(L):=\left\{
\begin{array}{ll}
   \max\left\{A_L-\frac{1}{\gamma},B_L-\frac{1}{\gamma a_r+\xi}\right\},  & \text{if }A_L=B_L,  \\
  \max\left\{A_L,B_L-\frac{1}{\gamma a_r+\xi}\right\},  & \text{if }A_L<B_L,  \\ 
  \max\left\{A_L-\frac{1}{\gamma},B_L\right\}, & \text{if }A_L>B_L. 
\end{array}
\right.$$
Moreover, recalling that~$u$ denotes the unit for the convolution product~$\star$ on functions~$\Div^+(K)\to \RR$ (see Definition~\ref{def convolution}), we extend the definitions of~$N^1_m(D)$ and~$\tilde{N}_m(D)$ given in~\eqref{eq def of Nmn(D)} and~\eqref{eq def of tildeNm(D)}, in order to include the case~$m=0$, by putting~$N^1_0(D):=1$ and~$\tilde{N}_0(D):=u(D)$. Then, for every integer~$m\geq 0$ we have that $(N^1_m,\tilde{N}_m)$ is a $\mu$-couple and~$N_m^1(D)=q^{m\ell(D)}$. Finally, for every integer~$m\geq 0$ we define~$F_m:\Div^+(K)\to \RR$ as
$$F_{m}(D):=\prod_{\substack{v\in \Val(K)\\ v(D)>0}}(1-q^{-mf_v}).$$
In particular, $F_0=u$.

The first main theorem of this section gives an asymptotic formula for the number
$$\tilde{N}(U,H_L,q^M):=\#\{P\in U(K):H_L(P)=q^M\},$$
for~$M\in \eta_L\ZZ$ large, $U$ the good open subset of a Hirzebruch--Kleinschmidt variety~$X=X_d(a_1,\ldots,a_r)$ with~$a_r>0$ and~$L\in \Pic(X)$ big. The case~$a_r=0$ is treated later in this section.

\begin{thm}\label{thm:AsymptoticFormulaGeneralversion}
 Let~$X:=X_d(a_1,\ldots,a_r)$ be a Hirzebruch--Kleinschmidt variety over the global function field~$K=\F_q(\mathcal{C})$ with~$a_r>0$, $U:=U_d(a_1,\ldots,a_r)$ the  good open subset of~$X$ and~$L=\gamma h+\xi f\in \Pic(X)$ big. For every integer~$n\in \eta_L\ZZ$ we put~$n_0:=\frac{n}{\eta_L}$. Then, %there exists a monic polynomial~$Q(M)$ degree~$b(L)-1$ such that 
 for every~$\delta>a'(L)$ we have the estimate
 $$\tilde{N}(U,H_L,q^M)=Q_L(M)q^{a(L)M}+O\left(q^{\delta M}\right) \quad \text{as }M\to \infty,M\in \eta_L\ZZ,$$
 where~$Q_L(M)$ is given as follows:
 \begin{enumerate}
     \item If~$A_L=B_L$, then~$Q_L(M)=C_{L}M+\tilde{C}_L(M)$ with
     $$C_{L}=\frac{q^{(d+2)(1-g)} \eta_L h_K^2}{\zeta_K(t)\zeta_K(r+1)(\gamma a_r+\xi)\gamma (q-1)^2}>0,$$
     and~$\tilde{C}_L(M)\in \RR$ depending only on~$L$ and on~$M_0$ mod~$\gamma_0(\gamma a_r+\xi)_0$.
     %In particular, $C_{L,M}$ is independent of~$M$.
     \item If~$A_L<B_L$, then
     $$Q_L(M)=\frac{q^{(d+2-N_X)(1-g)}  h_K}{\zeta_K\left(t  \right)(q-1)}\sum\limits_{\substack{D\geq 0\\ \deg(D)\equiv M_0\gamma_0'\, (\gamma a_r+\xi)_0}}(\tilde{N}_{N_X-1}\star F_{r+1-N_X})(D) q^{-\deg(D)\left(\gamma B_{L}-r+N_X-1\right)},$$
     where~$\gamma_0'$ denotes the multiplicative inverse of~$\gamma_0$ mod~$(\gamma a_r+\xi)_0$ and the sum is over all divisors~$D\geq 0$ satisfying~$\deg(D)\equiv M_0\gamma_0'\, \text{mod }(\gamma a_r+\xi)_0$. In particular, $Q_L(M)>0$ and it depends only on~$L$ and on~$M_0$ mod~$(\gamma a_r+\xi)_0$.
     \item If~$A_L>B_L$, then
     $$Q_L(M)=\frac{ q^{(r+1)(1-g)} h_K}{\zeta_K(r+1)(q-1)}
     \sum_{\substack{D\geq 0\\ \deg(D)\equiv M_0(\gamma a_r+\xi)_0'\, (\gamma_0) }}\tilde{N}_{t-1}(D)q^{-\deg(D)(( \gamma a_r+\xi)A_{L}- ((r+1)a_r-|\mathbf{a}|))},$$
     where~$(\gamma a_r+\xi)_0'$ denotes the multiplicative inverse of~$(\gamma a_r+\xi)_0$ mod~$\gamma_0$ and the sum is over all divisors~$D\geq 0$ satisfying~$\deg(D)\equiv M_0(\gamma a_r+\xi)_0'\, \text{mod }\gamma_0$. In particular, $Q_L(M)>0$ and it depends only on~$L$ and on~$M_0$ mod~$\gamma_0$.
 \end{enumerate}
\end{thm}

The proof of Theorem~\ref{thm:AsymptoticFormulaGeneralversion} is given below, after the following lemma.

\begin{lem}\label{lem positivity of exponential sums}
    Let~$a,b\in \ZZ$ with~$a\leq 0\leq b$ and~$s\in \CC$ with~$\Re(s)>1+b-a$. Then, the following properties hold:
    \begin{enumerate}
        \item We have
    \begin{equation}\label{eq computation RK divided by zetaK}
      \frac{\R_K(a,s-b)}{\zeta_K(s)}=\sum_{D\geq 0}(\tilde{N}_{|a|}\star F_{b})(D)q^{-(s-b)\deg(D)}.
    \end{equation}
    \item If~$(a,b)\neq (0,0)$ then for every choice of integers~$M,N$ with~$N\geq 1$ we have
    \begin{equation}\label{eq sum over arithmetic progression is positive}
      \sum_{\substack{D\geq 0\\ \deg(D)\equiv M \, (N)}}(\tilde{N}_{|a|}\star F_{b})(D)q^{-(s-b)\deg(D)}>0,  
    \end{equation}
    where the sum is over all divisors~$D\geq 0$ satisfying~$\deg(D)\equiv M$ mod~$N$.
    \end{enumerate}
    %In particular:
    %\begin{equation}\label{eq expansion of RK over zetaK}
    %   \frac{\R_K(a,s-b)}{\zeta_K(s)}=\sum_{n=0}^{\infty}c_n(a,b)q^{-ns} 
    %\end{equation}
    %with~$c_{n}(a,b)\in \NN_0$, and if~$(a,b)\neq (0,0)$ then~$c_{n}(a,b)>0$ for~$n$ large enough depending only on~$K$.
\end{lem}
\begin{proof}
By definition of~$\R_K$ and Corollary~\ref{cor inverse of zetaK converges absolutely}, we have\footnote{The hypothesis~$\Re(s)>1+b-a$ ensures that all the series appearing in this proof are absolutely convergent.}
\begin{eqnarray*}
\R_K(a,s-b)%&=&\sum_{D\geq 0}q^{-a\ell(D)-(s+b)\deg(D)}\\
&=& \sum_{D\geq 0}q^{-a\ell(D)}q^{-(s-b)\deg(D)}\\
&=& \left(\sum_{D\geq 0}N^1_{|a|}(D)q^{-(s-b)\deg(D)}\right)\left(\sum_{D\geq 0}\mu(D)q^{-(s-b)\deg(D)}\right)\zeta_K(s-b).
\end{eqnarray*}
This implies
\begin{equation}\label{eq computation RK over zetaK 1}
\frac{\R_K(a,s-b)}{\zeta_K(s)}=\left(\sum_{D\geq 0}\tilde{N}_{|a|}(D)q^{-(s-b)\deg(D)}\right)\frac{\zeta_K(s-b)}{\zeta_K(s)}.
\end{equation}
Now, using the Euler product expansion of~$\zeta_K(s)$, we have
$$\frac{\zeta_K(s-b)}{\zeta_K(s)}=\prod_{v\in \Val(K)}\frac{1-q^{-sf_v}}{1-q^{-(s-b)f_v}}=\prod_{v\in \Val(K)}\left(1+(1-q^{-bf_v})\sum_{n=1}^{\infty}q^{-(s-b)f_vn}\right).$$
This implies
\begin{equation}\label{eq computation zetaKs-b over zetaKs}
\frac{\zeta_K(s-b)}{\zeta_K(s)}=\sum_{D\geq 0}F_b(D)q^{-(s-b)\deg(D)}.    
\end{equation}
Then, identity~\eqref{eq computation RK divided by zetaK} follows from~\eqref{eq computation RK over zetaK 1} together with~\eqref{eq computation zetaKs-b over zetaKs}.
Now, we rewrite the sum in~$\eqref{eq sum over arithmetic progression is positive}$ as
$$\sum_{\substack{n\geq 0 \\ n\equiv M (N)}}c_n(a,b)q^{-ns}, \text{ where } c_n(a,b):=q^{nb}\sum_{\substack{D\geq 0 \\ \deg(D)=n}}(\tilde{N}_{|a|}\star F_b)(D).$$
By definition of~$\tilde{N}_{|a|}$ and~$F_b$, we have~$\tilde{N}_{|a|}(D), F_b(D)\geq 0$ for all~$D\in \Pic^+(K)$. This shows that~$c_n(a,b)\geq 0$. Now, assume~$b>0$. Then~$q^{nb}F_b(D)\geq 1$ for all~$D\in \Pic^+(K)$ with~$\deg(D)\leq n$, hence using that~$\tilde{N}_{|a|}(0)\geq 1$ we get
$$c_n(a,b)\geq q^{nb}\sum_{\substack{D\geq 0 \\ \deg(D)=n}}F_b(D)\geq \#\{D\in \Div^+(K):\deg(D)=n\}\geq \#\{v\in \Val(K):f_v=n\}.$$
Similarly, if~$b=0$ and~$a<0$ then, recalling that~$F_0=u$ and~$\tilde{N}_{1}=N_1^1\star \mu$, we get
$$c_n(a,b)=\sum_{\substack{D\geq 0 \\ \deg(D)=n}}\tilde{N}_{|a|}(D)\geq \sum_{\substack{D\geq 0 \\ \deg(D)=n}}\tilde{N}_{1}(D) \geq \sum_{\substack{v\in \Val(K) \\ f_v=n}}\tilde{N}_{1}(v)\geq  \sum_{\substack{v\in \Val(K) \\ f_v=n}}(q^{\ell(v)}-q).$$
%=\#\{x\in K:\deg((x)_{\infty})=n\}.$$
Then, since the number of valuations~$v\in \Val(K)$ with~$f_v=n$ is positive for large enough~$n$ depending only on~$K$ (see, e.g., \cite[Theorem~5.12]{Ros02}), and~$\ell(v)>1$ if~$f_v>\max\{g,2g-2\}$ by the Riemann--Roch theorem, we conclude that~$c_n(a,b)>0$ for large enough~$n$ provided~$(a,b)\neq (0,0)$. This completes the proof of the lemma.
\end{proof}

\begin{remark}
When~$(a,b)=(0,0)$ we have~$(\tilde{N}_{|a|}\star F_{b})(D)=u(D)$ and~\eqref{eq sum over arithmetic progression is positive} holds only when~$M\equiv 0$ mod~$N$.
\end{remark}

\begin{proof}[Proof of Theorem~\ref{thm:AsymptoticFormulaGeneralversion}]
Put~$L_0:=\eta_L^{-1}L$ and let~$M\in \eta_L\ZZ$. Then, we have~$\tilde{N}(U,H_L,q^M)=\tilde{N}(U,H_{L_0},q^{M_0})$. Now, applying Theorems~\ref{thm:ThmGeneralversion} and~\ref{thm:tauberianThmFF} to~$\zeta_{U,H_{L_0}}(s)$ we get that, for every~$\delta_0\in \, ]a'(L_0),a(L_0)[$ we have the estimate
$$\tilde{N}(U,H_{L_0},q^{M_0})=\log(q)^{b(L_0)}Q_0(M_0)q^{a(L_0)M_0}+O\left(q^{\delta_0 M_0}\right) \quad \text{as }M_0\to \infty,M_0\in \ZZ,$$
where~$Q_{0}(M_0)$ is as follows:
\begin{enumerate}
    \item If~$A_{L_0}=B_{L_0}$, then~$Q_0(M_0)=C_1M_0+C_2(M_0)$ with
    $$C_1=\lim_{s\to 2}(s-2)^2\zeta_{U,H_{L_0}}(s)=\frac{q^{(d+2)(1-g)}  h_K^2}{\zeta_K(t)\zeta_K(r+1)(\gamma a_r+\xi)_0\gamma_0 (q-1)^2\log(q)^2},$$
    and~$C_2(M_0)\in \RR$ depending on~$M_0$ mod~$\gamma_0(\gamma a_r+\xi)_0$.
    %In particular, $c_{M_0}$ is independent of~$M_0$.
    \item If~$A_{L_0}<B_{L_0}$, then
     $$Q_0(M_0)=\frac{q^{(d+2-N_X)(1-g)} h_K}{\zeta_K\left(t  \right)(\gamma a_r+\xi)_0(q-1)\log(q)}\sum_{k=0}^{(\gamma a_r+\xi)_0-1} e^{\frac{2\pi i M_0 k}{(\gamma a_r+\xi)_0}} \frac{\R_K\left(1-N_X,\gamma_0 w_k-r+N_X-1\right)}{\zeta_K\left(\gamma_0 w_k\right)}$$
     where~$w_k:=B_{L_0}+\frac{2\pi i k}{(\gamma a_r+\xi)_0\log(q)}$. Moreover, using Lemma~\ref{lem positivity of exponential sums}(1) we have
     \begin{eqnarray*}
      \frac{\R_K\left(1-N_X,\gamma_0 w_k-r+N_X-1\right)}{\zeta_K\left(\gamma_0 w_k\right)}
      &=&\sum_{D\geq 0}(\tilde{N}_{N_X-1}\star F_{r+1-N_X})(D)e^{-\deg(D)\frac{2\pi i \gamma_0 k}{(\gamma a_r+\xi)_0}}\\
      & & \hspace{1cm} \times q^{-\deg(D)\left(\gamma_0B_{L_0}-r+N_X-1\right)}.   
     \end{eqnarray*}
     This implies
     $$Q_0(M_0)=\frac{q^{(d+2-N_X)(1-g)} h_K}{\zeta_K\left(t  \right)(q-1)\log(q)}\sum_{\substack{D\geq 0\\ \deg(D)\equiv M_0\gamma_0'\, ((\gamma a_r+\xi)_0)}}(\tilde{N}_{N_X-1}\star F_{r+1-N_X})(D) q^{-\deg(D)\left(\gamma_0B_{L_0}-r+N_X-1\right)},$$
     and this value is positive by Lemma~\ref{lem positivity of exponential sums}(2).
     \item If~$A_{L_0}>B_{L_0}$, then
     $$Q_0(M_0)=\frac{ q^{(r+1)(1-g)} h_K}{\zeta_K(r+1)\gamma_0 (q-1)\log(q)}
     \sum_{k=0}^{\gamma_0-1} e^{\frac{2\pi i M_0 k}{\gamma_0}} \frac{\R_K\left(1-t, ( \gamma a_r+\xi)_0w_k- ((r+1)a_r-|\mathbf{a}|) \right)}{\zeta_K((\gamma a_r+\xi )_0w_k-((r+1)a_r-|\mathbf{a}|))}$$
     where~$w_k:=A_{L_0}+\frac{2\pi i k}{ \gamma_0 \log(q)}$ in this case. As in the previous case, we use Lemma~\ref{lem positivity of exponential sums}(1) to write
     \begin{eqnarray*}
     \frac{\R_K\left(1-t, ( \gamma a_r+\xi)_0w_k- ((r+1)a_r-|\mathbf{a}|) \right)}{\zeta_K((\gamma a_r+\xi )_0w_k-((r+1)a_r-|\mathbf{a}|))}&=&\sum_{D\geq 0}\tilde{N}_{t-1}(D)e^{-\deg(D)\frac{2\pi i ( \gamma a_r+\xi)_0 k}{\gamma_0}}\\
      & & \hspace{1cm} \times q^{-\deg(D)(( \gamma a_r+\xi)_0A_{L_0}- ((r+1)a_r-|\mathbf{a}|))}
     \end{eqnarray*}
     and get
     $$Q_0(M_0)=\frac{ q^{(r+1)(1-g)} h_K}{\zeta_K(r+1)(q-1)\log(q)}
     \sum_{\substack{D\geq 0\\\deg(D)\equiv M_0(\gamma a_r+\xi)_0'\, (\gamma_0) }}\tilde{N}_{t-1}(D)q^{-\deg(D)(( \gamma a_r+\xi)_0A_{L_0}- ((r+1)a_r-|\mathbf{a}|))}.$$
     By Lemma~\ref{lem positivity of exponential sums}(2) this value is positive.     
\end{enumerate}
Then, the desired results follow from these computations and the fact that~$A_{L_0}=\eta_L A_L,B_{L_0}=\eta_L B_L, b(L_0)=b(L),a(L_0)=\eta_L a(L), a'(L_0)=\eta_L a'(L)$ and putting 
%~$Q_0(M_0)=\eta_L^{1-b(L)}Q(M)$ with~$Q(M)$ a monic polynomial of degree~$b(L)-1$ and putting
$Q_L(M):=\log(q)^{b(L)}Q_0(M_0)$. This completes the proof of the theorem.
\end{proof}

We now treat the easier case where~$a_r=0$. In this setting, for~$L=\gamma h+\xi f\in \Pic(X)$ big, we define
$$a''(L):=\left\{
\begin{array}{ll}
   \max\left\{\frac{1}{2\gamma},\frac{1}{2\xi}\right\},  & \text{if }A_L=B_L,  \\
  \max\left\{A_L,\frac{1}{2\xi}\right\},  & \text{if }A_L<B_L,  \\ 
  \max\left\{\frac{1}{2\gamma},B_L\right\}, & \text{if }A_L>B_L. 
\end{array}
\right.$$
Also, recall that~$\tilde{N}(\PP^n,q^d):=\{P\in \PP^n(K):H_{\PP^n}(P)=q^d\}$ (see Section~\ref{sec height zeta function of projective spaces}).

 \begin{thm}\label{thm:AsymptoticFormulaGeneralversion when ar=0}
  Let~$X\simeq \PP^r\times \PP^{t-1}$ be a Hirzebruch--Kleinschmidt variety over the global function field~$K=\F_q(\mathcal{C})$ with~$a_r=0$ and let~$L=\gamma h+\xi f\in \Pic(X)$ big. Then, for every~$\delta\in \,]a''(L),a(L)[$ we have the estimate
 $$\tilde{N}(X,H_L,q^M)=Q_L(M)q^{a(L)M}+O\left(q^{\delta M}\right) \quad \text{as }M\to \infty,M\in \eta_L\ZZ,$$
 where, using the same notation as in Theorem~\ref{thm:AsymptoticFormulaGeneralversion}, the function~$Q_L(M)$ is given as follows:
 \begin{enumerate}
     \item If~$A_L=B_L$, then~$Q_L(M)=C_LM+\tilde{C}_L(M)$ with
     $$C_{L}=\frac{q^{(d+2)(1-g)} \eta_L h_K^2}{\zeta_K(t)\zeta_K(r+1)\xi\gamma (q-1)^2},$$
     and~$\tilde{C}_L(M)\in \RR$ depending only on~$L$ and on~$M_0$ mod~$\gamma_0 \xi_0$.
     %In particular, $C_{L,M}$ is independent of~$M$.
     \item If~$A_L<B_L$, then
     $$Q_L(M)=\frac{q^{t(1-g)} h_K}{\zeta_K\left(t  \right)(q-1)}\sum\limits_{\substack{d\geq 0\\ d\equiv M_0\gamma_0'\, (\xi_0)}}\tilde{N}(\PP^r,q^d) q^{-d\gamma B_L}.$$
     In particular, $Q_L(M)>0$ and it depends only on~$L$ and on~$M_0$ mod~$\xi_0$.
     \item If~$A_L>B_L$, then
     $$Q_L(M)=\frac{ q^{(r+1)(1-g)}h_K}{\zeta_K(r+1)(q-1)}
     \sum\limits_{\substack{d\geq 0\\ d\equiv M_0\xi_0'\, (\gamma_0)}}\tilde{N}(\PP^{t-1},q^d) q^{-d\xi A_L}.$$
     In particular, $Q_L(M)>0$ and it depends only on~$L$ and on~$M_0$ mod~$\gamma_0$.
 \end{enumerate}
\end{thm}
\begin{proof}
    The result follows from Theorems~\ref{thm:ThmGeneral when ar=0} and~\ref{thm:tauberianThmFF}, and is analogous to the proof of Theorem~\ref{thm:AsymptoticFormulaGeneralversion}. Hence, we omit the details for brevity. Here, we only mention that, in the case~$A_L\neq B_L$, the non-vanishing of~$Q_L(M)$ follows from the fact that~$\tilde{N}(\PP^{n},q^d)>0$ for large~$d\in \ZZ$ depending only on~$n$ and~$K$, because of~\eqref{eq Wan's estimate}. This proves the theorem.
\end{proof}

\subsection{Decomposition of Hirzebruch--Kleinschmidt varieties}\label{sec decomposition of HK}

In this section we explain in detail the decomposition~\eqref{eq decomp HK intro} of Hirzebruch--Kleinschmidt varieties mentioned in the introduction.

First, note that for~$1\leq r'\leq r$ and~$2\leq t'\leq t$, there is a natural embedding %of Hirzebruch--Kleinschmidt varieties
\begin{equation}\label{eq embedding smaller HK in X}
 X_{t'+r'-1}(a_1,\ldots,a_{r'})\hookrightarrow X_{d}(a_1,\ldots,a_{r})   
\end{equation}
%$$X_{d-1}(a_1,\ldots,a_{r-1})\hookrightarrow X_{d}(a_1,\ldots,a_{r}),$$
given in homogeneous coordinates by
$$([x_0: (x_{ij})_{i\in I_{t'},j\in I_{r'}}],[y_1:\ldots:y_{t'}])\mapsto ([x_0: (\tilde{x}_{ij})_{i\in I_t,j\in I_r}],[y_1:\ldots:y_{t'}:0:\ldots:0]),$$
where
$$\tilde{x}_{ij}:=\left\{
\begin{array}{ll}
x_{1j}     & \textup{if }j\leq r', a_j=0, i\leq t,  \\
x_{ij} &  \textup{if }j\leq r',a_j>0,i\leq r', \\
0     & \textup{otherwise}. 
\end{array}
\right.$$
%In other words, we enlarge the matrix~$(x_{ij})$ of size~$t'\times r'$ to a matrix of size~$t\times r$ by adding zero columns on the right and zero rows on the bottom, and similarly we enlarge~$[y_1:\ldots:y_{t'}]$ by adding zeroes on the right.

We also consider the natural embeddings
\begin{equation}\label{eq embedding P(r') in X}
\PP^{r'} \hookrightarrow X_{d}(a_1,\ldots,a_r), \quad [z_0:\ldots:z_{r'}]\mapsto ([z_0:(\tilde{z}_{ij})_{i\in I_t,j\in I_r}], [1:0:\ldots:0]),
\end{equation}
for~$1\leq r'\leq r$, where
$$\tilde{z}_{ij}:=\left\{\begin{array}{ll}
  z_j   & \textup{if } j\leq r', a_j=0, i\leq r, \textup{ or } j\leq r', a_j>0, i=1, \\
  0  & \textup{otherwise},
\end{array}\right.$$
and
\begin{equation}\label{eq embedding P(t-1) in X}
  \PP^{t-1}\hookrightarrow X_{d}(a_1,\ldots,a_r), \quad [y_1:\ldots:y_t]\mapsto ([1:(0)_{i\in I_t,j\in I_r}], [y_1:\ldots:y_t]).
\end{equation}
Note that the images of~$\PP^{t-1}$ and~$\PP^{r}$ inside~$X_{d}(a_1,\ldots,a_r)$ under these embeddings intersect exactly at the point~$P_0:=([1:(0)_{i\in I_t,j\in I_r}], [1:0:\ldots:0])$.

For~$1\leq r'\leq r$ and~$2\leq t'\leq t$, let us identify~$ U_{t'+r'-1}(a_1,\ldots,a_{r'})$ and~$X_{t'+r'-1}(a_1,\ldots,a_{r'})$ with their images under~\eqref{eq embedding smaller HK in X}. Similarly, let us identify~$\PP^{r'}$ and~$\PP^{t-1}$ with their images under~\eqref{eq embedding P(r') in X} and~\eqref{eq embedding P(t-1) in X}, which we denote by~$\PP^{r'}_1$ and~$\PP^{t-1}_2$, respectively.

\begin{lem}\label{lem decomposition HK over FF}
Let~$X_d(a_1,\ldots,a_r)$ be a Hirzebruch--Kleinschmidt variety with~$a_r>0$. If $r\geq 2$, then we have the disjoint union decomposition
\begin{equation}\label{eq decomposition2 HK}
 X_d(a_1,\ldots,a_r)=X_{d-1}(a_1,\ldots,a_{r-1})\sqcup (\PP^r_1\setminus\PP^{r-1}_1 ) \sqcup \bigg(\bigsqcup_{2\leq t'\leq t}U_{t'+r-1}(a_1,\ldots,a_{r})\bigg).
\end{equation}
If $r=1$, then we have the  disjoint union decomposition
\begin{equation}\label{eq decomposition2 HK r=1}
     X_d(a_1)=\PP^{t-1}_2\sqcup  (\PP^1_1\setminus \{P_0\} ) \sqcup  \bigg(\bigsqcup_{2\leq t'\leq t}U_{t'}(a_1)\bigg),
\end{equation}
where~$P_0=([1:(0)_{i\in I_t,j\in I_r}], [1:0:\ldots:0])$.
\end{lem}
\begin{proof}
Assume $r\geq 2$ and let~$x=([x_0: (x_{ij})_{i\in I_{t},j\in I_{r}}],[y_1:\ldots:y_{t}])\in X_d(a_1,\ldots,a_r)$. If~$x_{ir}=0$ for all~$i\in I_t$, then~$x\in X_{d-1}(a_1,\ldots,a_{r-1})$. Assume this is not the case, and let~$t':=\max\{i\in I_t:x_{ir}\neq 0\}$. From equation~\eqref{equations HK} with~$j=r$ and~$m=t'$ we see that~$y_n=0$ for all~$n>t'$ (because~$a_r>0$). At the same time we see that having~$y_{t'}=0$ would imply~$y_n=0$ also for all~$n\neq t'$, which is impossible since~$[y_1:\ldots:y_{t}]\in \PP^{t-1}$. Hence, $y_{t'}\neq 0$. Now, from equation~\eqref{equations HK} with~$n=t'$ and~$m=i>t'$ we see that
$$x_{ij}=0 \text{ for all }i\in I_t \text{ with }i>t' \text{ and all }j\in I_r \text{ with }a_j>0.$$
If~$t'\geq 2$, then this shows that~$x\in U_{t'+r-1}(a_1,\ldots,a_{r})$. If~$t'=1$, then we conclude that~$x\in \PP^r_1\setminus \PP^{r-1}_1$. 
This proves that
\begin{equation*}
 X_d(a_1,\ldots,a_r)=X_{d-1}(a_1,\ldots,a_{r-1})\cup  (\PP^r_1\setminus \PP^{r-1}_1) \cup \bigg(\bigcup_{2\leq t'\leq t}U_{t'+r-1}(a_1,\ldots,a_{r})\bigg).
\end{equation*}
Moreover, these unions are disjoint by construction.  This proves~\eqref{eq decomposition2 HK}. The proof of the decomposition~\eqref{eq decomposition2 HK r=1} in the case~$r=1$ is completely analogous. This proves the lemma. 
\end{proof}

\begin{remark}
    In practice, we identify the components~$\PP^r_1\setminus \PP^{r-1}_1$ and~$\PP^1_1\setminus \{P_0\}$ in~\eqref{eq decomposition2 HK} and~\eqref{eq decomposition2 HK r=1} with~$\A^r$ and~$\A^1$, respectively, as done in Example~\ref{ex intro HK threefold} in the introduction.
\end{remark}

In the case~$a_r=0$ we have the following simpler decomposition.

\begin{lem}\label{lem decomposition HK over FF case a_r=0}
Let~$X_d(a_1,\ldots,a_r)\simeq \PP^{r}\times \PP^{t-1}$ be a Hirzebruch--Kleinschmidt variety with~$a_r=0$. If $r\geq 2$, then we have the disjoint union decomposition
\begin{equation*}%\label{eq decomposition HK r>1 case a_r=0}
 X_d(a_1,\ldots,a_r)=X_{d-1}(a_1,\ldots,a_{r-1})\sqcup U_{d}(a_1,\ldots,a_{r}).
\end{equation*}
If $r=1$, then we have the  disjoint union decomposition
\begin{equation*}%\label{eq decomposition2 HK r=1 case a_r=0}
     X_d(a_1)=\PP^{t-1}_2\sqcup  U_d(a_1).
\end{equation*}
\end{lem}
\begin{proof}
Write~$x\in X_d(a_1,\ldots,a_r)\simeq \PP^{r}\times \PP^{t-1}$ as~$x=([x_0:x_1:\ldots :x_r],[y_1,\ldots,y_t])$. Then, the components in the above decompositions correspond to the cases~$x_r=0$ and~$x_r\neq 0$. This proves the lemma.
%In the case~$r=1$, the components in~\eqref{eq decomposition HK r=1 case a_r=0} corresponds to the cases~$x_1=0$, $x_1\neq 0$
\end{proof}

Finally, we state the following lemma regarding the restriction of a height function on a Hirzebruch--Kleinschmidt variety, to each component in the above decompositions.
%Finally, the statements about the various restrictions of the height function~$H_L$ follow directly from~\eqref{eq height function HL}.

\begin{lem}\label{lem restriction of heights}
    Given~$L=\gamma h+\xi f\in \Pic(X_d(a_1,\ldots,a_r))$ big, we have the following properties:
\begin{enumerate}
\item For all~$1\leq r'\leq r$, the restriction of~$H_L$ to~$\PP^{r'}_1\simeq \PP^{r'}$ corresponds to~$H_{\PP^{r'}}^\gamma$.
    \item The restriction of~$H_L$ to~$\PP^{t-1}_2\simeq \PP^{t-1}$ corresponds to~$H_{\PP^{t-1}}^\xi$.
    \item For~$1\leq r'\leq r$ and~$2\leq t'\leq t$, the restriction of~$H_L$ to~$ U_{t'+r'-1}(a_1,\ldots,a_{r'})$ corresponds to~$H_{L'}$ where~$L':=\gamma h'+\xi f'$ with~$\{h',f'\}$ the basis of~$\Pic(X_{t'+r'-1}(a_1,\ldots,a_{r'}))$ given in Proposition~\ref{prop effective cone}.
\end{enumerate}
\end{lem}
\begin{proof}
    These statements are a direct consequence of~\eqref{eq height function HL}.
\end{proof}

\subsection{The anticanonical height zeta function - Proof of Theorem~\ref{thm:MainThmAnticanonicalCaseoverFF} and Corollary~\ref{cor:MainThmAnticanonicalCaseoverFF}}\label{sec proof of main thm for anticanonical class}

We now prove Theorem~\ref{thm:MainThmAnticanonicalCaseoverFF} and Corollary~\ref{cor:MainThmAnticanonicalCaseoverFF} from the introduction. We choose~$L=\eta_X^{-1}(-K_X)=\gamma h+\xi f$ with~$\gamma=\frac{r+1}{\eta_X}$ and~$\xi=\frac{t-|\mathbf{a}|}{\eta_X}$ (see Proposition~\ref{prop effective cone}(2)). Hence, $a(L)=A_L=B_L=\eta_X$ and~$b(L)=2$. Assume~$a_r>0$. Then, the absolute convergence, rationality and holomorphicity statements about~$\zeta_{U,-K_X}(s)=\zeta_{U,L}(\eta_X s)$ %for~$\Re(s)>1$, the fact that~$\zeta_{U,-K_X}(s)$ is a rational function in~$q^{-\eta_X s}$ 
follow directly from Theorem~\ref{thm:ThmGeneralversion}. We also have
$$\lim_{s\to 1}(s-1)^2\zeta_{U,-K_X}(s)=\eta_X^{-2}\lim_{s\to a(L)}(s-a(L))^2\zeta_{U,L}(s)$$
and this value is exactly the one given in~\eqref{eq constant main thm anticanonical}. Moreover, since~$\zeta_{U,L}(s)$ has at most simple poles at each~$s\in \mathcal{A}_L\cup \mathcal{B}_L$, we conclude that~$\zeta_{U,-K_X}(s)$ has at most simple poles on~$\Re(s)=1,s\not \in  \left\{1+\frac{2\pi i m}{\eta_X \log(q)}:m\in \ZZ\right\}$. This proves the result in the case $a_r>0$. Now, assume~$a_r=0$. Then, by Lemmas~\ref{lem decomposition HK over FF case a_r=0} and~\ref{lem restriction of heights} we have
$$\zeta_{U,-K_X}(s)= \left\{ 
\begin{array}{ll}
 \zeta_{X,L}(\eta_X s)-\zeta_{X_{d-1}(a_1,\ldots,a_{r-1}),L'}(\eta_X s)    &  \text{if }r>1,\\
 \zeta_{X,L}(\eta_X s)-\zeta_{\PP^{t-1}}\left(t s\right)   &  \text{if }r=1,
\end{array}
\right.$$
where in the case~$r>1$ we put~$L':=\gamma h'+\xi f'$ with~$\{h',f'\}$ the basis of~$\Pic(X_{d-1}(a_1,\ldots,a_{r-1}))$ given in Proposition~\ref{prop effective cone}. In that case, we apply Theorem~\ref{thm:ThmGeneral when ar=0} to both~$\zeta_{X,L}(s)$ and~$\zeta_{X_{d-1}(a_1,\ldots,a_{r-1}),L'}(s)$ in order to deduce the desired properties of~$\zeta_{U,-K_X}(s)$. Since~$A_{L'}=\frac{r}{r+1}\eta_X$ and~$B_{L'}=\eta_X$, we get~$a(L')=\eta_X$ and~$b(L')=1$. We deduce that both~$\zeta_{X,L}(\eta_X s)$ and~$\zeta_{X_{d-1}(a_1,\ldots,a_{r-1}),L'}(\eta_X s)$ converge absolutely for~$\Re(s)>1$, are rational functions in~$q^{-\eta_X s}$ and~$\zeta_{X_{d-1}(a_1,\ldots,a_{r-1}),L'}(\eta_X s)$ has only simple poles for~$\Re(s)=1$ while~$\zeta_{X,L}(\eta_X s)$ has a double pole at~$s=1$ and at most simple poles for~$\Re(s)=1,s\not \in \left\{1+\frac{2\pi i m}{\eta_X \log(q)}:m\in \ZZ\right\}$. Moreover, we see that~$\zeta_{U,-K_X}(s)$ is holomorphic for
$$\Re(s)>\frac{r}{r+1},\Re(s)\neq 1.$$
Since~$\frac{r}{r+1}\leq 1-\min\left\{\frac{1}{r+1},\frac{1}{t}\right\}$, we conclude that~$\zeta_{U,-K_X}(s)$ satisfies all the desired analytic properties. Additionally, we have
$$\lim_{s\to 1}(s-1)^2\zeta_{U,-K_X}(s)=\lim_{s\to 1}(s-1)^2\zeta_{X,-K_X}(s)=\eta_X^{-2}\lim_{s\to a(L)}(s-a(L))^2\zeta_{X,L}(s)$$
and this values is the one given in~\eqref{eq constant main thm anticanonical}. This proves Theorem~\ref{thm:MainThmAnticanonicalCaseoverFF} in the case~$r\geq 1$. The case~$r=1$ follows from Theorems~\ref{thm:ThmGeneral when ar=0} and~\ref{thm height zeta function projective} in an analogous way. We omit the details for brevity. This completes the proof of Theorem~\ref{thm:MainThmAnticanonicalCaseoverFF}. 

Finally, Corollary~\ref{cor:MainThmAnticanonicalCaseoverFF} is a direct consequence of Theorems~\ref{thm:MainThmAnticanonicalCaseoverFF} and~\ref{thm:tauberianThmFF}.

%We now turn to the proof of Corollary~\ref{cor:MainThmAnticanonicalCaseoverFF}. In the case~$a_r>0$ the result follows directly from Theorem~\ref{thm:AsymptoticFormulaGeneralversion} choosing~$L=(r+1)h+((r+1)a_r-|\mathbf{a}|+t)$. In the case~$a_r=0$ we write
%$$\tilde{N}(U,H_{-K_X},q^M)=\left\{ 
%\begin{array}{ll}
%\tilde{N}(X,H_{-K_X},q^M)-\tilde{N}(X_{d-1}(a_1,\ldots,a_r),H_{L'},q^M)    &  \text{if }r>1,\\
%\tilde{N}(X,H_{-K_X},q^M)-\tilde{N}(\PP^{t-1},q^{\frac{M}{t}})  &  \text{if }r=1,
%\end{array}
%\right.$$
%and obtain the desired estimate from Theorem~\ref{thm:AsymptoticFormulaGeneralversion when ar=0} and~\eqref{eq Wan's estimate}. This proves Corollary~\ref{cor:MainThmAnticanonicalCaseoverFF}.

\section{Hirzebruch surfaces}\label{sec ex Hirzebruch}

In this last section we apply our results to the case of Hirzebruch surfaces~$X:=X_2(a)$ with~$a>0$ an integer, and general~$L=\gamma h+\xi f\in \Pic(X)$ big. In this context, Lemmas~\ref{lem decomposition HK over FF} and~\ref{lem restriction of heights} give
\begin{equation}\label{eq decomposition Hirzebruch}
    X\simeq \PP^1 \sqcup \A^1\sqcup U_2(a),
\end{equation}
with associated height functions~$H_{\PP^1}^{\xi}, H_{\PP^1}^{\gamma}$ and~$H_L$. Moreover, we have
$$A_L=\frac{2}{\gamma}, \quad B_L=\frac{a+2}{\gamma a+\xi}.$$
Using the same notation as in Section~\ref{sec counting rationa points large height}, we get
$$\tilde{N}(U,H_L,q^M)=Q_L(M)q^{a(L)M}+O\left(q^{\delta_1 M}\right) \quad \text{as }M\to \infty,M\in \eta_L\ZZ,$$
for every~$\delta_1>a'(L)$, where~$Q_L(M)=C_1M+C_0(M)$ with~$C_0(M)\in \RR$ depending only on~$L$ and~$M_0$ mod~$\gamma_0(\gamma a +\xi)_0$ and
$$C_1=\frac{q^{4(1-g)} \eta_L h_K^2}{\zeta_K(2)^2(\gamma a+\xi)\gamma (q-1)^2}$$
if~$A_L=B_L$, or~$Q_L(M)>0$ depending only on~$L$ and~$M_0$ mod~$(\gamma a +\xi)_0$ (resp.~mod~$\gamma_0$) if~$A_L<B_L$ (resp.~$A_L>B_L$). 

Now, if~$\xi\leq 0$ then the component~$\PP^1$ in~\eqref{eq decomposition Hirzebruch} has no rational points of large height, and if~$\xi>0$ then we have
$$\tilde{N}(\PP^1,H_{\PP^1}^\xi,q^M)=C_2q^{\frac{2M}{\xi}}+O\left(q^{\delta_2 M}\right) \quad \text{as }M\to \infty,M\in \xi \ZZ,$$
for every~$\delta_2>\frac{1}{2\xi}$, where
$$C_2=\frac{h_Kq^{2(1-g)}}{\zeta_K(2)(q-1)}.$$

Finally, the component~$\A^1$ in~\eqref{eq decomposition Hirzebruch} satisfies
$$\tilde{N}(\A^1,H_{\PP^1}^\gamma,q^M)=C_2q^{\frac{2M}{\gamma}}+O\left(q^{\delta_3 M}\right) \quad \text{as }M\to \infty,M\in \gamma \ZZ,$$
for every~$\delta_3>\frac{1}{2\gamma}$.

%and  the analytic properties of~$\zeta_{U,L}(s)$ can be extracted from Theorem~\ref{thm:ThmGeneralversion}. In particular, $\zeta_{U,L}(s)$ converges absolutely for~$\Re(s)>\max\{A_L,B_L\}$, it has a pole at~$s=\max\{A_L,B_L\}$, and this pole  is of order 2 if~$A_L=B_L$ and of order 1 otherwise. For the zeta function~$\zeta_{\A^1}(\gamma s)$, we know by Theorem~\ref{thm height zeta function projective} that it converges absolutely for~$\Re(s)>A_L$ and it has a simple pole at~$s=A_L$. Finally, the zeta function~$\zeta_{\PP^1}(\xi s)$ has no finite abscissa of absolute convergence if~$\xi\leq 0$, and it converges absolutely for~$\Re(s)>\frac{2}{\xi}$ with a simple pole at~$s=\frac{2}{\xi}$ if~$\xi>0$. This allows for a complete analysis of the analytic properties of the height zeta function of each component in~\eqref{eq decomposition Hirzebruch}.

In the case~$L=-K_X$, we have~$\gamma=2,\xi=2-a,a(L)=1$ and~$\eta_L=\mathrm{g.c.d.}(2,a)$. As in Example~\ref{ex intro HK threefold}, we can  distinguish two cases:
\begin{enumerate}
    \item If~$a=1$ then~$\eta_L=1$ and
   $$\tilde{N}(X,H_{-K_X},q^M)\sim \tilde{N}(\PP^1,H_{\PP^1},q^M)\sim C_2q^{2M}  \quad \text{as }M\to \infty,M\in \ZZ.$$
   \item If~$a\geq 2$ then
   $$\tilde{N}(X,H_{-K_X},q^M)\sim \tilde{N}(U,H_{-K_X},q^M)\sim C_1Mq^{M}  \quad \text{as }M\to \infty,M\in \eta_L\ZZ.$$
\end{enumerate}

%we find it necessary to remove the component~$\PP^1$ in~\eqref{eq decomposition Hirzebruch} in order to verify the analogue of Conjecture~\ref{conj MP} over global function fields. Indeed, the zeta series~$\zeta_{X\setminus \PP^1}(s)= \zeta_{\A^1}(\gamma s)+\zeta_{U,L}(s)$ converges absolutely for~$\Re(s)>1$, it is a rational function in~$q^{-s}$ and satisfies
%$$\lim_{s\to 1}(s-1)^2\zeta_{X\setminus \PP^1}(s)=\lim_{s\to 1}(s-1)^2\zeta_{U,L}(s)=\frac{q^{4(1-g)}h_K^2}{\zeta_K(2)^2(a+2)2(q-1)^2\log(q)^2}.$$

Finally, in order to show an example different from the anticanonical height, let us choose~$a=\gamma=\xi=1$. Then, since~$A_L=2,B_L=\frac{3}{2},a(L)=2,\eta_L=1$, we get
$$\tilde{N}(U,H_L,q^M)\sim C_3q^{2M} \text{ and } \tilde{N}(\PP^1,H_{\PP^1},q^M)\sim\tilde{N}(\A^1,H_{\PP^1},q^M)\sim C_2q^{2M}  \quad \text{as }M\to \infty,M\in \ZZ,$$
where
$$C_3=\frac{ q^{2(1-g)} h_K}{\zeta_K(2)(q-1)}
     \sum_{D\geq 0}\tilde{N}_{1}(D)q^{-3\deg(D)}=\frac{ q^{2(1-g)} h_K \R_K(-1,3)}{\zeta_K(2)\zeta_K(3)(q-1)}.$$
In the special case~$K=\FF_q(T)$ we can use that~$g=0,h_K=1,\R_K(-1,3)=q\zeta_K(2)$ and
$$\zeta_K(s)=\frac{1}{(1-q^{1-s})(1-q^{-s})}$$
(see, e.g., \cite[p.~52]{Ros02}) to compute
\begin{eqnarray*}
    C_2 &=& \frac{q^2-1}{q}, \\
    C_3&=& \frac{(q^2+q+1)(q^2-1)}{q^2}.
\end{eqnarray*}
In particular, $C_3>2C_2$. Hence, in the case of the Hirzebruch surface~$X_2(1)$ over~$K=\FF_q(T)$, with respect to the height function~$H_L$ associated with~$L=h+f\in \Pic(X_2(1))$, our computations show that the component~$U_2(1)$ in~\eqref{eq decomposition Hirzebruch} has asymptotically more rational points of large height than the union~$\PP^1\sqcup \A^1$ of the other two components.
%$$\lim_{s\to 2}(s-2)\zeta_{\PP^1}(s)=\lim_{s\to 2}(s-2)\zeta_{\A^1}(s)=\frac{q^2-1}{q\log(q)},$$
%and 
%$$\lim_{s\to 2}(s-2)\zeta_{U,L}(s)=\frac{(q^2-1)(q^2+q+1)}{q^2\log(q)}.$$
%(see Remark~\ref{rmk after main general thm}(1)). Hence, in the case of the Hirzebruch surface~$X_2(1)$ over~$K=\FF_q(T)$, with respect to the height function~$H_L$ associated with~$L=h+f\in \Pic(X_2(1))$, we can say that the component~$U=U_2(1)$ in~\eqref{eq decomposition Hirzebruch} has asymptotically more rational points of bounded height than the union~$\PP^1\sqcup \A^1$ of the other two components, since the residue of~$\zeta_{U,L}(s)$ at~$s=2$ is larger than the sum of the residues of~$\zeta_{\PP^1}(s)$ and~$\zeta_{\A^1}(s)$.

 \bibliographystyle{alpha}
\bibliography{biblio}

\begin{thebibliography}{ADHL15}

\bibitem[ADHL15]{Arzhantsev_Derenthal_Hausen_Laface_2014CoxRings}
Ivan Arzhantsev, Ulrich Derenthal, J\"urgen Hausen, and Antonio Laface.
\newblock {\em Cox rings}, volume 144 of {\em Cambridge Studies in Advanced
  Mathematics}.
\newblock Cambridge University Press, Cambridge, 2015.

\bibitem[BLRT22]{BLRT22}
Roya Beheshti, Brian Lehmann, Eric Riedl, and Sho Tanimoto.
\newblock Moduli spaces of rational curves on {F}ano threefolds.
\newblock {\em Adv. Math.}, 408:Paper No. 108557, 60, 2022.

\bibitem[BM90]{Bat/Man90}
V.~V. Batyrev and Yu.~I. Manin.
\newblock Sur le nombre des points rationnels de hauteur born\'{e} des
  vari\'{e}t\'{e}s alg\'{e}briques.
\newblock {\em Math. Ann.}, 286(1-3):27--43, 1990.

\bibitem[Bou02]{BOURQUI2002}
David Bourqui.
\newblock Fonction zêta des hauteurs des surfaces de {H}irzebruch dans le cas
  fonctionnel.
\newblock {\em Journal of Number Theory}, 94(2):343--358, 2002.

\bibitem[Bou03]{Bourqui2003VarietiesSplit}
David Bourqui.
\newblock Fonction z\^eta des hauteurs des vari\'et\'es toriques d\'eploy\'ees
  dans le cas fonctionnel.
\newblock {\em J. Reine Angew. Math.}, 562:171--199, 2003.

\bibitem[Bou11a]{Bou2011}
David Bourqui.
\newblock Asymptotic behaviour of rational curves.
\newblock {\em \textup{arXiv:1107.3824}}, 2011.

\bibitem[Bou11b]{Bourqui2011VarietiesNonSplit}
David Bourqui.
\newblock Fonction z\^eta des hauteurs des vari\'et\'es toriques non
  d\'eploy\'ees.
\newblock {\em Mem. Amer. Math. Soc.}, 211(994):viii+151, 2011.

\bibitem[Bou16]{Bou16}
David Bourqui.
\newblock Algebraic points, non-anticanonical heights and the {S}everi problem
  on toric varieties.
\newblock {\em Proc. Lond. Math. Soc. (3)}, 113(4):474--514, 2016.

\bibitem[Bro07]{Bro07}
T.~D. Browning.
\newblock An overview of {M}anin's conjecture for del {P}ezzo surfaces.
\newblock In {\em Analytic number theory}, volume~7 of {\em Clay Math. Proc.},
  pages 39--55. Amer. Math. Soc., Providence, RI, 2007.

\bibitem[BT98]{BT98}
Victor~V. Batyrev and Yuri Tschinkel.
\newblock Manin's conjecture for toric varieties.
\newblock {\em J. Algebraic Geom.}, 7(1):15--53, 1998.

\bibitem[CLS11]{CoxLittleSchenckToricVarieties}
David~A. Cox, John~B. Little, and Henry~K. Schenck.
\newblock {\em Toric varieties}, volume 124 of {\em Graduate Studies in
  Mathematics}.
\newblock American Mathematical Society, Providence, RI, 2011.

\bibitem[Del54]{DelangeTauberianThm}
Hubert Delange.
\newblock G\'en\'eralisation du th\'eor\`eme de {Ikehara}.
\newblock {\em Annales scientifiques de l'\'Ecole Normale Sup\'erieure}, 3e
  s{\'e}rie, 71(3):213--242, 1954.

\bibitem[FMT89]{Fran/Man/Tsch89}
Jens Franke, Yuri~I. Manin, and Yuri Tschinkel.
\newblock Rational points of bounded height on {F}ano varieties.
\newblock {\em Invent. Math.}, 95(2):421--435, 1989.

\bibitem[Har77]{Har77}
Robin Hartshorne.
\newblock {\em Algebraic geometry}, volume No. 52 of {\em Graduate Texts in
  Mathematics}.
\newblock Springer-Verlag, New York-Heidelberg, 1977.

\bibitem[HMM24]{HMM24a}
Sebasti\'an Herrero, Tob\'ias Mart\'inez, and Pedro Montero.
\newblock Counting rational points on {H}irzebruch--{K}leinschmidt varieties
  over number fields.
\newblock {\em \textup{arXiv:2407.19408}}, 2024.

\bibitem[Kle88]{Kleinschmidt88}
Peter Kleinschmidt.
\newblock A classification of toric varieties with few generators.
\newblock {\em Aequationes Math.}, 35(2-3):254--266, 1988.

\bibitem[Laz04]{LazPAGI}
Robert Lazarsfeld.
\newblock {\em Positivity in algebraic geometry. {I}}, volume~48 of {\em
  Ergebnisse der Mathematik und ihrer Grenzgebiete. 3. Folge. A Series of
  Modern Surveys in Mathematics [Results in Mathematics and Related Areas. 3rd
  Series. A Series of Modern Surveys in Mathematics]}.
\newblock Springer-Verlag, Berlin, 2004.
\newblock Classical setting: line bundles and linear series.

\bibitem[LR19]{LR19}
C\'ecile Le~Rudulier.
\newblock Points alg\'ebriques de hauteur born\'ee sur une surface.
\newblock {\em Bull. Soc. Math. France}, 147(4):705--748, 2019.

\bibitem[LT17]{LT17}
Brian Lehmann and Sho Tanimoto.
\newblock On the geometry of thin exceptional sets in {M}anin's conjecture.
\newblock {\em Duke Math. J.}, 166(15):2815--2869, 2017.

\bibitem[LT19]{LT19}
Brian Lehmann and Sho Tanimoto.
\newblock Geometric {M}anin's conjecture and rational curves.
\newblock {\em Compos. Math.}, 155(5):833--862, 2019.

\bibitem[Pey03]{Peyre2003PHBTA}
Emmanuel Peyre.
\newblock Points de hauteur born\'{e}e, topologie ad\'{e}lique et mesures de
  {T}amagawa.
\newblock volume~15, pages 319--349. 2003.
\newblock Les XXII\`emes Journ\'{e}es Arithmetiques (Lille, 2001).

\bibitem[Pey12]{Pey12PHB}
Emmanuel Peyre.
\newblock Points de hauteur born\'ee sur les vari\'et\'es de drapeaux en
  caract\'eristique finie.
\newblock {\em Acta Arith.}, 152(2):185--216, 2012.

\bibitem[Ros02]{Ros02}
Michael Rosen.
\newblock {\em Number theory in function fields}, volume 210 of {\em Graduate
  Texts in Mathematics}.
\newblock Springer-Verlag, New York, 2002.

\bibitem[Ser89]{Serre1989}
Jean-Pierre Serre.
\newblock {\em Lectures on the {M}ordell-{W}eil theorem}, volume E15 of {\em
  Aspects of Mathematics}.
\newblock Friedr. Vieweg \& Sohn, Braunschweig, 1989.

\bibitem[Tan21]{Tan21}
Sho Tanimoto.
\newblock An introduction to {G}eometric {M}anin's conjecture.
\newblock {\em \textup{arXiv:2110.06660}}, 2021.

\bibitem[Tsc03]{Tschinkel03}
Yuri Tschinkel.
\newblock Fujita's program and rational points.
\newblock In {\em Higher dimensional varieties and rational points ({B}udapest,
  2001)}, volume~12 of {\em Bolyai Soc. Math. Stud.}, pages 283--310. Springer,
  Berlin, 2003.

\bibitem[Tsc09]{Tsch008}
Yuri Tschinkel.
\newblock Algebraic varieties with many rational points.
\newblock In {\em Arithmetic geometry}, volume~8 of {\em Clay Math. Proc.},
  pages 243--334. Amer. Math. Soc., Providence, RI, 2009.

\bibitem[Wan92]{Wan92}
Da~Qing Wan.
\newblock Heights and zeta functions in function fields.
\newblock In {\em The arithmetic of function fields ({C}olumbus, {OH}, 1991)},
  volume~2 of {\em Ohio State Univ. Math. Res. Inst. Publ.}, pages 455--463. de
  Gruyter, Berlin, 1992.

\end{thebibliography}
\end{document}